\documentclass[10pt,reqno]{amsart} %
\usepackage{graphicx}
\usepackage{latexsym, amsmath,amssymb}
\usepackage{graphics}
\usepackage[T1]{fontenc}
\usepackage[latin1]{inputenc}
\usepackage{graphics}
\usepackage{graphicx}
\usepackage{sidecap}

\usepackage{epsfig}
\usepackage{latexsym}
\usepackage{eurosym}
\usepackage{amscd}
\usepackage{shadow}
\usepackage{a4wide}
\usepackage{caption}

\usepackage{subfigure}
\usepackage{wrapfig}

\usepackage{amsmath,amsfonts,amsthm,amssymb}

\usepackage{mathtools}
\mathtoolsset{showonlyrefs}

\usepackage{hyperref}
\hypersetup{colorlinks=true}

\usepackage[margin=2cm,bottom=3cm,top=2cm,footskip=1.5cm]{geometry}

\newcommand\N{{\mathbb N}} \newcommand\R{{\mathbb R}}
\newcommand\Z{{\mathbb Z}} \newtheorem{thm}{Theorem}
\newtheorem{conjecture}{Conjecture}
\newtheorem{rmq}{Remark} 
\newtheorem{lemma}{Lemma} %
\newtheorem{cor}{Corollary} \newtheorem{prop}{Proposition}

\begin{document}

 \title[Quantum bouncing ball]{Strichartz and dispersive estimates for quantum bouncing ball model : exponential sums and Van der Corput methods in 1D semi-classical Schr\"odinger equations}



\author{Oana Ivanovici}
\address{Sorbonne Université, CNRS, Laboratoire Jacques-Louis Lions, LJLL, F-75005 Paris, France} \email{oana.ivanovici@math.cnrs.fr}
    \thanks{The author was
     supported by ERC grant ANADEL 757 996.
    } 
\begin{abstract}
We analyze the one-dimensional semi-classical Schrödinger equation on the half-line with a linear potential and Dirichlet boundary conditions. Our main focus is on establishing improved dispersive and Strichartz estimates for this model, which govern the space-time behavior of solutions. We prove refined Strichartz bounds using Van der Corput-type derivative tests, beating previous known results where Strichartz estimates incur $1/4$ losses. Moreover, assuming sharp bounds for certain exponential sums, our results indicate the possibility to reduce these losses further to $1/6+\epsilon$ for all $\epsilon>0$, which would be sharp. We further expect that analogous Strichartz bounds should hold within the Friedlander model domain in higher dimensions.
\end{abstract}

\maketitle

\section{Introduction}
This paper focusses on the one-dimensional semi-classical Schr\"odinger equation on the half-line with a linear potential and Dirichlet boundary condition
\begin{equation}\label{PP}
ih\partial_t v_h-h^2 \partial^2_x v_h+xv_h=0, \quad \text{ in } x>0, \quad v_{h,t=0}=v_{0},\quad v_{h,x=0}=0,
\end{equation}
where $h\in (0,1)$ is a small parameter and where the initial data is $v_0=\delta_a$ representing a Dirac mass at height $a\in (0,1]$.
This model describes a quantum particle bouncing on a perfectly reflecting surface under gravity, capturing essential features of the quantum bouncing ball.

Our main contributions concern refined dispersive and Strichartz estimates for this $1D$ problem. In Theorem \ref{thm1D}, we prove that dispersive estimates with a loss of $1/4$ previously known in higher dimensions, also hold in the one-dimensional case, with sharp realization at certain intermittent times. Theorem \ref{thm>} improve these bounds - whenever they aren't reached - using Van der Corput derivative tests. Building on these dispersive improvements, Theorem \ref{thmStritoutT} establishes improved Strichartz bounds, reducing losses strictly below $1/4$.

This paper is motivated by the long-standing open question of sharpening Strichartz estimates inside convex domains in dimensions $d\geq 2$. In fact, these one-dimensional results are not only interesting in their own right but also serve as a foundation for understanding the semi-classical Schrödinger flow in higher dimensions within strictly convex domains, e.g. the Friedlander model domain, where the tangential directions complicate the analyse, but where only the normal variable is responsible for losses in dispersion. Our work complements and extends existing dispersive estimates in higher dimensions $d\geq 2$ (see \cite{Iva23}), revealing the fundamental role of the behavior in the normal direction and providing precise insight into the semi-classical dynamics in convex domains.

The intrinsic spectral decomposition of solutions involves Airy functions and exponential sums with highly oscillatory phases and delicate behavior. Our approach carefully separates wave components with distinct behavior and applies oscillatory integral analysis alongside exponential sum bounds. These tools allow us to precisely characterize dispersive decay and to identify the mechanisms behind losses appearing in Strichartz estimates.\\

Before stating the main results, we briefly discuss dispersive estimates on manifolds and recall the key results from \cite{Iva23} in dimensions $d\geq 2$, which inspired the detailed study of the 1D problem as a natural and foundational step to better understand the dispersion phenomena occurring in higher dimensions.

Classical dispersive estimates on $\R^d$ for the linear Schr\"odinger operator with Laplacian $\Delta_{\R^d}$ are well understood:
\begin{equation}\label{dispschrodrd}
\|e^{\pm it\Delta_{\mathbb{R}^d}}\|_{L^1(\mathbb{R}^d)\rightarrow L^{\infty}(\mathbb{R}^d)}\leq C(d)t^{-d/2}, \text{ for all } t\neq 0.
\end{equation}
On manifolds without boundary $(\Omega,g)$ with Laplacian $\Delta_g$, local parametrix constructions (and finite propagation speed at semi-classical scales) show similar decay 
\begin{equation}\label{sclSchrod}
\Big\| \psi(hD_{t})e^{\pm i th \Delta_g}\Big\|_{L^1(\Omega)\rightarrow L^{\infty}(\Omega)}
\leq \frac{C(d)}{h^{d}} \min \Big(1, ({h\over t})^{\frac{d}{2}}\Big) \text{ for all } 0<|t|\leq t_{0}\,,
\end{equation}
where $\psi \in C_{0}^{\infty}$ is a frequency cutoff, $D_{t}=-i\partial_{t}$, and $t_{0}$ depends on the injectivity radius.
 
Analysis on curved manifolds began with Bourgain's work on the torus and was extended to various low-regularity contexts by Staffilani-Tataru \cite{stta02}, Burq-Gérard-Tzvetkov \cite{bgt04}, Smith \cite{sm98}, Tataru \cite{tat02}, among others. In \cite{bgt04} linear estimates and Yudovitch's now classical argument yielded  global well-posedness for the defocusing cubic NLS on compact $3D$ manifolds without boundary. However, for compact manifolds or domains with boundaries, including convex domains, wave reflections and finite volume yield unavoidable losses in dispersion, whose sharp quantification is a challenging open problem. On compact manifolds, dispersive decay eventually deteriorates due to the finite volume: wave packets cannot disperse indefinitely. Infinite propagation speed for the Schrödinger flow causes unavoidable loss of derivatives. This phenomenon, linked to eigenfunctions, remains poorly understood even for the torus. Boundaries introduce further complications by wave reflection.

In \cite{Iva23}, the results of \cite{bgt04} have been extended to the convex-boundary domains for $d=3$ using dispersion and Strichartz bounds with $1/4$ loss. There, a higher-dimensional analog of the equation \eqref{PP} was studied in the semi-classical regime: 
\begin{equation}\label{scl} 
ih\partial_t v_h - h^2\Delta_F v_h = 0, \quad v_h|_{t=0} = v_{0}, \quad v_h|_{\partial \Omega_d} = 0, 
\end{equation} 
posed on the Friedlander model domain $\Omega_{d}=\{(x,y)\in \R_+\times \R^{d-1}\}$, $d\geq 2$, with metric induced by the operator: 
\begin{equation}\label{eq:LapM} 
\Delta_F = \partial_x^2 + \sum_j \partial_{y_j}^2 + x \sum_{j,k} q_{j,k} \partial_{y_j}\partial_{y_k}, 
\end{equation} 
where $q(\theta)=\sum_{j,k}q_{j,k}\theta_j \theta_k$ is a positive-definite quadratic form. Unlike in the rotationally invariant case $q(\theta)=|\theta|^{2}$, this setting lacks symmetry in $y$, preventing reduction to radial analysis. The model approximates geodesic normal coordinates near a strictly convex boundary. The key result from \cite{Iva23} is the following dispersive estimate:
\begin{thm}\label{thmdispSchrodinger}\cite[Thm.1]{Iva23}
Let $\psi\in C^{\infty}_0([\frac 12, \frac 32])$, $0\leq \psi \leq 1$. There exists $C>0$, $t_0\in (0,1)$ and $a_0\leq 1$ such that, for all $a\in(0, a_0]$, $h\in (0,1)$, $|t|\in (h,t_0]$, the solution $v_{h}(t,\cdot)$ to \eqref{scl} with data $v_{0}(x)=\delta_{x=a,y=0}$ satisfies $\forall x\in \Omega_d$
  \begin{equation}
    \label{eq:1}
 \Big|\psi(hD_{t})v_{h}(t,x,y)\Big|
    \leq \frac{C(d)}{h^{d}}\Big(\frac{h}{|t|}\Big)^{\frac{(d-1)}{2}}\Big(\Big(\frac{h}{|t|}\Big)^{\frac 1 4}+h^{1/3}\Big)\,.
  \end{equation}
Moreover, for all $h^{2/3}<a$, for all $|t|\in (\sqrt{a},\min(T_0, a h^{-1/3})]$, the bound is sharp (at $x=a$):
\begin{equation}\label{optloss}
\Big|\psi(hD_{t})v_{h}(t,a,y)\Big|
 \sim  \frac{a^{\frac 1 4}}{h^{d}}\Big(\frac{h}{|t|}\Big)^{\frac{(d-1)}{2}+\frac 1 4}\,.
\end{equation}
\end{thm}
Here $t$ denotes the semi-classical time, $h\in (0,1)$ is the semi-classical parameter (with frequency scale $|D_t|\sim 1/h$) and $a$ measures the distance to the boundary of the initial data, taken small. The main interest is in behaviour after multiple reflections for times $t\lesssim 1$. For large $t$ or $a$, existing parametrix construction from \cite{Iva21} suffices but are out of scope of this work.
A direct consequence of Theorem \ref{thmdispSchrodinger} are the following Strichartz estimates 

\begin{cor}\label{corStrichartz}\cite[Thm.2]{Iva23}
Let $d\geq 2$, $(q,r)$ such that $\frac 1q\leq \Big(\frac d 2 -\frac 14\Big)\Big(\frac 1 2 -\frac 1 r\Big)$ and $\beta=d(\frac 12-\frac 1r)-\frac 1q$.
There exist $C(d)>0$, $t_{0}>0$ such that, for $v$ solution to \eqref{scl} with data $v_{0}\in L^{2}(\Omega_{d})$,
  \begin{equation}
    \label{eq:2}
  h^{\beta}  \|\psi(hD_{t})v_{h}\|_{L^q([-t_{0},t_0],L^r(\Omega_d))}\leq C(d) \|v_{0}\|_{L^2(\Omega_d)}\,.
  \end{equation}
\end{cor}
Corollary \ref{corStrichartz} follows from the $TT^*$ method. For $d=3$, the endpoint $(2,10)$ of \cite{Iva23} enables an adaptation of \cite{bgt04}'s argument to obtain well-posedness for the cubic nonlinear Schrödinger equation. \\

While Theorem \ref{thmdispSchrodinger} shows that a loss of $1/4$ is sharp for the dispersive bounds, the $TT^*$ argument doesn't yield sharp Strichartz estimates (not even near moments where \eqref{optloss} is reached). The $1/4$ loss arises only at specific times, and when the tangential variable is located in some narrow regions, suggesting that integration over time and space could improve the bounds. In model convex domains (e.g. the ball or  Friedlander domain), where the Laplace operator coefficients \eqref{eq:LapM} do not depend on the tangential variable $y$, Fourier transform in $y$ reduces the problem to a $1D$ equation resembling \eqref{PP}, differing only in the coefficient of $x$ by a factor $|\eta|^2$, where $\eta$ is the dual variable of $y$. Understanding the solution to this $1D$ problem and its dispersive properties thus provides valuable insight into the original semi-classical Schrödinger flow $v_h$ satisfying \eqref{scl}. Though the proof is not provided here for dimension $d\geq 2$, we claim that under the same hypothesis, the solution $v_h$ to \eqref{scl} also satisfies similar Strichartz estimates, the main technical challenge when $d\geq 2$ being the localization of the tangential variable $y$. \\

Here are our main results: firstly, we have the $1D$ version of Theorem \ref{thmdispSchrodinger} with $\Omega_{d=1}=\R_+$ and  operator $h^2\partial^2_x-x$.
\begin{thm}\label{thm1D}
Let $d=1$. The estimates \eqref{eq:1} and \eqref{optloss} also hold. 
Depending on $T=t/\sqrt{a}$, $\lambda=a^{3/2}/h$, we have :
\begin{itemize}
\item If $1\leq T < \lambda^{1/3}$, that is for $a>th^{1/3}$, two situations may arise :
\[
  \Big|\psi(hD_t)v_h(t,x)\Big| \sim \frac{1}{h}\Big({\frac{ha}{t}}\Big)^{1/4}=\frac 1h h^{1/3}\Big(\frac{\lambda^{1/3}}{T}\Big)^{1/4}, \text{ if } 1\leq T < \lambda^{1/3} \text{ is such that $dist(T,2\N)< \frac{1}{4T^2}$},
\]
while 
\begin{equation}\label{T<dist>}
  \Big|\psi(hD_t)v_h(t,x)\Big| \lesssim \frac1h h^{1/3}, \text{ if } T< \lambda^{1/3} \text{ is such that $dist(T,2\N)\geq \frac{1}{4T^2}$}.
\end{equation}
\item If $T\geq \lambda^{1/3}$, which corresponds to $a\leq t h^{1/3}$, depending on whether $T<\lambda$ or $T>\lambda$, we have
 \begin{equation}
    \label{eq:3}
    \Big|\psi(hD_t)v_h(t,x)\Big|\lesssim 
   \left\{ \begin{array}{l} 
 \frac{1}{h} \Big(\frac{ht}{a}\Big)^{1/2}=\frac 1h h^{1/3} \Big(\frac{T}{\lambda^{1/3}}\Big)^{1/2}, \text{ if } \lambda^{1/3}\leq T\leq \lambda, \text{ i.e. for } (ht)^{1/2}\leq a\leq th^{1/3}, \\\\
 \frac1h (ht)^{1/4}= \frac 1h h^{1/3} (T\lambda^{1/3})^{1/4} \text{ if } T>\lambda, \text{ i.e. for } a\leq (ht)^{1/2}.
 \end{array} \right.
  \end{equation}
  \end{itemize}
   \end{thm}
\begin{rmq}
The dispersive bounds for $1\leq T<\lambda^{1/3}$ (i.e. for $\sqrt{a}\leq t<a/h^{1/3}$) are sharp, but occur only at intermittent moments in time such that $T=t/\sqrt{a}\in 2\mathbb{N}$. For this regime, we have $(ha/t)^{1/4}>h^{1/3}$ ($\Leftrightarrow  T<\lambda^{1/3}$). The factor $h^{1/3}=(h/t)^{1/2}\times t^{1/2}h^{-1/6}$ yields a loss $1/6$ in the dispersive bounds compared to the free case \eqref{sclSchrod}, while $(ha/t)^{1/4}=(h/t)^{1/2}\times (at/h)^{1/4}$ provides up to $1/4$ loss in the dispersive bounds - which is reached when $t/\sqrt{a}\in 2\mathbb{N}$, and also in the Strichartz bounds via the $TT^*$ argument. Corollary \ref{corStrichartz} holds for $d=1$ (however it is far from sharp).\\

When $T\geq \lambda^{1/3}$ the bounds \eqref{eq:3} are no longer sharp. These estimates are obtained as follows : for $T<\lambda$, we construct a parametrix expressed as a sum of wave packets indexed by the number of reflections on the boundary (see formula \eqref{eq:bis48bis}). In section \ref{sec>12}, we obtain sharp bounds for each individual wave packet : however, because all wave packets interact at every moment in time, the sum of their absolute values yields the contribution in the first line of \eqref{eq:3}. The same approach applies for all $T$, but when $T>\lambda$ the resulting estimate become worst than the second line of \eqref{eq:3}, which is derived directly from the spectral decomposition of the solution combined with Sobolev bounds.
\end{rmq}

\begin{rmq}
The variable $T=t/\sqrt{a}$ is introduced as a natural normalization of the time variable $t$. Starting from a small initial distance $a<1$, a wave packet reaches the boundary in a time comparable to $\sqrt{a}$, therefore $T>1$ corresponds to at least one reflection. Since the time elapsed between two consecutive reflections is $\sim 2\sqrt{a}$, $T$ effectively counts the number of reflections on the boundary. The parameter $\lambda=a^{3/2}/h$ arises naturally in various contexts (and is large as $\lambda \lesssim 1$ means $a\lesssim h^{2/3}$, when both dispersion and Stricharz hold with $1/6$ loss). 
It represents the number of waves significantly contributing to the spectral sum defining the solution to \eqref{PP} (see section \ref{secspectralprop}). More precisely, in the Green function formula \eqref{greenfct} for \eqref{PP}, the terms with indices $k\sim\lambda$ yield dominant contributions affecting dispersive bounds. For smaller $k<\lambda/4$, the Airy factor in the eigenfunctions $e_k$ decays exponentially, while for larger $k>4\lambda$ the waves are "transverse" and their contribution to the solution is significantly better than those near $k\sim \lambda$. This is why the analysis deals with "tangential" waves separately in section \ref{secttangcas}, corresponding to $k\sim \lambda$, from "transverse" waves discussed in section \ref{sectransvwa} and corresponding to $k\sim \gamma^{3/2}/h$ with $\gamma>4a$.
\end{rmq}

Depending on $T$ and $\lambda$, we improve upon \eqref{eq:3} and Corollary \ref{corStrichartz}. The proof of Theorem \ref{thm>} is provided in Section \ref{secT>}.

\begin{thm}\label{thm>}
Let $T\geq \lambda^{1/3}$. Then the Van der Corput's $j$-th derivative test estimates ((VdCj), see  section \ref{secVCD}) allow to improve upon \eqref{eq:3} as follows
 \begin{equation}
    \label{eq:499}
    \Big|\psi(hD_t)v_h(t,x)\Big| \lesssim \frac1h h^{1/3}\times
     \left\{ \begin{array}{l} 
 \Big(\frac{T}{\lambda^{1/3}}\Big)^{1/2}, \text{ if } \lambda^{1/3}\leq T\leq \lambda^{1/2},\\\\
T^{1/6}, \text{ if } \lambda^{1/2}\leq T<\lambda^{5/4},\\\\
\lambda^{5/42}T^{1/14}, \text{ if } \lambda^{5/4}\leq T< \lambda^{29/12},\\\\
\lambda^{1/3}, \text{ if } T> \lambda^{29/12}.\\
 \end{array} \right.
  \end{equation}
These estimates induce dispersive bounds with $\frac16+\frac{5}{114}=\frac14-\frac{3}{76}$ loss for $T\geq \lambda^{1/3}$ as follows
\begin{equation}\label{dispT>}
  \Big|\psi(hD_t)v_h(t,x)\Big| \lesssim \frac 1h h^{1/3}\times h^{-5/114}\leq \frac 1h (\frac ht)^{1/2}\times h^{-(1/6+5/114)}, \quad \text{ if } t\geq a/h^{1/3}.
\end{equation}
\end{thm}

\begin{rmq}
The bounds \eqref{dispT>} follow from Van der Corput's $3$th derivative test (VdC3, see Prop. \ref{propbestk=3}) in the regime $T\in [\lambda^{1/2}, \lambda^{5/4}]$ (see the second line of \eqref{eq:499}); all the others regimes ($T\leq \lambda^{1/2}$ or $T>\lambda^{5/4}$) provide better contributions. Any improvement on (VdC3) would allow better dispersive bounds in \eqref{dispT>}.
\end{rmq}

Theorems \ref{thm1D} and \ref{thm>} yield the following result involving Strichartz bounds with $\frac16+\frac{5}{114}=\frac14-\frac{3}{76}$ loss, which improves upon the $\frac 14$ loss in Corollary \ref{corStrichartz} for $d=1$. The proof of Theorem \ref{thmStritoutT} is provided in Section \ref{secT<}.
\begin{thm}\label{thmStritoutT}
Let $d=1$, $(q,r)$ such that $\frac 1q\leq (\frac{1}{2}-(\frac 14-\frac{3}{76}))(\frac 12-\frac 1r)$ and $\beta=(\frac 12-\frac 1r)-\frac 1q$.There exists $C>0$, $t_0>0$ such that the solution $v_h$ to \eqref{PP} satisfies the following Strichartz bounds 
  \begin{equation}\label{Stri1D}
 h^{\beta}  \|\psi(hD_{t})v_{h}\|_{L^q([-t_{0},t_0],L^r(\R_+))}\leq C  \|v_{0}\|_{L^2(\R_+)}\,, \quad \forall v_0\in L^2(\R_+).
  \end{equation}
  \end{thm}
\begin{rmq}
For $t\geq a/h^{1/3}$, the bounds \ref{Stri1D} follow using \eqref{dispT>} and the $TT^*$ argument. For $t< a/h^{1/3}$, which corresponds to values $T<\lambda^{1/3}$, Theorem \ref{thm1D} cannot be use directly, as done in Corollary \ref{corStrichartz} for $d=1$, as the (sharp) loss of $1/4$ in dispersion necessarily induces a $1/4$ loss in Strichartz via the $TT^*$ argument. As the loss occurs at intermittent moments, it is clear that integration in time should allow to improve the corresponding Srichartz in this regime. Indeed, we show that for $t\geq a/h^{1/3}$ the Strichartz bounds hold with $1/6$ loss (which is the best result we can expect): for dist$(T,2\N)\geq 1/4 T^2$ we use the $TT^*$ argument and the fact that the dispersive bounds hold with $1/6$ loss (due to the factor $h^{1/3}$ in \eqref{T<dist>}), while for dist$(T,2\N)< 1/4 T^2$ use carefully use the sharp bounds \eqref{eq:2hh} in Proposition \ref{dispNpetitpres} and integrate in time over small neighborhoods of the critical moments when the $1/4$ loss arises. However, in this regime one can still follow the approach of the proof of Theorem \ref{thmStritoutT} and gain by carefully integrating over small neighborhoods of the moments of time when \eqref{optloss} occurs.

\end{rmq}

\begin{rmq}
We claim that similar results (as in Theorems \ref{thm>} and \ref{thmStritoutT}) hold for the Friedlander model domain in dimension $d\geq 2$. However, extending the proof of Theorem \ref{thm>} when $T\geq \lambda^{1/3}$ to higher dimensions presents significant technical challenges due to the presence of the tangential variable $y\in \mathbb{R}^{d-1}$. Also, when $T<\lambda^{1/3}$, it was shown in \cite{Iva23} that, for each fixed time $t$, there exists a small subset of space - specifically, points of the form $(x,y)$ with $x=a$ and $|y|\sim t$ whose size depend on $t, a, h$ - where the $1/4$ loss is realised. Crucially, this loss is not confined to isolated moments in time, but occurs persistently, making it essential to carefully analyze the contribution of the $y$-variable in $L^q$ norms when attempting to reduce the loss and improve the estimate.
\end{rmq}

\begin{rmq}
The factor $\frac 1h h^{1/3}=\frac 1h (h/t)^{1/2}\times t^{1/2}h^{-1/6}$ in \eqref{eq:499} corresponds to a $1/6$ loss in dispersion, while the right hand side factors depending on $T, \lambda$ and on the regimes correspond to the loss due to the Van der Corput's tests, which are not sharp (unless in very specific situations) but they still allow to improve upon the bounds in Theorem \ref{thm1D} when $T>\lambda^{1/3}$. The "worst" bounds in \eqref{eq:499} come from the regime $T\in [\lambda^{1/2}, \lambda^{5/4}]$, and the corresponding loss in the Strichartz estimates becomes $\frac16+\frac{5}{114}=\frac14-\frac{3}{76}$, hence it is strictly less than $1/4$. In \eqref{eq:499} we have

 \begin{itemize}
 \item $\Big(\frac{T}{\lambda^{1/3}}\Big)^{1/2}$ is obtained with (VdC2), $\delta_2\sim T/\lambda^2$, after Abel summation;  $T^{1/6}$ is obtained with (VdC3), $\delta_3\sim T/\lambda^2$, after Abel summation;  $\lambda^{5/42} T^{1/14}  $ is obtained with (VdC4), $\delta_4\sim T/\lambda^3$.
 \end{itemize}
 
The loss in Strichartz using the bounds \eqref{eq:499} is computed as follows, depending on each regime

\begin{itemize}
 \item Loss in Strichartz at $T\sim \lambda$ ($\Leftrightarrow$ $t/\sqrt{a}\sim a^{3/2}/h$)  $\Leftrightarrow$ $a\sim h^{1/2}$ when $t\sim 1$, hence
 \[
 T^{1/6}\sim \lambda^{1/6}\sim (t/\sqrt{a})^{1/6}\sim a^{-1/12}\sim h^{-1/24}\Rightarrow \frac{1}{6}+\frac{1}{24}.
 \]

 \item Loss in Strichartz at $T\sim \lambda^{5/4}$ $\Leftrightarrow$ $t/\sqrt{a}\sim (a^{3/2}/h)^{5/4}$ $\Leftrightarrow a\sim h^{10/19}$ $\Leftrightarrow \lambda\sim h^{-4/19}$ when $t\sim 1$, hence
 \[
 T^{1/6}\sim (\lambda^{5/4})^{1/6}\leq h^{-(5/24)\times (4/19)}\sim h^{-5/114}\Rightarrow \frac{1}{6}+\frac{20}{19}\times \frac{1}{24}=\frac16+\frac{5}{114} \quad\text{ as in \eqref{Stri1D}}.
 \]
 \item For larger $T$ the loss becomes smaller. Better (VdC3,4) $\Rightarrow$ $T^{1/6}$ for $T\ll \lambda^{5/4}$ $\Rightarrow$ better bounds in \eqref{eq:499}
 \item For now : $\frac{1}{6}+\frac{20}{19}\times \frac{1}{24}=\frac 14-\frac{3}{76}$. Expected : $\frac16+\epsilon$ $\forall \epsilon>0$, see the conjecture below.

\end{itemize}
\end{rmq}

In Theorems \ref{thm>} and \ref{thmStritoutT}, we establish improved Strichartz estimates for the one-dimensional semi-classical Schrödinger equation with linear potential on the half-line. Our method is based on Van der Corput-type derivative tests, yielding bounds that are as sharp as currently possible. 
Finally, it was shown in \cite{doi} that a minimal loss of $1/6$ derivatives in Strichartz estimates is unavoidable, as demonstrated by gallery mode initial data, and whether this is sharp remains an open problem. 
We assert that, if optimal exponential sum bounds are available (see section \ref{secExpSums}), then dispersive bounds with $1/6+\epsilon$ should hold for all $T\geq \lambda^{1/3}$ ; this would further imply optimal Strichartz bounds with a loss of $1/6+\epsilon$ for all $\epsilon>0$ (in $1D$ and similarly in higher-dimensional strictly convex domains). 

\begin{conjecture}\label{thmLindel}
Let $T\geq \lambda^{1/3}$ and assume that sharp exponential sums bounds hold $\forall \epsilon>0$, then the following dispersive bounds should hold true :
 \begin{equation}
    \label{eq:444}
    \Big|\psi(hD_t)v_h(t,x)\Big|\lesssim \frac1h h^{1/3-\epsilon},\quad \forall \epsilon>0\,.
        \end{equation}
As a consequence (of \eqref{eq:444} and of the proof of Theorem \ref{thmStritoutT} when $T<\lambda^{1/3}$), the Strichartz estimates should hold with $1/6+\epsilon$ loss for any $L^2(\R_+)$ data. Moreover, the same results are expected to hold for the solution to the semi-classical Schrödinger equation inside the Friedlander domain in $d\geq 2$ or in a ball. 
  \end{conjecture}
 
Our conjectured improvements of Strichartz and dispersive estimates fundamentally rely on achieving sharp bounds for certain exponential sums (see section \ref{secExpSums}). These sums naturally arise from the spectral decomposition of solutions to the quantum bouncing ball and related models (the Friedlander model or the ball in higher dimensions), where all wave packets interact simultaneously and contribute significantly. It is important to emphasize that  sharp dispersion bounds cannot be obtained without correspondingly sharp cancellation in these oscillatory sums. Optimal bounds for polynomial exponential sums, such as those established by Wooley \cite{Woo16} for the cubic Vinogradov mean value theorem, serve as a model benchmark for cancellation phenomena.  Although classical exponential sum results provide useful intuition, our problem involves more complex phases, which require careful analysis. In this work, we employ Van der Corput derivative tests to navigate the difficulties posed by certain "bad" subsets, achieving the best possible bounds currently accessible with available analytic techniques.\\

Before beginning the proof of the above theorem, we first discuss the connection between exponential sum estimates and the semi-classical Schrödinger flow. Within a bounded domain, the solution to the semi-classical Schrödinger equation with Dirac initial data at $t=0$ can be expressed via  the eigenvalues and the eigenvalues of the Laplace operator. For the model case of the Friedlander domain - the half space with metric inherited from the Laplace operator \eqref{eq:LapM} - the spectrum of $\Delta_F$ is well understood : the eigenfunctions are given in terms of Airy functions, while the corresponding eigenvalues correspond to the zeros of the Airy function, reflecting the Dirichet condition. As the coefficients of $\Delta_F$ are independent of the tangential variable $y\in \R^{d-1}$, taking the Fourier transform in $y$ reduces the problem to the $1D$ operator  on the half line given in \eqref{PP}. The spectrum of $-h^2\partial^2_x+x$ is explicitly described in Lemma \ref{lemorthog}, where $(-\omega_k)_{k\geq 1}$ denote the zeros of the Airy function. Consequently, the Green function of \eqref{PP} can be written as a spectral sum (see \eqref{greenfct}, where $\lambda_k=h^{-4/3}\omega_k$) and where the main contribution arise from indices $k\sim \lambda=a^{3/2}/h$. Normalising variables with $T:=t/\sqrt{a}$, $X:=x/a$ and $\lambda=a^{3/2}/h$ and using that $\omega_k=Ck^{2/3}(1+O(1/k))$ for some constant $C>0$ (see section \ref{secT>}), yields 
\[
ht\lambda_k=ht(h^{-4/3}\omega_k)=(t/\sqrt{a})\times (\sqrt{a}/h^{1/3})\times \omega_k=C\lambda T (k/\lambda)^{2/3}(1+O(1/k)).
\]
Here, $T\geq 1$ since smaller values correspond to waves not reaching the boundary, and $\lambda>1$ to avoid exponentially small the eigenfunctions due to Airy function decay. The spectral sum of interest for dispersive bounds comes from values $k\sim \lambda$ and equals
\[
\sum_{k\sim \lambda} e^{i\lambda T \frac{\omega_k}{\lambda^{2/3}}} \frac{1}{\sqrt{\omega_k}}Ai(X \lambda^{2/3}-\omega_k)Ai(\lambda^{2/3}-\omega_k),
\]
where the factor $1/\sqrt{\omega_k}$ normalises the eigenfunctions $e_k$ in $L^2(\R_+)$ (see Lemma \ref{lemL}). Each Airy factor can be decomposed into $\sum_{\pm} A_{\pm}$ defined in \eqref{eq:Apm}, \eqref{eq:ApmAE}, and the analysis can be reduced to the case $X=1$ which captures the worst regime where swallowtail singularities appear in the wavefront when the spectral sum is transformed, via a Poisson summation, in a sum reflected waves. Hence, we focus on estimating
\[
\frac{1}{h^{2/3}}\Big| \sum_{k\sim \lambda}  e^{i\lambda T \frac{\omega_k}{\lambda^{2/3}}}
 \frac{1}{\sqrt{\omega_k}}  Ai^2(\lambda^{2/3}-\omega_{k})\Big| ,\quad \text{ where } Ai^2(-z)\sim z^{-1/2} (1+\sum_{\pm} e^{\pm \frac 43 z^{3/2}}), \quad \omega_k\sim k^{2/3}=(\lambda+j)^{2/3},
\]
when $j\geq \lambda^{1/3}$. The contribution for $|j|\leq \lambda^{1/3}$ - i.e. for $\omega_k$ extremely close to $\lambda^{2/3}$ - is easier to control with and yields a $1/6$ derivative loss in the dispersive and Strichartz bounds. The main challenge is therefore to bound the absolute value of sums of the form 
\[
\frac{1}{h^{2/3}}\Big|\sum_{\kappa=0,\pm1}{\sum_{j=\lambda^{1/3}}^{\lambda}  e^{i\lambda Tf_{\kappa}(j)} \frac{1}{\lambda^{1/6}\sqrt{j}}}\Big|,
\]
where
\[
f_{\kappa}(j)=\Big(\frac{\lambda +j}{\lambda}\Big)^{\frac 23}+\frac 43 \frac{\kappa}{T}\Big(\Big(\frac{\lambda+j}{\lambda}\Big)^{2/3}-1\Big)^{3/2}.
\]
Since $\sqrt{\omega_k}\sim k^{1/3}\sim \lambda^{1/3}$, applying the Abel summation shows that the main contribution arises from exponential sums alone, whose phases $f_{\kappa}(j)$ are studied in Lemmas \ref{lemfkap1}, \ref{lemfkap2} and \ref{lemdelta} in section \ref{secT>}. The phase $f_0$ is as in \eqref{fctalpha} (see Appendix \ref{secapp}), while for $\kappa=\pm1$ the phases satisfy the assumptions required for the Van der Corput's $j$ derivative tests (VdCj) for all $j\geq 2$. Thus, known Van der Corput estimates apply and imply the results in Theorem \ref{thm>}. Any advancement in these higher-order derivative tests, especially the (VdC3) and (VdC4), would permit further improvements in the Strichartz estimates presented  in Theorem \ref{thm>}.

\vskip2mm
We conclude this introduction by outlining the structure of the following sections. Section \ref{sectconstruct} presents the spectrally localized Green function for \eqref{PP}, first as a spectral sum over the spectrum of \eqref{eq:LapM}, and then, via a variant of the Poisson summation formula, as a sum of oscillatory integrals indexed by the number of reflections at the boundary. In Section \ref{sectdisp}, we use both representations to derive dispersive bounds with the $1/4$ loss of Theorem \ref{thm1D}, depending on the initial distance to the boundary $a>0$. Notably, when $a>\max(h^{2/3-\epsilon},(ht)^{1/2})$, the   oscillatory integral sum \eqref{eq:bis48bis} proves particularly effective, enabling stationary phase analysis. Since only waves "launched" from $x=a$ at $t=0$ within a narrow cone of aperture $\sqrt{a}$ can cause the $1/4$ loss, we treat this "tangential" scenario (Section \ref{secttangcas}) separately from the "transverse" case (covered as in \cite{Iva23}). In particular, Section \ref{secttangcas} provides three main results (Propositions \ref{dispNgrand}, \ref{dispNpetitloin}, \ref{dispNpetitpres}), giving  refined estimates for each integral in the Green function, depending on the number of reflections and spatial position, thereby establishing Theorem \ref{thmdispSchrodinger} for "not too small" $a$. Section \ref{sectopt} demonstrates optimality when $a>h^{1/3}$. For small $a$, Section \ref{sectasmall} leverages the spectral formula \eqref{greenfctbis} to provide dispersive bounds with $1/4$ loss using Sobolev estimates to suppress oscillations.

In Section \ref{secT<} we prove Theorem \ref{thmStritoutT}: when $T\geq \lambda^{1/3}$, with $T=t/\sqrt{a}$ and $\lambda=a^{3/2}$, this immediately follows from Theorem \ref{thm>}. When $T<\lambda^{1/3}$, when swallow tails singularities in the wavefront occur intermittently (at $T=2N$) and account for the sharp $1/4$ loss, we need to carefully integrate over small time intervals around $t/\sqrt{a}=2N$ in order to improve this loss. In this regime we eventually show that the Strichatz estimates hold with $1/6+\epsilon$ loss for all $\epsilon>0$, which is the best result we can hope for (as a $1/6$ loss in Strichartz is known to be unavoidable).

Section \ref{secT>} considers $T\geq \lambda^{1/3}$ (or $a\leq h^{1/3}$) where swallowtail singularities persist but do not drive dispersive losses; the main challenge becomes interference among many wave packets. Although sharp estimates are available for individual packets, summing absolute values fails to exploit cancellation. In this regime, Van der Corput derivative tests (with $j=2,3,4$) improve the dispersive bounds allowing to obtain a loss below the $1/4$ in Theorem \ref{thm1D}. 
Section \ref{sectproofsprops} proves sharp wave packet bounds for $a>h^{1/2}$ as used in previous sections. Finally, the Appendix (section \ref{secapp}) recalls the Van der Corput derivative tests used in these arguments, together with a brief paragraph on exponential sums bounds and Conjecture \ref{thmLindel}.

\vskip2mm
In the paper, $A\lesssim B$ means that there exists a constant $C$ such that $A\leq CB$; this constant may change from line to line and is independent of all parameters but the dimension $d$. It will be explicit when (very occasionally) needed. Similarly, $A\sim B$ means both $A\lesssim B$ and $B\lesssim A$.

\section{The semi-classical Schrödinger propagator : parametrix construction}\label{sectconstruct}

\subsection{Some properties of the Airy function }
Let $Ai(x)=\frac{1}{2\pi } \int_{\R} e^{ i  (\frac{\sigma^{3}}{3}+\sigma x)} \,d\sigma$.
Define
\begin{equation}
  \label{eq:Apm}
  A_\pm(z)=e^{\mp i\pi/3} Ai(e^{\mp i\pi/3} z)=-e^{\pm 2i\pi/3} Ai(e^{\pm 2i\pi/3} (-z))\,,\,\,\text{ for } \,
  z\in \mathbb{C}\,,
\end{equation}
then one checks that $Ai(-z)=A_+(z)+A_-(z)$ (see \cite[(2.3)]{AFbook}).
 We have
\begin{equation}\label{eq:ApmAE}
A_{\pm}(z)=\Psi(e^{\mp i\pi/3} z)e^{\mp \frac 23 iz^{3/2}},\quad \Psi(z)\sim_{1/z} z^{-1/4}\sum_{j=0}^{\infty} a_j z^{-3j/2}, \quad a_0=\frac{1}{4\pi^{3/2}}.
\end{equation}

\begin{lemma}\label{lemL} (see \cite[Lemma 1]{Iva23})
Define, for $\omega \in \R$, $  L(\omega)=\pi+i\log \frac{A_-(\omega)}{A_+(\omega)}$,
then $L$ is real analytic and strictly increasing. We also
have
\begin{equation}
  \label{eq:propL}
  L(0)=\pi/3\,,\,\,\lim_{\omega\rightarrow -\infty} L(\omega)=0\,,\,\,
  L(\omega)=\frac 4 3 \omega^{\frac 3 2}+\frac{\pi}{2}-B(\omega^{\frac 3
    2})\,,\,\,\text{ for } \,\omega\geq 1\,,
\end{equation}
with $ B(u)\sim_{1/u} \sum_{k=1}^\infty b_k u^{-k}$, $b_k\in\R$, $b_1> 0$.
Finally,
  $Ai(-\omega_k)=0 \iff L(\omega_k)=2\pi k$ and
 $ L'(\omega_k)=2\pi \int_0^\infty Ai^2(x-\omega_k) \,dx\,\sim \sqrt{2\omega_k}$
where here and thereafter, $\{-\omega_k\}_{k\geq 1}$ denote the zeros of the Airy function in decreasing order.
\end{lemma}

We briefly recall a variant of the Poisson summation formula that will be crucial to analyse the spectral sum defining $G_{h,\gamma}$. We denote \eqref{eq:AiryPoissonBis} the Airy - Poisson formula.
\begin{lemma}\label{lemAP}
  In $\mathcal{D}'(\R_\omega)$, one has
$ \sum_{N\in \Z} e^{-i NL(\omega)}= 2\pi \sum_{k\in \N^*} \frac 1
    {L'(\omega_k)} \delta(\omega-\omega_k)\,$, e.g.  $\forall \phi\in C_{0}^{\infty}$, 
  \begin{equation}
    \label{eq:AiryPoissonBis}
        \sum_{N\in \Z} \int e^{-i NL(\omega)} \phi(\omega)\,d\omega = 2\pi \sum_{k\in \N^*} \frac 1
    {L'(\omega_k)} \phi(\omega_k)\,.
  \end{equation}
\end{lemma}

\subsection{Spectral properties of the operator and parametrix in terms of a spectral sum}\label{secspectralprop}

As $-\partial^2_x+x$ is a positive self-adjoint operator on $L^2(\mathbb{R}_+)$, with compact resolvent, we have: 

\begin{lemma}\label{lemorthog} (see \cite[Lemma 2]{ILP3})
There exist eigenfunctions $\{e_k(x)\}_{k\geq 1}$ of $-h^2 \partial^2_x+x$ with corresponding eigenvalues $\lambda_k=\omega_k h^{-4/3}$, that form an Hilbert basis for $L^{2}(\mathbb{R}_{+})$. These eigenfunctions are explicit in terms of Airy functions:
\begin{equation}\label{eig_k}
 e_k(x)=\frac{\sqrt{2\pi h^{-2/3}}}{\sqrt{L'(\omega_k)}}
 Ai\Big(xh^{-2/3}-\omega_k\Big)\,,
\end{equation}
and $L'(\omega_k)$ (with $L$ from Lemma \ref{lemL}) is such that
$\|e_k(.)\|_{L^2(\mathbb{R}_+)}=1$.
\end{lemma}
For $x_0>0$, $\delta_{x=x_0}$ on $\mathbb{R}_+$ may be decomposed as
 $\delta_{x=x_0}=\sum_{k\geq 1} e_k(x)e_k(x_0)$.
Consider $v_{h,0}(x)=\delta_{x=x_0}$, then the Green function for \eqref{PP} in $\{x>0\}$ reads as follows
\begin{equation}\label{greenfct} 
G_{h}(t,x,x_0)=\sum_{k\geq 1}e^{iht\lambda_k}
e_k(x)e_k(x_0)\,.
\end{equation}
As explained in \cite{Iva23}, the significant part of the sum over $k$ in \eqref{greenfct} becomes a finite sum over $k\ll1/h$ (as larger values of $k$ correspond to transverse wave packets (see \cite[Section 2.1]{Iva23})). Thus, we consider the part of the Green function \eqref{greenfct} where the sum is taken for $k\leq \varepsilon_0/h$ for some small, fixed $\varepsilon_0>0$. As in \cite{Iva23}, the remaining part of the Green function (corresponding to values $k\gtrsim1/h$) will essentially be transverse: at most one reflection for $t\in [0,T_0]$ with $T_0$ small (depending on the above choice of $\varepsilon_0$). Hence, this regime can be dealt with as in \cite{bss08} to get the free space decay and we will ignore it in the upcoming analysis.

Reducing the sum to $k\leq \varepsilon_0/h$ is equivalent to adding a spectral cut-off $\phi_{\varepsilon_0}(x+h^2D_{x}^{2})$ in the Green function (with $D_x=\frac{1}{i}\partial_x$), where $\phi_{\varepsilon_0}=\phi(\cdot/\varepsilon_0)$ for some smooth cut-off function $\phi\in C^{\infty}_0([-1,1])$. Notice that
$(x+h^2D_{x}^{2})e_{k}(x)=\omega_k h^{2/3} e_{k}(x)$ and this new localization operator is exactly associated by symbolic calculus to the cut-off $\phi_{\varepsilon_0}(\omega_{k}h^{2/3})$. We therefore set, for $(t_0,x_0)=(0,a)$,
\begin{equation}\label{greenfctbiseps0} 
G^{\varepsilon_0}_h(t,x,a) : =  \sum_{k\geq 1}
e^{iht\lambda_k}
 \phi_{\varepsilon_0}(\omega_{k}h^{2/3})
 e_k(x)e_k(a)\,.
\end{equation}
Set $a^{\natural}=\max{(a,h^{2/3})}$: in the following we introduce a new, small parameter $\gamma$ satisfying ${a^{\natural}}\lesssim \gamma\leq \varepsilon_0$ and then split the (tangential part of the) Green function $G^{\varepsilon_0}_h$ into a dyadic sum $G_{h,\gamma}$ corresponding to a dyadic partition of unity supported for $\omega_k h^{2/3}\sim \gamma \sim 2^{2j}{a^{\natural}}\leq \varepsilon_0$.
Let $\psi_2(\cdot/\gamma):=\phi_{\gamma}(\cdot)-\phi_{\gamma/2}(\cdot)$, set $\Gamma_{l}({a^{\natural}})=\{ \gamma=2^j {a^{\natural}}, l\leq j<\frac 12 \log_2(\varepsilon_0/{a^{\natural}})\}$ (we will use $l=0,1,3$) and decompose $\phi_{\varepsilon_0}$ as follows
\begin{equation}\label{partunitpsi2}
\phi_{\varepsilon_0}(\cdot) %
=\phi_{{a^{\natural}}}(\cdot)+\sum_{\gamma\in \Gamma_{1}({a^{\natural}})}\psi_2(\cdot/\gamma),
\end{equation}
which allows to write $G_{h}^{\varepsilon_0}=\sum_{{a^{\natural}}\leq \gamma<1}G_{h,\gamma}$ where the sum is understood as over dyadic $\gamma$'s, and $G_{h,\gamma}$ reads as
\begin{equation}
\label{greenfctbis} 
G_{h,\gamma}(t,x,a)  =  \sum_{k\geq 1}
e^{iht\lambda_k}
 \psi_{2}(h^{2/3}\omega_{k}/\gamma)
e_k(x)e_k(a).
\end{equation}

Notice that, when $\gamma={a^{\natural}}$, according to \eqref{partunitpsi2}, we should, in \eqref{greenfctbis}, write $\phi_{{a^{\natural}}}$ instead of $\psi_2(\cdot/{a^{\natural}})$. However, for values $h^{2/3}\omega_k\lesssim \frac 12 {a^{\natural}}$, the corresponding Airy factors are exponentially decreasing and provide an irrelevant contribution:  writing $\phi_{{a^{\natural}}}$ or $\psi_2(\cdot/{a^{\natural}})$ yields the same contribution in $G_{h,{a^{\natural}}}$ modulo $O(h^{\infty})$. 
In order to streamline notations, we use the same formula \eqref{greenfctbis} for each $G_{h,\gamma}$.  
From an operator point of view, with $G_h(\cdot)$ the semi-classical Schr\"odinger propagator, we are considering (with $i D=\partial$) $G_{h,\gamma}=
\psi_{2}((x+h^2D_{x}^{2})/\gamma) G_h$.
 \begin{rmq}\label{rmqsharp}
For $a\lesssim h^{2/3}$, it is easy to see that $ \|G_{h,h^{2/3}}(t,\cdot,a,\cdot)\|_{L^{\infty}}\lesssim \frac{1}{h}h^{1/3}$.
For $d\geq 2$, same estimates (hence with a loss of $h^{1/6}$ with respect to \eqref{sclSchrod} coming from the presence of the factor $h^{1/3}$ instead of $(\frac ht)^{1/2}$) had been obtained in \cite{doi} (where $q(\eta)=|\eta|^2$ but  the proof easily extends to a positive definite quadratic form $q$). The subsequent $1/6$ loss in homogeneous Strichartz estimates is optimal for $a\lesssim h^{2/3}$: in \cite[Theorem 1.8]{doi} we suitably chose Gaussian data whose associated semi-classical Schr\"odinger flow saturates the above bound (the so-called gallery modes).
 \end{rmq}

\subsection{A new form of the parametrix in terms of reflections}
Using \eqref{eq:AiryPoissonBis} on $G_{h,\gamma}$, we transform the sum over $k$ into a sum over $N\in \Z$, as follows
\begin{equation}\label{vhNgamN}
  G_{h,\gamma}(t,x,a)= \frac 1 {2\pi} \sum_{N\in \Z} \int_{\R}e^{-i NL(\omega)}h^{-2/3}e^{\frac ih t h^{2/3}\omega} \psi_{2}(h^{2/3}\omega/\gamma) Ai (x/h^{2/3}-\omega) Ai(a/h^{2/3}-\omega)d\omega.
\end{equation}
For $\sup{(a,h^{2/3})}\leq \gamma <\varepsilon_0$, let $\lambda_{\gamma} =\frac{\gamma^{3/2}}{h}$; when $h^{2/3}\lesssim a$ and $\gamma\sim a$ write $\lambda:=\frac{a^{3/2}}{h}$. Airy factors are (after rescaling)
\begin{equation}
  \label{eq:bis47}
  Ai(x/h^{2/3}-\omega) %
  =\frac{\lambda_{\gamma}^{1/3}}{2\pi } \int e^{i\lambda_{\gamma}(\frac{\sigma^{3}}{3}+\sigma (\frac{x}{\gamma} -\omega/\lambda_{\gamma}^{2/3}))} \,d\sigma.
\end{equation}
Rescaling $\omega=\lambda_{\gamma}^{2/3} \alpha=\gamma \alpha /h^{2/3}$ in \eqref{vhNgamN} yields $\alpha\sim 1$ on the support of $\psi_2$ and 
\begin{equation}
  \label{eq:bis48}
  G_{h,\gamma}(t,x,a)= \frac{\lambda_{\gamma}^{4/3}}{(2\pi)^{3}h^{2/3}} \sum_{N\in \Z} \int_{\R}\int_{\R^{2}} e^{\frac i h  \tilde\Phi_{N,a,\gamma}(\alpha,s,\sigma,t,x)} \psi_{2}(\alpha)\, ds d\sigma d\alpha\,,
\end{equation}
\begin{equation}
  \label{eq:bis49}
     \tilde\Phi_{N,a,\gamma}(\alpha,s,\sigma,t,x)=t\gamma\alpha-Nh L(\lambda_{\gamma}^{2/3} \alpha)
     +\gamma^{3/2} \Big(\frac{\sigma^{3}} 3+\sigma(\frac{x}{\gamma}-\alpha)
     +\frac {s^{3}} 3+s(\frac{a}{\gamma}-\alpha)\Big)\,.
\end{equation}
Here $Nh L(\lambda_{\gamma}^{2/3} \alpha)=\frac 43 N(\gamma\alpha)^{3/2}-NhB(\lambda_{\gamma} \alpha^{3/2})+Nh\pi/2$ and 
$B(\lambda_{\gamma} \alpha^{3/2})\sim_{1/(\lambda_{\lambda}\alpha^{3/2})} \sum_{k\geq 1}\frac{b_k}{(\lambda_{\gamma}\alpha^{3/2})^k}$.
Therefore,
\begin{equation}
  \label{eq:bis48bis}
    G_{h,\gamma}(t,x,a)= \frac 1 {(2\pi)^{3}} \frac{\gamma^2}{h^{2}}\sum_{N\in \Z} \int_{\R}\int_{\R^{2}} e^{\frac i h \tilde\Phi_{N,a,\gamma}}   \psi_{2}(\alpha)  \, ds d\sigma  d\alpha\,.
\end{equation}
Formulas \eqref{eq:bis48bis} and \eqref{greenfctbis} represent the same object and are both essential for establishing dispersive estimates. The eigenfunction expansion is most effective when $a\lesssim (ht)^{1/2}$, while the reflection sum becomes preferable for larger distances to the boundary. Though equivalent, the two are dual in nature: fewer terms appear in the eigenfunction sum for $a\lesssim (ht)^{1/2}$, and in the reflection sum for $a> (ht)^{1/2}$.

The symmetry of the Green function (or its suitable spectral truncations) with respect to $x$ and $a$ allows to restrict the computations of the $L^{\infty}$ norm to the region $0\leq x\leq a$. In other words, instead of evaluating $\|G^{\varepsilon_0}_h\|_{L^{\infty}(0\leq x)}(t,\cdot)$ it will be enough to bound $\|G^{\varepsilon_0}_h\|_{L^{\infty}(0\leq x\leq a)}(t,\cdot)$.

\section{Dispersive estimates for the semi-classical Schrödinger flow - proof of Theorem \ref{thm1D}}\label{sectdisp}
We prove dispersive bounds for $G^{\varepsilon_0}_h(t,x,a)$ on $\{x>0\}$ for fixed $|t|\in [h,T_0]$, with small $T_0>0$. We  estimate separately $\|G_{h,\gamma}(t,\cdot)\|_{L^{\infty}(x>0)}$ for every $\gamma$ such that ${a^{\natural}}\lesssim \gamma\leq \varepsilon_0$, where we recall that $a^{\natural}=\max{(a,h^{2/3})}$.
Henceforth we assume $t>0$. We sort out several situations, with a fixed (small) $\epsilon>0$. Firstly, $\max{(h^{2/3-\epsilon},(ht)^{1/2})}\leq a\leq \varepsilon_0$: in this case, for all $\gamma$ such that $a={a^{\natural}}\lesssim \gamma\leq \varepsilon_0$ we have $\max{(h^{2/3-\epsilon},(ht)^{1/2})}\leq a\lesssim\gamma\leq \varepsilon_0$. In this case, formula \eqref{eq:bis48bis} is particularly useful; integrals with respect to $\sigma,s$ have up to third order degenerate critical points and we perform a detailed analysis of these integrals. In particular, the "tangential" case $\gamma\sim a$ provides the worst decay estimates (see the first line of \eqref{eq:3}). When $8a\leq \gamma$, integrals in \eqref{eq:bis48bis} have degenerate critical points of order at most two. We call this regime "transverse": summing up $\sum_{8a\leq \gamma}\|G_{h,\gamma}(t,\cdot)\|_{L^{\infty}}$ still provides a better contribution than $\|G_{h,a}(t,\cdot)\|_{L^{\infty}}$.
Secondly, for $a\lesssim \max{(h^{2/3-\epsilon},(ht)^{1/2})}$, we further subdivide: either $\max{(h^{2/3-\epsilon},(ht)^{1/2})}\leq \gamma \leq \varepsilon_0$, which is similar to the previous "transverse" regime, and estimates will follow using \eqref{eq:bis48bis} ; or ${a^{\natural}}\lesssim \gamma\lesssim \max{(h^{2/3-\epsilon},(ht)^{1/2})}$, and we use \eqref{greenfctbis} to evaluate the $L^{\infty}$ norm of $G_{h,\gamma}$ and its sum over relevant $\gamma$'s. In this regime, the method of \cite{Iva23} only gives the bounds in the last line of \eqref{eq:3}; to obtain better estimates, we use Van der Corput's $j$th derivative test estimates (or generalized Lindelöf bounds for sharp results). In fact, for all $T:=t/\sqrt{a}\geq (a^{3/2}/h)^{1/3}=\lambda^{1/3}$, the higher order derivatives of the exponential functions in \eqref{greenfctbis} behave (more or less) like the ones of $e^{T\lambda (k/\lambda)^{2/3}}$ hence Van der Corput type bounds do hold and provide better estimates than in \cite{Iva23}.

\subsection{Case $\max{(h^{2/3-\epsilon},(ht)^{1/2})}\leq a \leq \varepsilon_0$, with  (small) $\epsilon>0$}\label{sec>12}
As ${a^{\natural}}=a$, we consider $\gamma$ such that $a\leq \gamma\leq \varepsilon_0$. Let $\lambda_{\gamma}:=\gamma^{3/2}/h$, then $\lambda_{\gamma}\geq h^{-3\epsilon/2}$. While the approach in this section applies for all $h^{2/3-\epsilon}\lesssim a\leq \varepsilon_0$, when summing up over $a\l\lesssim \gamma\leq (ht)^{1/2}$, bounds for $G^{\varepsilon_0}_h$  get worse than announced in Theorem \ref{thm1D}. Hence we restrict to values $\max{(h^{2/3-\epsilon},(ht)^{1/2})}\leq a \leq \varepsilon_0$, while lesser values will be dealt with differently later. First, we prove that the sum defining $G_{h,\gamma}$ in \eqref{eq:bis48bis} over $N$ is essentially finite and we estimate the number of terms in the relevant sum.
\begin{prop}\label{propcardN}
For a fixed $t\in (h,T_0]$ the sum \eqref{eq:bis48bis} over $N$ is essentially finite and $0\leq N\lesssim \frac{t}{\sqrt{\gamma}}$. In other words, if $M$ is a sufficiently large constant, then
\[
 \frac 1 {(2\pi)^{3}} \frac{\gamma^2}{h^{2}}\sum_{N\in \N, N\geq \frac{M t}{\sqrt{\gamma}}} \int_{\R}\int_{\R^{2}} e^{\frac i h \tilde\Phi_{N,a,\gamma}} \psi_{2}(\alpha)  \, ds d\sigma  d\alpha=O(h^{\infty}).
\]
\end{prop}
\begin{proof}
The proof follows easily using non-stationary phase arguments for $N\geq M  \frac{t}{\sqrt{\gamma}}$ for some $M$ sufficiently large. Critical points with respect to $\sigma,s$ are such that
\begin{equation}\label{statssigma}
\sigma^2=\alpha-{x}/{\gamma},\quad s^2=\alpha-{a}/{\gamma},
\end{equation}
and as $x\geq 0$, $\tilde\Phi_{N,a,\gamma}$ may be stationary in $\sigma$, $s$ only if $|(\sigma,s)|\leq \sqrt{\alpha}$. As $\psi_2(\alpha)$ is supported near $1$, it follows that we must also have $x\leq 2\gamma$, otherwise $\tilde\Phi_{N,a,\gamma}$ is non-stationary with respect to $\sigma$.
If $|(\sigma,s)|\geq (1+N^{\epsilon})\sqrt{\alpha}$ for some $\epsilon>0$ we can perform repeated integrations by parts in $\sigma,s$ to obtain $O(((1+N^{\epsilon})\lambda_{\gamma})^{-n})$ for all $n\geq 1$. Let $\chi$ a smooth cutoff supported in $[-1,1]$ and write $1=\chi(\sigma/(N^{\epsilon}\sqrt{\alpha}))+(1-\chi)(\sigma/(N^{\epsilon}\sqrt{\alpha}))$, then
\begin{multline}
\sum_{N\in \mathbb{Z}} \int_{\R}\int_{\R^{2}} e^{\frac i h  \tilde\Phi_{N,a,\gamma}}  
\psi_{2}(\alpha)\chi(s/(N^{\epsilon}\sqrt{\alpha})) (1-\chi)(\sigma/(N^{\epsilon}\sqrt{\alpha})) \, ds d\sigma  d\alpha \\
\lesssim \lambda_{\gamma}^{-1/3} \sup_{\alpha,|\eta|\in[1/2,3/2]} \Big|Ai\Big((a-\gamma \alpha)/h^{2/3}\Big)\Big|\sum_{N\in \mathbb{Z}} \Big(((1+N^{\epsilon})\lambda_{\gamma})^{-n}\Big) %
=O(h^{\infty})\,,
\end{multline}
where in the last line we used $\lambda_{\gamma}\geq h^{-3\epsilon/2}$, $\epsilon>0$. In the same way, we can sum on the support of $(1-\chi)(s/(N^{\epsilon}\sqrt{\alpha}))$ and obtain a $O(h^{\infty})$ contribution. Therefore, we may add cut-offs $\chi(\sigma/(N^{\epsilon}\sqrt{\alpha}))$ and $\chi(s/N^{\epsilon}\sqrt{\alpha}))$ in $G_{h,\gamma}$ without changing its contribution modulo $O(h^{\infty})$. Using again \eqref{eq:bis49}, we have, at the critical point of $\tilde\Phi_{N,a,\gamma}$
\begin{equation}\label{statalpha}
\frac{t}{{\gamma}^{1/2}}-(s+\sigma)=2N\sqrt{\alpha}\Big(1-\frac 34 B'(\lambda \alpha^{3/2})\Big),
\end{equation}
and as $|(\sigma,s)|/\sqrt{\alpha}\lesssim 1+N^{\epsilon}$ on the support of $\chi(\sigma/(N^{\epsilon}\sqrt{\alpha}))\chi(s/(N^{\epsilon}\sqrt{\alpha}))$, $\tilde\Phi_{N,a,\gamma}$ may be stationary with respect to $\alpha$ only when $\frac{t}{\sqrt{\gamma}}\sim 2N$. As  $B'(\lambda \alpha^{3/2})=O(\lambda_{\gamma}^{-3})=O(h^{9\epsilon/2})$, its contribution is irrelevant. From \eqref{statssigma} and \eqref{statalpha}, if $\frac{t}{{\gamma}^{1/2}}\notin [2(N-1),2(N+1)]\sqrt{\alpha}$, $\alpha\in[\frac 12,\frac 32]$, then the phase is non-stationary in $\alpha$.
Repeated integrations by parts allow to sum up in $N$ as above, and conclude.
\end{proof}

\begin{rmq}\label{rmqA}
We can in fact add an even better localization with respect to $\sigma$ and $s$: on the support of $(1-\chi)(\sigma/(2\sqrt{\alpha}))$ and $(1-\chi)(s/(2\sqrt{\alpha}))$ the phase is non-stationary in $\sigma$ or $s$, and integrations by parts yield an $O(\lambda_{\gamma}^{-\infty})$ contribution. According to Proposition \ref{propcardN}, the sum over $N$ has finitely many terms, and summing yields an $O(h^{\infty})$ contribution.
\end{rmq}

\begin{lemma}\label{rmqB}
For $\gamma\gtrsim a\geq (ht)^{1/2}$, the factor $e^{iNB( \lambda_{\gamma}\alpha^{3/2})}$ can be moved into the symbol.
\end{lemma}
\begin{proof}
As $\alpha \in [\frac 12,\frac 32]$ and $N\sim \frac{t}{\sqrt{\gamma}}$, Lemma \ref{lemL} gives
$NB(\lambda_{\gamma} \alpha^{3/2})\sim N\sum_{k\geq 1}\frac{b_k}{(\lambda_{\gamma} \alpha^{3/2})^k} \sim \frac{N b_1}{\lambda_{\gamma}}\sim \frac{ht}{\gamma^2}$. As we consider here only values $(ht)^{1/2}\lesssim \gamma$, this term remains bounded (so it does not oscillate).
\end{proof}
Let $\Phi_{N,a,\gamma}=\tilde\Phi_{N,a,\gamma}-NhB(\lambda_{\gamma}\alpha^{3/2})$, then, by Lemma \ref{rmqB}, $\Phi_{N,a,\gamma}$ are the phase functions of $G_{h,\gamma}$ from \eqref{eq:bis48bis}, where
\begin{equation}
  \label{eq:bis49PhiN}
     \Phi_{N,a,\gamma}(\alpha,s,\sigma,t,x)=t\gamma\alpha 
     +\gamma^{3/2} \Big(\frac{\sigma^{3}} 3+\sigma(\frac{x}{\gamma}-\alpha)
     +\frac {s^{3}} 3+s(\frac{a}{\gamma}-\alpha)-\frac 43 N\alpha^{3/2}\Big).
\end{equation}
In the following we study, at fixed $N\sim \frac{t}{\sqrt{\gamma}}$, the integral appearing in the sum \eqref{eq:bis48bis}. Notice that the integral corresponding to $N=0$ is the free semi-classical Schr\"odinger flow, and the sum over $\gamma\in \Gamma_{0}({a^{\natural}})=\{ \gamma=2^{2j} {a^{\natural}}, 0\leq j<\frac 12 \log_2(\varepsilon_0/{a^{\natural}})\}$ satisfies the dispersive estimates follow as in $\R$. Let therefore $N\geq 1$.

\begin{prop}\label{propcritptsalphaeta}
The phase function $\Phi_{N,a,\gamma}$ can have at most one critical point $\alpha_c$ on the support $[\frac 12, \frac 32]$ of $\psi_2$. At critical points in $\alpha$,  the determinant of the Hessian matrix is comparable to $\gamma^{3/2}N$, $N\geq 1$. The stationary phase applies in $\alpha$ and yields a decay factor ${(\lambda_{\gamma}N)}^{-1/2}$.
\end{prop}
\begin{proof}
The derivative of the phase $\Phi_{N,a,\gamma}$ with respect to $\alpha$ is
$\partial_{\alpha}\Phi_{N,a,\gamma}=\gamma^{3/2}\Big(\frac{t}{\sqrt{\gamma}}-(\sigma+s)-2N\sqrt{\alpha}\Big)$.
At $\partial_{\alpha}\Phi_{N,a,\gamma}=0$, the critical point is such that
\begin{equation}\label{eq:alphac}
\sqrt{\alpha}=\frac{t}{2N\sqrt{\gamma}}-\frac{s+\sigma}{2N}, \quad \alpha\in [\frac 12,\frac 32].
\end{equation}
At the stationary point in $\alpha$ we get a decay factor 
$({\lambda_{\gamma}N})^{-1/2}$ as $|\frac 1h\partial^2\Phi_{N,a,\gamma}|\Big|_{\partial\Phi_{N,a,\gamma}=0}\sim \lambda_{\gamma}N$.
\end{proof}
\begin{cor}\label{coruhgam}
We have $G_{h,\gamma}(t,x)=\frac{1}{h}\sum_{N\sim \frac{t}{\sqrt{\gamma}}} V_{N,h,\gamma}(t,x) +O(h^{\infty})$, where
\begin{equation}
V_{N,h,\gamma}(t,x)=\frac{\gamma^2}{h}\frac{1}{\sqrt{\lambda_{\gamma}N}}\int e^{\frac ih \phi_{N,a,\gamma}(\sigma,s,t,x)}\varkappa(\sigma,s,t,x;h,\gamma,1/N)d\sigma ds\,,
\end{equation}
$\phi_{N,a,\gamma}(\sigma,s,t,x)=\Phi_{N,a,\gamma}(\alpha_c,\sigma,s,t,x)$ and $\varkappa(\cdots;h,\gamma,1/N)$ has main contribution $\psi_2(\alpha_c)e^{iNB(\lambda_{\gamma}\alpha_{c}^{3/2})}$.
\end{cor}
This immediately follows from stationary phase in $\alpha$, with $\psi_2(\alpha_{c})e^{iNB(\lambda_{\gamma}\alpha_{c}^{3/2})}$ as leading order term for $\varkappa$.
Notice that this main contribution for the symbol $\varkappa(\cdot;h,\gamma,1/N)$ has an harmless dependence on the parameters $h,a,\gamma, 1/N$, as $\varkappa(\cdot,h,\gamma,1/N)$ reads as an asymptotic expansion with small parameters $(\lambda_{\gamma}N)^{-1}=h/(N\gamma^{3/2})$ in $\alpha$, and all terms in the expansions are smooth functions of $\alpha_c$.
Using Remark \ref{rmqA}, we may introduce cut-offs $\chi(\sigma/(2\sqrt{\alpha_c}))$ and $\chi(s/(2\sqrt{\alpha_c}))$, supported for $|(\sigma,s)|\leq 2\sqrt{\alpha_c}$ in $V_{N,h,\gamma}$ without changing its contribution modulo $O(h^{\infty})$.

\subsubsection{"Tangential" waves $a\sim \gamma$}\label{secttangcas}

We abuse notations and write $G_{h,a}=G_{h,\gamma\sim a}$, $\lambda=a^{3/2}/h=\lambda_{a}$ and using Corollary \ref{coruhgam}, with $\phi_{N,a}(\sigma,s,t,x,y)=\Phi_{N,a,a}(\eta_c,\alpha_c,\sigma,s,t,x,y)$, we get
\begin{gather}\label{uhsumNa}
G_{h,a}(t,x)=\frac{1}{h}\sum_{N\sim \frac{t}{\sqrt{a}}} V_{N,h,a}(t,x) +O(h^{\infty})\,,\\
\label{defVNha}
V_{N,h,a}(t,x)=\frac{a^2}{h}\frac{1}{\sqrt{\lambda N}}\int e^{\frac ih \phi_{N,a}(\sigma,s,t,x)}\varkappa(\sigma,s,t,x,h,a,1/N)d\sigma ds\,.
\end{gather}
As $\gamma\geq (ht)^{1/2}$, only values $N\lesssim \lambda$ are of interest : indeed, $N\lesssim t/\sqrt{\gamma}\leq \gamma^{3/2}/h=\lambda_{\gamma}$. Fix $t$ and set $T=\frac{t}{\sqrt{a}}$ : notice that, if $\lambda^{1/3}\lesssim T\sim N$, then $\phi_{N,a}$ behaves like the phase of a product of two Airy functions and can be bounded using mainly their respective asymptotic behaviour. When $T\sim N< \lambda^{1/3}$, $\phi_{N,a}$ may have degenerate critical points up to order $3$ and we claim that there exists a sequence of times $T_n =2n<\lambda^{1/3}$, $n\in \mathbb{N}$ such that
\[\|G_{h,a}(t,\cdot)\|_{L^{\infty}(x>0)}\Big|_{t/\sqrt{a}=T_n=2n}\sim \frac{1}{h}\Big(\frac{ha}{t}\Big)^{1/4}\,,\quad \forall h^{1/3}t\leq a\lesssim \varepsilon_0 \quad (\text{ i.e. } \forall 1\leq T<\lambda^{1/3})
\]
as in the first line of \eqref{eq:3}. For all other values of $t$ the bounds are better.
Let $T=\frac{t}{\sqrt{a}}$ and $K=\sqrt{\frac{T}{2N}}$. 
\begin{prop}
\label{dispNgrand}
For $\lambda^{1/3}\lesssim T\sim N$, $\frac xa\leq 1$, we have
  \begin{equation}
    \label{eq:1ff}
       \left| V_{N,h,a}(t,x)\right|\lesssim  \frac {h^{1/3}} {(N/\lambda^{1/3})^{1/2} +\lambda^{1/6}\sqrt{4N}|K-1|^{1/2}}\,.
  \end{equation}
\end{prop}

\begin{prop}
\label{dispNpetitloin}
For $1\leq T\sim N<\lambda^{1/3}$, $K=\sqrt{\frac{T}{2N}}$ such that $|K-1|\gtrsim 1/N^2$, $\frac xa\leq 1$ we have
  \begin{equation}
    \label{eq:2ff}
       \left| V_{N,h,a}(t,x)\right| \lesssim\frac{h^{1/3}}{(1+2N|K-1|^{1/2})}\,.
  \end{equation}
\end{prop}

\begin{prop}
\label{dispNpetitpres}
For $1\leq T\sim N<\lambda^{1/3}$, $K=\sqrt{\frac{T}{2N}}$ such that $|K-1|\leq\frac{1}{4N^2}$, $\frac xa\leq 1$ we have
  \begin{equation}
    \label{eq:2hh}
       \left| V_{N,h,a}(t,x)\right| \lesssim \frac{h^{1/3}}{(N/\lambda^{1/3})^{1/4}+N^{1/3}|K-1|^{1/6}}\,.
  \end{equation}
Moreover, at $x=a$ and $K=1$ we have $\left| V_{N,h,a}(t,a)\right| \sim \frac{h^{1/3}}{(N/\lambda^{1/3})^{1/4}}$.
\end{prop}
We postpone the proofs of Propositions \ref{dispNgrand}, \ref{dispNpetitloin} and \ref{dispNpetitpres} to Section \ref{sectproofsprops} and we complete the proof of Theorem \ref{thm1D} in the case $(ht)^{1/2}\lesssim a\sim \gamma \leq \varepsilon_0<1$. 
Let $\sqrt{a} \lesssim t\lesssim 1$ be fixed and let $N_t\geq 1$ be the unique positive integer such that $T=\frac{t}{\sqrt{a}}> N_t\geq \frac{t}{\sqrt{a}}-1=T-1$, hence $N_t=[T]$, where $[T]$ denotes the integer part of $T$. 
If $N_t$ is bounded then the number of $V_{N,h,a}$ with $N\sim N_t$ in the sum \eqref{uhsumNa} is also bounded and we can easily conclude adding the (worst) bound from Proposition \ref{dispNpetitpres} a finite number of times.
Assume $N_t\geq 2$ is large enough. 

\begin{prop}\label{propsumNgrand}
There exists $C>0$ (independent of $h,a$) such that, if $N_t:=[\frac{t}{\sqrt{a}}]\geq  \lambda^{1/3}$, 
\[
\|G_{h,a}(t,\cdot)\|_{L^{\infty}(x>0)}\leq \frac{C}{h}\Big(\Big(\frac{ht}{a}\Big)^{1/2}+h^{1/3}\Big)\,.
\]
\end{prop}

\begin{proof}
If $\lambda^{1/3}< N_t$, then we estimate the $L^{\infty}$ norms of $G_{h,a}(t,\cdot)$ using Proposition \ref{dispNgrand}.
For $2N=N_t+j$ and $2\leq |j|\leq N_t/2$, we have $\Big|2NK-2N\Big|=2N|\sqrt{\frac{T}{2N}}-1|=\frac{2N|T-2N|}{\sqrt{2N}(\sqrt{2N+T})}\geq |j|-1$,
 and therefore
 \begin{equation}\label{estimsumNgrandj}
 \sum_{N\sim N_t}|V_{N,h,a}|\lesssim 
 \frac{h^{\frac 13}}{\sqrt{N_t}} \Big(3\lambda^{\frac 16}+ \sum_{|2N-N_t|=|j|\geq 2}\frac {\lambda^{\frac 16}} {(1+j/N_t)^{1/2}+\lambda^{\frac 13}|(|j|-1)/N_t|^{\frac 1 2}}\Big).
\end{equation}
The sum over $2N=N_t\pm(j+1)$, $1\leq j\leq N_t/2$, read as
\begin{multline}\label{estsumVNTgrand}
\frac{h^{1/3}N_t^{1/2}}{\lambda^{1/6}N_t}\sum_{2N= N_t\pm(j+1), j\geq 1}\frac {1} {(1\pm (j+1)/N_t)^{1/2}\lambda^{-1/3} +|j/N_t|^{1/2}}\\
\leq h^{1/3}\frac{\sqrt{N_t}}{\lambda^{1/6}}\sum_{\pm}\int_0^{1/2}\frac{dx}{\sqrt{x}+\lambda^{-1/3}(1\pm N_t^{-1}\pm  x)^{1/2}}
\lesssim h^{1/3}\Big(\frac{t/\sqrt{a}}{\lambda^{1/3}}\Big)^{1/2} =\Big(\frac{ht}{a}\Big)^{1/2},
\end{multline}
which achieves the proof of Proposition \ref{propsumNgrand}.\end{proof}

\begin{prop}\label{proptangmain}
There exists $C>0$ (independent of $h,a$) such that, if $T=N_t=[\frac{t}{\sqrt{a}}]< \lambda^{1/3}$, then 
\begin{equation}\label{ptmaxGha}
\|G_{h,a}(t,\cdot)\|_{L^{\infty}(x>0)}\sim \frac{C}{h}\Big(\Big(\frac{ha}{t}\Big)^{1/4}+h^{1/3}\Big).
\end{equation}
\end{prop}
\begin{proof}
For all such $N$ we then use Proposition \ref{dispNpetitloin} to obtain
\begin{multline}\label{smalltermsinthesum}
\sum_{2N\sim N_t, N \neq N_t} |V_{N,h,a}|
 \lesssim  h^{1/3} \sum_{2N\sim N_t, N\neq N_t}\frac {1} {1 +(2N)^{3/4}|\sqrt{T}-\sqrt{2N}|^{1/2}}\\
 \lesssim h^{1/3}\sum_{2N=N_t+j, 1\leq |j|\lesssim N_t/2}\frac {1} {1 +(N_t+j)^{1/2}|j|^{1/2}}
 \leq h^{1/3}\sum_{\pm}\int_0^1\frac{dx}{x^{1/2}(1\pm x)^{1/2}+N_t^{-1}},
\end{multline}
where the last two integrals are uniform bounds for the sum over $2N\sim N_t$ with $2N< N_t$ or $2N>N_t$, respectively. When $2N>N_t$, the integral over $[0,1]$ is bounded by a uniform constant while when $2N< N_t$, write $x=\sin^2 \theta$, $\theta\in [0,\pi/2)$, therefore $1-x=\cos^2\theta$, $dx=2\sin\theta\cos\theta$ : the corresponding integral is also bounded by at most $\pi$.

When $2N=N_t$ we apply Proposition \ref{dispNpetitpres} with $N=N_t$ provided that we have $\Big|T-2N\Big|=|T-[T]|\lesssim \frac{1}{N}\,$, 
otherwise we apply again Proposition \ref{dispNpetitloin} and find
\begin{equation}\label{smallTinteger}
 |V_{N_t,h,a}|  \lesssim \frac{h^{\frac 13}}{(N/\lambda^{\frac 13})^{\frac 14}}+\frac {h^{\frac 1 3}} {(1+N_t^{3/4}|\sqrt{T}-\sqrt{2N}|^{\frac 12})} \lesssim\Big(\frac{ha}{t}\Big)^{1/4}+ h^{1/3}\,.
\end{equation}
As for $\frac{t}{\sqrt{a}}\ll \frac{\sqrt{a}}{h^{1/3}}=\lambda^{1/3}$ we have $h^{1/3}\ll \Big(\frac{ha}{t}\Big)^{1/4}$, it follows that at fixed $t$, the supremum of the sum over $V_{N,h,a}(t,x)$ is reached at $x=a$. As the contribution from \eqref{smalltermsinthesum} in the sum over $2N\neq N_t$ is bounded by $ h^{1/3}$, we obtain an upper bound for $G_{h,a}(t,\cdot)$. The last line of \eqref{smallTinteger} and the strict inequality $h^{1/3}\ll \Big(\frac{ha}{t}\Big)^{1/4}$ provide a similar lower bound for $G_{h,a}$ and therefore \eqref{ptmaxGha} holds true, concluding the proof of Proposition \ref{proptangmain}.
\end{proof}

\subsubsection{"Transverse" waves $\gamma=2^{2j}a$, $j\geq 1$}\label{sectransvwa}
Let $\gamma>8a$ and recall $\lambda_{\gamma}:=\frac{\gamma^{3/2}}{h}$.
\begin{prop}\label{proptransvmain}
Let $t>h$ %
and $\varepsilon_0>\gamma\geq 4a$. Let $T_{\gamma}:=\frac{t}{\sqrt{\gamma}}$.
\begin{equation}\label{transwave}
\|G_{h,\gamma}(t,\cdot)\|_{L^{\infty}(x\leq a)}\lesssim 
\left\{ \begin{array}{l} 
\frac{1}{h}\Big({\frac{th}{\gamma}}\Big)^{1/2}, \text{ if }\frac{t}{\sqrt{\gamma}}\geq \lambda_{\gamma}^{1/3}, \,\\
\frac{1}{h} h^{1/3}, \text{ if } 1/4\leq  \frac{t}{\sqrt{\gamma}}<\lambda_{\gamma}^{1/3},\\
\frac{1}{h}\Big(\frac ht \Big)^{\frac 12},\text{ if } \frac{t}{\sqrt{\gamma}}<1/4.
 \end{array} \right.
\end{equation}
\begin{equation}\label{estimtransvN>1}
\sum_{\gamma \in \Gamma_{1}(a)} \|G_{h,\gamma}(t,\cdot)\|_{L^{\infty}(x\leq a)}\lesssim 
\left\{ \begin{array}{l} 
\frac{1}{h} h^{1/3}\log_2(\frac{\varepsilon_0}{a}), \text{ if } a\lesssim  t\leq \frac{a}{h^{1/3}}(<\frac{\gamma}{h^{1/3}}),\\
 \frac{1}{h}\Big[\Big(\frac{ht}{a}\Big)^{\frac 12}+h^{\frac 13}\log_2(\frac{\varepsilon_0}{a})\Big],\text{ if }  t\geq \frac{a}{h^{1/3}}.
  \end{array} \right.
\end{equation}
\end{prop}
\begin{proof}
The last line in \eqref{transwave} follows as the time is too small for the waves to reach the boundary. 
Let $T_{\gamma}:=\frac{t}{\sqrt{\gamma}}\geq 1/4$. Let $V_{N,h,\gamma}$ as in Corollary \ref{coruhgam}, then $G_{h,\gamma}(t,x,y)=\sum_{N\sim T_{\gamma}}V_{N,h,\gamma}(t,x,y)$. 
For $x\leq a$, $4a\leq \gamma$ and $1\leq N\sim T_{\gamma}$ the following holds
  \begin{equation}
    \label{eq:1ffgam}
      \left| V_{N,h,\gamma}(t,x,y)\right|\lesssim 
       \frac{\gamma^2}{h}\times\frac{1}{\sqrt{N\lambda_{\gamma}}}\times \frac{1}{\lambda_{\gamma}}.
  \end{equation}
Indeed, as long as $x\leq a$, we easily see that, for each $N$, the phase function of $V_{N,h,\gamma}$ has non-degenerate critical points with respect to both $\sigma, s$, hence the estimate \eqref{eq:1ffgam} follows.
Summing up over $N\gtrsim \lambda_{\gamma}^{1/3}$ as in the proof of Proposition \ref{propsumNgrand} yields the first line of \eqref{transwave}. Summing over $N\lesssim \lambda_{\gamma}^{1/3}$ as in the proof of Proposition \ref{proptangmain} yields the second line of \eqref{transwave}, (but where the main contribution $(h\gamma/t)^{1/4}$ is missing as it occurs only for $\gamma=a$ and not when $\gamma\geq 4a$).

Let $h^{1/3} t< a\leq \gamma/4$, then $T_{\gamma}\leq \lambda_{\gamma}^{1/3}$.
Summing up for $\gamma_j=2^{2j} a$, yields the first line in \eqref{estimtransvN>1}, as $j\leq \frac 12 \log_2(\frac{\varepsilon_0}{a})$. Let now $a/h^{1/3}\leq t\leq T_0$, then for $4a\leq\gamma \lesssim th^{1/3}$, $|G_{h,\gamma}(t,\cdot)|$ is bounded as in the first line of \eqref{transwave}, while for $th^{1/3}\leq \gamma \leq \varepsilon_0$, it is bounded as in the second line of \eqref{transwave}. The sum for $\gamma_j=2^{2j} a$ over $j\leq \frac 12 \log_2(\frac{\max(\varepsilon_0,th^{1/3})}{a})$ and over $\frac{\max(\varepsilon_0,th^{1/3})}{a}< j\leq \frac 12 \log_2(\frac{\varepsilon_0}{a})$ yields the first and second contributions of \eqref{estimtransvN>1} .
\end{proof}
Gathering Propositions \ref{propsumNgrand}, \ref{proptangmain} and \ref{proptransvmain} we obtain the upper bound from Theorem \ref{thm1D} in the range $(ht)^{1/2}\lesssim a\leq \varepsilon_0$.

\subsubsection{Optimality for $1\leq t/\sqrt{a} \ll\frac{\sqrt{a}}{h^{1/3}}$} \label{sectopt} In this case we have $a^{\natural}=a$.
The first line in \eqref{eq:3} follows easily from the next lemma, considering the reductions we performed earlier.
\begin{lemma}\label{lemop}
For $\sqrt{a}\leq t\ll \frac{a}{h^{1/3}}$ we have $\|G^{\varepsilon_0}_h(t,\cdot)\|_{L^{\infty}(x\leq a)}\sim  \frac{1}{h}\Big(\frac{ah}{t}\Big)^{1/4}$. 
\end{lemma}
\begin{proof}
Write, for $1\leq \frac{t}{\sqrt{a}}\leq \frac{1}{16}\frac{\sqrt{a}}{h^{1/3}}=\frac{1}{16} \lambda^{1/3}$ and $\Gamma_{0}({a^{\natural}})=\{ \gamma=\gamma_j=2^{2j} {a^{\natural}}=2^ja, 0\leq j<\frac 12 \log_2(\varepsilon_0/{a^{\natural}})\}$ 
\[
\|G^{\varepsilon_0}_h(t,\cdot)\|_{L^{\infty}(x\leq a)}\geq \|G_{h,a}(t,\cdot)\|_{L^{\infty}(x\leq a)}-\sum_{\gamma \in \Gamma_{0}(a) } \|G_{h,\gamma_j}(t,\cdot)\|_{L^{\infty}(x\leq a)}.
\]
From \eqref{ptmaxGha} we have $\|G_{h,a}(t,\cdot)\|_{L^{\infty}(x\leq a)}\sim \frac{1}{h}\Big(\frac{ah}{t}\Big)^{1/4}$ and from the first line of \eqref{estimtransvN>1} %
we have
\[
\sum_{\gamma\in \Gamma_{0}(a)} \|G_{h,\gamma_j}(t,\cdot)\|_{L^{\infty}(x\leq a)}\leq \frac{1}{h}h^{1/3}\log_2 (\frac{\varepsilon_0}{a})\,.
\]
Notice that $\Big(\frac{ah}{t}\Big)^{1/4}> h^{1/3}$ $\forall t$ such that $1\leq \frac{t}{\sqrt{a}}< \lambda^{1/3}=\frac{\sqrt{a}}{h^{1/3}}$. Taking $T=t/\sqrt{a}\leq \lambda^{1/3-\epsilon}$ for any $\epsilon>0$ yields t $\Big(\frac{ah}{t}\Big)^{1/4}\gg h^{1/3}\log_2(\frac{\varepsilon_0}{a})$. This concludes our proof.
\end{proof}

\subsection{Case $a\lesssim \max{(h^{2/3-\epsilon},(ht)^{1/2})}$ for (small) $\epsilon>0$}\label{sectasmall}

\subsubsection{The sum over ${a^{\natural}}\lesssim \gamma\lesssim \max{(h^{2/3-\epsilon},(ht)^{1/2})}$ }
In \cite{Iva23}, this part has been entirely dealt with (in dimension $d\geq 2$) using the spectral sum \eqref{greenfctbis} and the next Lemma.

\begin{lemma}\label{lemsob}
There exists $C_0$ such that for $L\geq 1$ the following holds true
\begin{equation}\label{estairy2}
\sup_{b\in \R}\Big (\sum_{1\leq k\leq L}\omega_k^{-1/2}Ai^2(b-\omega_{k})\Big) \leq C_{0}L^{1/3}\,.
\end{equation}
\end{lemma}
Taking $L=\lambda_{\gamma}=\gamma^{3/2}/h$, then $L=\lambda_{max}:=((ht)^{1/2})^{3/2}/h$ gives, respectively  
\begin{prop}\label{prop11}
For $t\in (h,T_0]$, the following dispersive estimates hold
\begin{equation}
\|G_{h,\gamma}(t,\cdot)\|_{L^{\infty}(x\geq a)}\lesssim \frac{1}{h^{2/3}}\lambda_{\gamma}^{1/3}\,,
\end{equation}
\begin{equation}
\|\sum_{{a^{\natural}}\leq\gamma\leq (ht)^{1/2}}G_{h,\gamma}(t,\cdot)\|_{L^{\infty}(x\geq a)}\lesssim \frac{1}{h^{2/3}}\lambda_{max}^{1/3}=
 \frac{1}{h} (h\lambda_{\max})^{1/3}=\frac1h (ht)^{1/4}\,.
\end{equation}
\end{prop}

Gathering the previous bounds, we therefore complete the proof of the upper bound of Theorem \ref{thm1D}. Notice that in the regime $\gamma\lesssim \max(h^{2/3-\epsilon},(ht)^{1/2})$ the loss $1/4$ occurs for all $t\sim1$ and cannot be improved using Lemma \ref{lemsob}.
To do better than Proposition \ref{prop11} we use the Van der Corput estimates for higher order derivatives.

\section{Proof of Theorem \ref{thmStritoutT}}\label{secT<}

Let $G_h^{\varepsilon_0}(t,x,a)$ be the Green function for \eqref{PP} for some small, fixed $\varepsilon_0\in (0,1)$, independent of $h,a$, as in \eqref{greenfctbiseps0}. For a compactly supported function $f$ in the variables $(s,a\geq 0)$, we set
\[
\mathcal{A}(f)(t,x):=\int G_h^{\varepsilon_0}(t-s,x,a)f(s,a) ds da.
\]
For $d=1$, the Strichartz endpoints - such that $\frac 1q=\frac d2 (\frac 12-\frac 1r)$ with $d=1$) - are $q=4, r=\infty$. We need to prove that the operator $\mathcal{A}$ is bounded from $L^{4/3}_t L^1(0,\infty)$ to $L^4_t L^{\infty}(0,\infty)$ with a norm of at most $h^{-(1/2+1/6+5/114)}$, that is 
\begin{equation}\label{estnormA}
\|\mathcal{A}(f)\|_{L^4_t L^{\infty}(0,\infty)}\lesssim \frac{1}{h^{2/3+5/114}}\|f\|_{L^{4/3}_t L^1(0,\infty)}.
\end{equation}
Indeed, if \eqref{estnormA} holds, it means that the operator $\mathcal{T}:L^2(0,\infty)\rightarrow L^4(0,t_0)L^{\infty}(0,\infty)$, which to $v_0$ associates $v_h$ and whose adjoint $\mathcal{T}^*:L^{4/3}(0,t_0)L^1(0,\infty)\rightarrow L^2(0,\infty)$ satisfies $\mathcal{A}=\mathcal{T}\mathcal{T}^*$, is such that 
\[
\|\mathcal{T}\|_{L^2(\R_+)\rightarrow L^4_tL^{\infty}(\R_+)}\lesssim h^{-(1/2+1/6+5/114)\times 1/2},
\]
which in turn means that \eqref{Stri1D} holds. In order to prove \eqref{estnormA}, we first write
\begin{multline}
|\mathcal{A}(f)(t,x)|\leq \int \sup_{x\leq a} |G_h^{\varepsilon_0}(t-s,x,a)| |f(s,a)| ds da
=\int (\sup_{x\leq a} |G_h^{\varepsilon_0}(\cdot,x,a)| * |f(\cdot,a)|)(t) da\\
\leq \Big(\sup_{a,x\leq a} |G_h^{\varepsilon_0}(\cdot,x,a)| * \|f(\cdot,\cdot)\|_{L^1(0,\infty)}\Big)(t).
\end{multline}
Using Young inequality for the convolution product $\|G*F\|_{L^{r_1}}\leq \|G\|_{L^{p_1}}\|F\|_{L^{q_1}}$ for $1+\frac{1}{r_1}=\frac{1}{p_1}+\frac{1}{q_1}$ with $r_1=4$, $p_1=2$ and $q_1=4/3$ and taking $G:=\sup_{a,x\leq a} |G_h^{\varepsilon_0}(\cdot,x,a)|$ and $F:=\|f(\cdot,\cdot)\|_{L^1(0,\infty)}$ yields
\[
\|\mathcal{A}\|_{L^4(0,t_0)L^{\infty}(\R_+)}\leq \Big\|\sup_{a,x\leq a} |G_h^{\varepsilon_0}(\cdot,x,a)|\Big \|_{L^2(0,t_0)}\times \|f\|_{L^{4/3}(0,t_0)L^1(\R_+)}.
\]
Therefore, we are left to prove that $ \Big\|\sup_{a,x\leq a} |G_h^{\varepsilon_0}(\cdot,x,a)|\Big \|_{L^2(0,t_0)}\lesssim \frac{1}{h^{2/3+5/114}}$.
To do that, write
\[
G_h^{\varepsilon_0}(t,x,a)=G_h^{\varepsilon_0}(t,x,a)\times 1_{t<a/h^{1/3}}+ G_h^{\varepsilon_0}(t,x,a)\times 1_{t\geq a/h^{1/3}}.
\]
From (the proof of) Theorem \ref{thm>}, we have that $\sup_{a,x\leq a} |G_h^{\varepsilon_0}(t,x,a)|\times 1_{t\geq a/h^{1/3}}\lesssim \frac{1}{h^{2/3+5/114}}$ (as the bounds \eqref{dispT>} are obtained from the bounds on the Green function for $t\geq a/h^{1/3}$, that is for $T\geq \lambda^{1/3}$).  In the following we focus on the contribution for $t<a/h^{1/3}$ and we prove the following result, which is better than announced (and shows that in this regime the Strichartz estimates are sharp) :
\begin{equation}\label{StriT<}
\Big\|\sup_{a,x\leq a} |G_h^{\varepsilon_0}(\cdot,x,a)|\times 1_{t<a/h^{13}}\Big \|_{L^2(0,t_0)}\lesssim 1/h^{2/3+\epsilon}, \forall \epsilon>0.
\end{equation}
In the following we prove \eqref{StriT<}. Write $G_h^{\varepsilon_0}(t,x,a)=G_{h,a}(t,x,a)+\sum_{\gamma\in \Gamma_0(a)}G_{h,\gamma}(t,x,a)$. For $\gamma>4a$, with $G_{h,\gamma}$ as in \eqref{greenfctbis}, it has been proved in \eqref{estimtransvN>1} that $\sum_{\gamma\in \Gamma_0(a)} \|G_{h,\gamma}(t,\cdot)\|_{L^{\infty}(x\leq a)}\lesssim \frac 1h h^{1/3}\log(\varepsilon_0/a)$, hence the same bound will hold for the $L^2$ norm in time, so $\sum_{\gamma\in \Gamma_0(a)}G_{h,\gamma}(t,x,a)$ satisfies \eqref{StriT<}. We now focus on $G_{h,a}$ for $t<a/h^{1/3}$.
The "swallow type singularities", which provide $1/4$ loss, appear in the wavefront only at $x=a$ for $T=t/\sqrt{a}\in 2\mathbb{N}$, $T<\lambda^{1/3}$, hence affecting only $G_{h,a}$ with an effect on intervals of time of the form $I_N:=(2N-1/N,2N+1/N)$. Outside these intervals $I_N$ there are only cusps singularities in the wavefront which yield \eqref{T<dist>}, hence the contribution of $G_{h,a}$ outside $\cup I_N$ also satisfies \eqref{StriT<}. 
Using Proposition \ref{proptangmain}, we may decompose $G_{h,a}(t,x,a)$ into two parts, one part, denoted $G_{sing,h,a}(t,x,a):=G_{h,a}(t,x,a)\times 1_{\cup I_N}(t)$, localised for $T=t/\sqrt{a}$ in small neighborhoods of size $1/N$ of $2N$ with $N\in \mathbb{N}$  and another one, denoted $(G_{h,a}-G_{sing,h,a})(t,x,a)$ localized for $T$ outside the reunion of $1/N$ - neighborhoods of $2N$. 
From (the proof of) Proposition \ref{proptangmain}, it follows that $|G_{h,a}-G_{sing,h,a}|(t,x,a)\lesssim 1/h^{2/3}$ hence its $L^2$ norm satisfies \eqref{StriT<}, so we are left with $G_{sing,h,a}(t,x,a)$, for which we need to carefully compute the $L^2$ norm using Proposition \ref{dispNpetitpres}. In the following we prove that 
\[
\Big\|\sup_{a,x\leq a}|G_{sing,h,a}(\cdot,x,a)|\Big\|_{L^2(0,t_0)}\lesssim \sqrt{\ln(1/h)}/h^{2/3},
\] 
which will achieve the proof of \eqref{estnormA} and hence of Theorem \ref{secT<}. Using \eqref{uhsumNa}, 
\[
G_{sing,h,a}(t,x,a)=\frac 1h \Big(\sum_{\tilde N\sim (t/\sqrt{a})}V_{\tilde N,h,a}\Big)\times 1_{(t/\sqrt{a})\in \cup_{ N} I_{ N}}(t)\times 1_{t<a/h^{1/3}},
\]
 with $V_{N,h,a}$ defined in \eqref{defVNha}. The intervals $I_N$ are disjoint, and for a fixed $N$ only one wave packet in the sum $\sum_{\tilde N\sim T}V_{N,h,a}$ provides non-trivial contribution, the one corresponding to $\tilde N=N$.
As $V_{N,h,a}$ satisfy \eqref{eq:2hh}, it will be enough to prove that $\sum_{N\sim t/\sqrt{a}<\lambda^{1/3}} \|V_{N,h,a}(\cdot,x)\|^2_{L^2(I_N)}\lesssim \ln(1/h)$ for all $\epsilon>0$. Using \eqref{eq:2hh}, we have
\begin{multline}
\sum_{N\sim t/\sqrt{a}<\lambda^{1/3}} \|V_{N,h,a}(\cdot,x)\|^2_{L^2(I_N)}\leq \sum_{N\sim t/\sqrt{a}<\lambda^{1/3}}\int_{t/\sqrt{a}\in I_N}\frac{1}{\Big[(N/\lambda^{1/3})^{1/4}+N^{1/3}|\sqrt{\frac{t}{2N\sqrt{a}}}-1|^{1/6}\Big]^2}dt\\
{\small (\frac{t}{2N\sqrt{a}}}=1+\frac{2w}{N^2})\quad \lesssim \sum_{N\sim t/\sqrt{a}<\lambda^{1/3}}\int_{-1}^1\frac{(\sqrt{a}/N)}{\Big[(N/\lambda^{1/3})^{1/4}+N^{1/3}|\frac{w}{N^2}|^{1/6}\Big]^2}dw\\
\sim 2\sum_{N\sim t/\sqrt{a}<\lambda^{1/3}}\frac{\sqrt{a}}{N} \int_0^1\frac{1}{(w^{1/6}+(N/\lambda^{1/3})^{1/4})^2}dw\lesssim \sum_{N\sim t/\sqrt{a}<\lambda^{1/3}}\frac{\sqrt{a}}{N} \sim \sqrt{a}\ln(\lambda^{1/3})\lesssim \ln(1/h),
\end{multline}
where in the last line we set $w=x^6$ and used that $N/\lambda^{1/3}<1$ to obtain
\[
\int_0^1\frac{1}{(w^{1/6}+(N/\lambda^{1/3})^{1/4})^2}dw\lesssim \int_0^{(N/\lambda^{1/3})^{1/4}}\frac{6x^5}{(N/\lambda^{1/3})^{1/2}}dx+\int_{(N/\lambda^{1/3})^{1/4}}^1\frac{6x^5}{x^2}dx\lesssim (N/\lambda^{1/3})^{3/2-1/2}+1\sim 1.
\]

\section{Exponential sums and Van der Corput type estimates - proof of Theorem \ref{thm>}}\label{secT>}
Recall from \eqref{greenfctbiseps0} that
\begin{equation} 
G^{\varepsilon_0}_h(t,x,a) : =  \sum_{k\geq 1}
e^{iht\lambda_k}
 \phi_{\varepsilon_0}(\omega_{k}h^{2/3})
 e_k(x)e_k(a)\,,
\end{equation}
where $k\leq \varepsilon_0/h$ on the support of $ \phi_{\varepsilon_0}(\omega_{k}h^{2/3})$.
Recall that $\lambda=a^{3/2}/h$ and write 
\[
ht\lambda_k=ht \omega_k h^{-4/3}=(t/\sqrt{a})\times (\sqrt{a}/h^{1/3})\times \omega_k=T\times\lambda^{1/3}\times\omega_k=T\lambda (\omega_k/\lambda^{2/3}).
\]
Recall also that $\omega_k=F(\frac{3\pi}{8}(4k-1))$ (see \cite[(2.52), (2.64)]{AFbook}, where 
$F(y)\sim_{1/y^2}y^{2/3}\Big(1+O(1/y^2)\Big)$.
In the notations $\lambda, T,X$, we may write $G^{\varepsilon_0}_h(t,x,a) $ as follows
\begin{equation} 
G^{\varepsilon_0}_h(t,x,a) = \frac{2\pi}{h^{2/3}}\times \sum_{1\leq k\leq \varepsilon_0/h}
e^{iT\lambda (\omega_k/\lambda^{2/3})}\frac{1}{L'(\omega_k)} Ai(X\lambda^{2/3}-\omega_k)Ai(\lambda^{2/3}-\omega_k)\,,
\end{equation}
whose main contribution, denoted $G_{h,a}$ and dealt with in Propositions \ref{proptangmain} and \ref{propsumNgrand}, corresponds to values $h^{2/3}\omega_k\sim a$ (that is $\omega_k\sim \lambda^{2/3}$ or $k\sim \lambda$), where the variable in the factor $Ai(\lambda^{2/3}-\omega_k)$ may be very small. Hence, we focus on
\begin{equation}\label{Ghaexpsum}
G_{h,a}(t,x) = \frac{2\pi}{h^{2/3}}\times  \sum_{ k\sim \lambda}
e^{iT\lambda (\omega_k/\lambda^{2/3})}\frac{1}{L'(\omega_k)} Ai(X\lambda^{2/3}-\omega_k)Ai(\lambda^{2/3}-\omega_k)\,.
\end{equation}
As a summary, in Section \ref{sectdisp} we have obtained the following bounds (with the new notations) :
\begin{prop}\label{propresume} 
If $\frac{a^{3/2}}{h}=\lambda\gg 1$ and $\frac{t}{\sqrt{a}}=T\geq 1$ and $\frac xa = X\leq 1$, we have
\begin{itemize}
\item For $T< \lambda^{1/3}$, 
and $T\in 2\N$, Proposition \ref{proptangmain} and Lemma \ref{lemop} yield (sharp bounds)
\begin{equation}
\|G_{h,a}(t,\cdot)\|_{L^{\infty}(x>0)}\sim \|G^{\varepsilon_0}_h(t,\cdot,a)\|_{L^{\infty}(x>0)}\sim \frac{C}{h}\Big(\frac{ha}{t}\Big)^{1/4}=\frac{C}{h^{2/3}}\times \frac{a^{1/8}}{h^{1/12}}\times \frac{1}{T^{1/4}},
\end{equation}
which can be rewritten as
\begin{equation}\label{est1}
 \sum_{ k\sim \lambda}
e^{iT\lambda (\omega_k/\lambda^{2/3})}\frac{1}{L'(\omega_k)} Ai^2(\lambda^{2/3}-\omega_k)\sim \Big(\frac{\lambda^{1/3}}{T}\Big)^{1/4}\,.
\end{equation}
\item For $T< \lambda^{1/3}$ such that $|\frac{T}{2N}-1|\gtrsim 1/T$,
\begin{equation}\label{est12}
 \sum_{ k\sim \lambda}
e^{iT\lambda (\omega_k/\lambda^{2/3})}\frac{1}{L'(\omega_k)} Ai^2(\lambda^{2/3}-\omega_k)\lesssim 1\,.
\end{equation}

\item For $\lambda^{1/3}\leq T$, Proposition \ref{propsumNgrand} yields
\begin{equation}
\|G_{h,a}(t,\cdot)\|_{L^{\infty}(x>0)}\sim \|G^{\varepsilon_0}_h(t,\cdot,a)\|_{L^{\infty}(x>0)}\sim \frac{C}{h}\Big(\frac{ht}{a}\Big)^{1/2}=\frac{C}{h^{2/3}}\times  \frac{h^{1/6}}{a^{1/4}}\times T^{1/2},
\end{equation}
hence
\begin{equation}\label{est2}
 \sum_{ k\sim \lambda}
e^{iT\lambda (\omega_k/\lambda^{2/3})}\frac{1}{L'(\omega_k)} Ai(X\lambda^{2/3}-\omega_k)Ai(\lambda^{2/3}-\omega_k)\lesssim \Big(\frac{T}{\lambda^{1/3}}\Big)^{1/2}\,.
\end{equation}

\item For all $T>\lambda$, Proposition \ref{prop11} and Lemma \ref{lemsob} yield

\begin{equation}\label{est3a}
 \sum_{ k\sim \lambda}
e^{iT\lambda (\omega_k/\lambda^{2/3})}\frac{1}{L'(\omega_k)} Ai(X\lambda^{2/3}-\omega_k)Ai(\lambda^{2/3}-\omega_k)\lesssim \lambda^{1/3}\,.
\end{equation}
\end{itemize}
\end{prop}
\begin{rmq}
Notice that for $T< \lambda^{1/3}$, the estimates \eqref{est1} and \eqref{est12} are sharp for dispersion. Integration in time yields (sharp) Stricharz with $1/6$ loss (as if one had applied $TT^*$ to \eqref{est12} only), as the intermittent moments of time $T\in 2\mathbb{N}$ near which \eqref{est1} holds become harmless when integrating over time.

For $\lambda^{1/3}\leq T$, the estimates \eqref{est2} may be useful as long as $T\ll \lambda$; however, when $T\sim \lambda$ the bound $\lambda^{1/3}$ yields $1/4$ loss is dispersion and Strichartz and need to be improved to prove better bounds.
In the regime $T\geq \lambda\gg 1$, the estimates \eqref{est3a} are obtained from the Sobolev type bounds in Lemma \ref{lemsob} (which, in particular, do not make use of the possible cancellations due to the exponential factors $e^{T\lambda (\omega_k/\lambda^{2/3})}$), they are (very) far from sharp. In particular, for $t\sim \lambda$ both \eqref{est2} and \eqref{est3a} provide a loss of $1/4$ in the dispersive and Strichartz bounds as 
\[
T=\frac{t}{\sqrt{a}}\sim\frac{a^{3/2}}{h}=\lambda\quad \Leftrightarrow a\sim (ht)^{1/2}, \quad \lambda^{1/3}=\frac{a^{1/2}}{h^{1/3}}\sim t^{1/4} h^{\frac14-\frac13}=t^{1/4} h^{-1/12},
\]
and the "loss" in dispersion equals $\frac16+\frac{1}{12}=\frac{1}{4}$ (where $1/6$ comes from the factor $\frac{1}{h^{2/3}}=\frac{1}{h}(h/t)^{1/2}\times t^{1/2} h^{-1/6}$).
These bounds from Proposition \ref{propresume} are sufficient to obtain dispersive estimates with $1/4$ loss for the semi-classical Schr\"odinger equation in dimension $d\geq 2$ in \cite{Iva23}. 
We aim at improving them using Van der Corput derivative test.
\end{rmq}
Let $h^{2/3}\ll a\leq (ht)^{1/2}$ and consider the sum from \eqref{Ghaexpsum}
\begin{equation}\label{formEla}
E_{\lambda}(T,X):= \sum_{ k\sim \lambda}
e^{iT\lambda (\omega_k/\lambda^{2/3})}\frac{1}{L'(\omega_k)} Ai(X\lambda^{2/3}-\omega_k)Ai(\lambda^{2/3}-\omega_k).
\end{equation}
The goal of this section is to prove the following results, which will achieve the proof of Theorem \ref{thm>}:
\begin{prop}\label{prop>}
Let $\lambda^{1/3}\leq T:=\frac{t}{2\sqrt{a}}$, then the Van der Corput's $j$-th derivative test estimates yield, for $j=2,3,4$, respectively,
\begin{equation}
    \label{eq:40}
    \|G_{h,a}(t,\cdot)\|_{L^{\infty}(x\leq a)} \lesssim \frac{1}{h^{2/3}}\times 
   \left\{ \begin{array}{l} 
(\frac{T}{\lambda^{1/3}})^{1/2}
\text{ if } \lambda^{1/3}\leq T\leq \lambda^{1/2},\\\\
 T^{1/6}, \text{ if } \lambda^{1/2}\leq T<\lambda^{5/4},\\\\
 \lambda^{5/42} T^{1/14}  \text{ if } \lambda^{5/4}\leq T<\lambda^{3},\\\\
 \lambda^{1/3}, \text{ if } T\geq \lambda^3.\\
 \end{array} \right.
  \end{equation}
  \end{prop} 
 \begin{cor}\label{corS}
As a consequence of Proposition \ref{prop>}, the corresponding "loss" in the dispersive and Strichartz bounds, compared to the bound $ \frac{1}{h} (\frac ht)^{1/2}$ of the flat case, equals $1/6+(20/19)*(1/24)$ and, depending on $T$, it equals
\begin{equation}
    \label{eq:loss}
   \left\{ \begin{array}{l} 
\text{ if } \lambda^{1/3}\leq T\leq \lambda^{1/2},\text{ the loss is }(\frac ht)^{-1/2} h^{1/3} (\frac{T}{\lambda^{1/3}})^{1/2} \leq t^{1/2}h^{-1/6} \lambda^{1/12}\leq h^{-(1/6+1/30)}
\\\\
\text{ if } \lambda^{1/2}\leq T<\lambda^{5/4}, \text{ the loss is } (\frac ht)^{-1/2} h^{1/3} T^{1/6}\leq t^{1/2+1/6} h^{-1/6}\lambda^{5/24}\leq h^{-(1/6+(5/6)*(1/19))} ,\\\\
 \text{ if } \lambda^{5/4}\leq T< \lambda^{3}, \text{ the loss is } (\frac ht)^{-1/2} h^{1/3} \lambda^{5/42} T^{1/14}\leq t^{1/2+1/14} h^{-1/6}\lambda^{5/42+3/14}\leq h^{-(1/6+1/30)},\\\\
\text{if } T>\lambda^{3}, \text{ the loss is } (\frac ht)^{-1/2} h^{1/3} \lambda^{1/3}\leq t^{1/2+4/11} h^{-(1/6+1/30)}. \\
 \end{array} \right.
  \end{equation}
\end{cor}
\begin{proof}
We prove the Corollary using \eqref{eq:40}. For every regime, the worst bound occurs when $T$ is maximum, hence
\begin{itemize}
\item for $T=\lambda^{1/2}$, $t/\sqrt{a}\sim (a^{3/2}/h)^{1/2}$ we have $a^{5/4}\sim th^{1/2}$ so $a\sim t^{4/5}h^{2/5}$, hence  $\lambda\sim t^{(4/5)*(3/2)} h^{(2/5)*(3/2)}/h\leq h^{-2/5}$, which further yields $\lambda^{1/12}\leq h^{-1/30}$ as $t\lesssim 1$.
\item for $T=\lambda^{5/4}$, $t/\sqrt{a}\sim (a^{3/2}/h)^{5/4}$ we have $a\sim t^{8/19} h^{10/19}$ which yields $\lambda\sim (t^{8/19} h^{10/19})^{3/2}/h\leq h^{-4/19}$, hence
\[
T^{1/6}=\lambda^{(5/4)*(1/6)}\leq h^{-(5/24)*(4/19)}\sim h^{-5/114}.
\]
Notice that $5/114>1/30$, hence the worst loss for $\lambda^{1/3}\leq T\leq \lambda^{5/4}$ is $1/6+5/114$, $114=6*19$.

\item for $T=\lambda^{3}$, $t/\sqrt{a}\sim (a^{3/2}/h)^{3}$, hence $a\sim t^{1/5}h^{3/5}$ and $\lambda\sim (t^{1/5}h^{3/5})^{3/2}/h = t^{3/10} h^{9/10-1}\leq h^{-1/10}$, hence $\lambda^{5/42}T^{1/14}\sim \lambda^{5/42+3/14}=\lambda^{1/3}\leq h^{-1/30}$.  For $T\gg \lambda^3$, then $a\ll h^{3/5}$ and in the same way $\lambda^{1/3}\ll h^{-1/30}$.
\end{itemize}
\end{proof}
In the remaining of this section we prove Proposition \ref{prop>}. 
To obtain \eqref{eq:40} we apply Van der Corput estimates whenever this is possible in order to improve the bounds for $T\geq \lambda^{1/3}$. Notice that, if $\frac{3\pi}{2} k<\lambda-\frac 32 \lambda^{1/3}$, then $\lambda^{2/3}-\omega_k>1$ and therefore the factor $Ai(\lambda^{2/3}-\omega_k)$ decays exponentially. Indeed, in this case we have 
\begin{equation}\label{omega_k}
\omega_k=(\frac{3\pi}{2} k)^{2/3}(1+O(\frac 1k))<(\lambda-\frac 32 \lambda^{1/3})^{2/3}(1+O(\frac{1}{\lambda}))=\lambda^{2/3}\Big(1-\frac 32\frac 23\frac{1}{\lambda^{2/3}}+O(\frac{1}{\lambda^{4/3}})\Big)=\lambda^{2/3}-1+O(\frac{1}{\lambda^{2/3}}).
\end{equation}
Therefore in the sum defining $E_{\lambda}$ we only need to consider values $k$ such that $\frac{3\pi}{2}k=\lambda +l$, where $-\lambda^{1/3}\lesssim l\lesssim \lambda$. We will deal separately with the sum over $-\lambda^{1/3}\lesssim l \lesssim \lambda^{1/3}$, when the Airy factors do not oscillate, and the sum over $\lambda^{1/3}\lesssim l\lesssim \lambda$, which represents the main contribution of $E_{\lambda}(T,X)$.
Write
\begin{multline}
E_{\lambda}(T,X)=\sum_{\frac{3\pi}{2}k=\lambda +l, |l|\lesssim \lambda^{1/3}}e^{iT\lambda (\omega_k/\lambda^{2/3})}\frac{1}{L'(\omega_k)} Ai(X\lambda^{2/3}-\omega_k)Ai(\lambda^{2/3}-\omega_k)\\
+\sum_{\frac{3\pi}{2}k=\lambda+l, \lambda^{1/3}\leq l \lesssim \lambda}e^{iT\lambda (\omega_k/\lambda^{2/3})}\frac{1}{L'(\omega_k)} Ai(X\lambda^{2/3}-\omega_k)Ai(\lambda^{2/3}-\omega_k)+O(1),
\end{multline}
where the term $O(1)$ comes from the sum over $l\leq -\lambda^{1/3}$. We let $X=1$ for convenience : exactly the same method applies for all $0<X\leq 1$ (and can provide even better bounds) but taking $X=1$ allows to simplify the computations (and is the worst situation as $|\lambda^{2/3}-\omega_k|\leq 1$). 
The sum over $|l|\lesssim \lambda^{1/3}$ may be estimate as follows 
\begin{equation}\label{sumsmalll}
\Big|\sum_{\frac{3\pi}{2}k=\lambda+l, |l|\lesssim \lambda^{1/3}}e^{iT\lambda (\omega_k/\lambda^{2/3})}\frac{1}{L'(\omega_k)} Ai^2(\lambda^{2/3}-\omega_k)\Big|=O(1),
\end{equation}
where we have used the fact that $L'(\omega_k)\sim \sqrt{\omega_k} \sim \lambda^{1/3}$, $Ai^2(\lambda^{2/3}-\omega_k)\lesssim 1$ and that there are $\lambda^{1/3}$ terms in the sum. We are left with the sum over $l\geq \lambda^{1/3}$. Write
\[
E_{\lambda}(T,1)=\sum_{\frac{3\pi}{2}k=\lambda+l, \lambda^{1/3}\leq l \lesssim \lambda}e^{iT\lambda (\omega_k/\lambda^{2/3})}\frac{1}{L'(\omega_k)} Ai^2(\lambda^{2/3}-\omega_k)+O(1).
\]
Since as soon as $\frac{3\pi}{2}k-\lambda\gtrsim \lambda^{1/3}$ the Airy factor start to oscillate, we decompose it as follows $Ai^2(-z)=A_+^2(z)+2A_+(z)A_-(z)+A-^2(z)$ where $A_{\pm}$ defined in \eqref{eq:Apm} are conjugate and of the form \eqref{eq:ApmAE}. 
We obtain from \eqref{eq:ApmAE}
\begin{multline}
E_{\lambda}(T,1)=\sum_{\varepsilon\in \{\pm 1\}} \sum_{\frac{3\pi}{2}k=\lambda+l, \lambda^{1/3}\leq l \lesssim \lambda}e^{iT\lambda (\omega_k/\lambda^{2/3})} e^{\varepsilon \frac 43 i(\omega_k-\lambda^{2/3})^{3/2}}\frac{\Psi^2(e^{\varepsilon i\pi/3}(\omega_k-\lambda^{2/3}))}{L'(\omega_k)} \\
+\sum_{\frac{3\pi}{2}k=\lambda+l, \lambda^{1/3}\leq l \lesssim \lambda}e^{iT\lambda (\omega_k/\lambda^{2/3})}\frac{\Psi(e^{i\pi/3}(\omega_k-\lambda^{2/3}))\Psi(e^{-i\pi/3}(\omega_k-\lambda^{2/3}))}{L'(\omega_k)} +O(1).
\end{multline}
We let (for $\tau:=\lambda T$, in the notations of section \ref{secExpSums})
\begin{equation}\label{deff2pe3}
f_{\tau}(l):=\tau\Big(\frac{\lambda+l}{\lambda}\Big)^{2/3},\quad f^{\varepsilon}_{\tau}(l):=f_{\tau}(l)+\varepsilon \frac 43 \Big((\lambda+l)^{2/3}-\lambda^{2/3}\Big)^{3/2}.
\end{equation}
As $\omega_k-\lambda^{2/3}\sim_{1/\lambda} (\lambda+l)^{2/3}(1+O(\frac{1}{\lambda}))-\lambda^{2/3}=\lambda^{2/3}\Big(1+\frac 23 \frac{l}{\lambda}+O(\frac{l^2}{\lambda^2})-1\Big)\sim \frac{l}{\lambda^{1/3}}(1+O(\frac{l}{\lambda}))$, then 
\begin{multline}\label{Elambda}
E_{\lambda}(T,1)=\sum_{\varepsilon\in \{\pm 1\}} \sum_{\frac{3\pi}{2}k=\lambda+l, \lambda^{1/3}\leq l \lesssim \lambda}e^{if^{\varepsilon}_{T\lambda}(l)} \frac{\Psi^2(e^{\varepsilon i\pi/3}(l/\lambda^{1/3})(1+O(\frac{l}{\lambda})))}{L'(\omega_k)} \\
+\sum_{\frac{3\pi}{2}k=\lambda+l, \lambda^{1/3}\leq l\lesssim \lambda}e^{if_{T\lambda}(l)}\frac{\Psi(e^{i\pi/3}(l/\lambda^{1/3})(1+O(\frac{l}{\lambda})))\Psi(e^{-i\pi/3}(l/\lambda^{1/3})(1+O(\frac{l}{\lambda})))}{L'(\omega_k)}+O(1).
\end{multline}
We recall from \eqref{eq:ApmAE} that 
\begin{equation}\label{estimPsi}
\Psi(e^{\varepsilon i\pi/3}(l/\lambda^{1/3})(1+O(\frac{l}{\lambda})))=\frac{e^{-\varepsilon i\pi/12}}{4\pi^{3/2}} (\lambda^{1/3}/l)^{1/4}\Big(1+O((\lambda^{1/3}/l)^{3/2})\Big).
\end{equation}
To estimate $E_{\lambda}(T,1)$ using Van der Corput's $j$-th derivative test, we need to understand the behaviour of the higher order derivatives of the phase functions $f^{\varepsilon}_{T\lambda}(l)$ for $\varepsilon\in \{0,\pm \}$. As $f_{T\lambda}$ is of the form \eqref{fctalpha} with $\alpha=2/3$ and $\tau=T\lambda$, $T\geq 1$, we compute the higher order derivatives $\partial^j(f^{\varepsilon}_{T\lambda}(l))$ for $\varepsilon\in \{\pm \}$ and $j\geq 2$ in the next Lemmas :
\begin{lemma}\label{lemfkap1} For all $1\leq M\leq \lambda $ and $l\in [1,M]$ explicit computations give
\begin{equation}\label{deriv023}
|\partial^2(f_{\tau}(l))|= \frac{\tau}{\lambda^2}\frac 29 \Big(1+\frac{l}{\lambda}\Big)^{-4/3},  |\partial^3(f_{\tau}(l))|=\frac{\tau}{\lambda^3}\frac{8}{27}\Big(1+\frac{l}{\lambda}\Big)^{-7/3},  |\partial^4(f_{\tau}(l))|= \frac{\tau}{\lambda^4} \frac{56}{81}\Big(1+\frac{l}{\lambda}\Big)^{-10/3}.
\end{equation}
\end{lemma}
Next, we study the derivatives of $f^{\varepsilon}_{\tau}(l)-f_{\tau}(l)=\varepsilon \frac 43 \lambda \Big((1+\frac{l}{\lambda})^{2/3}-1\Big)^{3/2}$ for $\varepsilon\in \{\pm\}$.
\begin{lemma}\label{lemfkap2}
For all $1\leq M\leq \lambda $ and $l\in [1,M]$, $\varepsilon\in \{\pm\}$ we have
\begin{align}\label{deriveps23}
|\partial^2(f^{\varepsilon}_{\tau}(l)-f_{\tau}(l))|&=\frac{4}{9\lambda}  \Big((1+\frac{l}{\lambda})^{2/3}-1\Big)^{-1/2}\Big(1+\frac{l}{\lambda}\Big)^{-4/3}\sim \frac{1}{\sqrt{l\lambda}}, \\
|\partial^3(f^{\varepsilon}_{\tau}(l)-f^{0}_{\tau}(l))|&=- \frac{4}{27\lambda^2} \Big((1+\frac{l}{\lambda})^{2/3}-1\Big)^{-3/2} \Big(5(1+\frac{l}{\lambda})^{2/3}-4\Big)(1+\frac{l}{\lambda})^{-7/3}\sim \frac{1}{l^{3/2}\lambda},\\
|\partial^4(f^{\varepsilon}_{\tau}(l)-f_{\tau}(l))|&\sim  \frac{1}{l^{5/2}\lambda}.
\end{align}

\end{lemma}
\begin{proof}
We have $\partial(\frac 43 \lambda \Big((1+\frac{l}{\lambda})^{2/3}-1\Big)^{3/2})=\frac{4}{3}  \Big((1+\frac{l}{\lambda})^{2/3}-1\Big)^{1/2}(1+\frac{l}{\lambda})^{-1/3}$, then the first line in \eqref{deriveps23} holds. Explicit computations allow to obtain the third and fourth order derivatives.
\end{proof}
Using the last two lemmas, in particular \eqref{deriv023} and \eqref{deriveps23}, we obtain the following result for $j\in \{2,3,4\}$ :
\begin{lemma}\label{lemdelta}
The higher order derivatives of $f^{\varepsilon}_{\tau}$ behave as follows :
\begin{enumerate}
\item For $j=2$ and $\tau=T\lambda$, we have $|\partial^2(f^{\varepsilon}_{T\lambda}(l))|\sim |\partial^2(f_{T\lambda}(l))|\sim T/\lambda$ only for values $T\gtrsim (\lambda/l)^{1/2}$, i.e. for $l\gtrsim \lambda/T^2$. Notice that, for $T\geq \lambda^{1/3}$, this condition is always satisfied for all $l\geq \lambda^{1/3} $ as, in this case, $l\geq \lambda^{1/3}=\lambda/\lambda^{2/3}\gtrsim \lambda/T^2$. As a consequence, for $\lambda>T\geq \lambda^{1/3}$, we have $\delta_2:=T/\lambda\in (0,1)$ and  
\begin{equation}\label{delta2}
|\partial^2(f^{\varepsilon}_{T\lambda}(l))|\sim |\partial^2(f_{T\lambda}(l))|\sim \delta_2:=T/\lambda.\\
\end{equation}

\item For $j=3$ and $\tau=T\lambda$, we have $|\partial^3(f^{\varepsilon}_{T\lambda}(l))|\sim |\partial^3(f_{T\lambda}(l))|\sim T/\lambda^2:=\delta_3$ only for values $T\gtrsim (\lambda/l)^{3/2}$, i.e. for $l\gtrsim \lambda/T^{2/3}$.
In particular, for $T\geq \lambda$, this condition holds for all $l\geq \lambda^{1/3}$.\\
\item For $j=4$ and $\tau=T\lambda$, we have $|\partial^4(f^{\varepsilon}_{T\lambda}(l))|\sim |\partial^4(f_{T\lambda}(l))|\sim T/\lambda^3:=\delta_4$ only for values $T\gtrsim (\lambda/l)^{5/2}$, i.e. for $l\gtrsim \lambda/T^{2/5}$. In particular, for $T\geq \lambda^{5/3}$, this condition holds for all $l\geq \lambda^{1/3}$.
\end{enumerate}
\end{lemma}
\begin{rmq}
For higher order derivatives one has to take into account the coefficients depending on $j$ that may become large. However, it turns out that only the third and the fourth derivatives are necessary, hence we only consider $j\leq 4$. In particular, any improvement of the $4$-th derivative test \eqref{exp4} allow to improve the bounds in Proposition \ref{prop>} and hence in the Strichartz bounds. The result of \cite{RoSa} yields such an improvement.
\end{rmq}

In the following we will use Lemma \ref{lemdelta} together with the Abel summation in order to obtain better bounds for $E_{\lambda}(T,1)$ and hence for $\|G_{h,a}(t,\cdot)\|_{L^{\infty}(0\leq x\leq a)}$. We recall the Abel summation formulas :

\begin{equation}\label{Abel}
\sum_{l=l_1}^{l_2} \psi_l e_l=\psi_{l_2}(\sum_{l=l_1}^{l_2} e_l) -\sum_{l=l_1}^{l_2-1}(\psi_{l+1}-\psi_l)(\sum_{p=l_1}^{l} e_p)=\psi_{l_1}(\sum_{l=l_1}^{l_2} e_l) +\sum_{l=l_1}^{l_2-1}(\psi_{l+1}-\psi_l)(\sum_{p=l}^{l_2} e_p).
\end{equation}

For $\varepsilon\in \{\pm \}$, we deal separately with the sums that appear in the formula \eqref{Elambda} of $E_{\lambda}(T,1)$ and set 
\[
e^{\varepsilon}_l(\tau):= e^{if^{\varepsilon}_{\tau}(l)} \text{ and } e^0_l(\tau):= e^{if_{\tau}(l)} \text{ where} f^{\varepsilon}_{\tau} \text{ is defined in } \eqref{deff2pe3}, \quad \tau=T\lambda,
\]
\[
\psi^{\varepsilon}_l:=\frac{\Psi^2(e^{\varepsilon i\pi/3}(l/\lambda^{1/3})(1+O(\frac{l}{\lambda})))}{L'(\omega_{\lambda+l})} \text{ and } \psi^0_l:=\frac{\Psi(e^{i\pi/3}(l/\lambda^{1/3})(1+O(\frac{l}{\lambda})))\Psi(e^{- i\pi/3}(l/\lambda^{1/3})(1+O(\frac{l}{\lambda})))}{L'(\omega_{\lambda+l})}.
\]
\begin{lemma}\label{lempsie}
Using \eqref{estimPsi} and $L'(\omega_k)\sim \sqrt{\omega_k}\sim k^{1/3}\sim (\lambda+l)^{1/3}$,
we have, for all $\lambda^{1/3}\leq l\leq \lambda$ and $\varepsilon\in \{0,\pm\}$,
\[
\psi^{\varepsilon}_{l}\sim (\lambda^{1/3}/l)^{1/2} (\lambda+l)^{-1/3}\sim \lambda^{-1/6}/\sqrt{l}, \quad \psi^{\varepsilon}_{\lambda/2}\sim (\lambda^{1/3}/\lambda)^{1/2} \lambda^{-1/3}\sim \lambda^{-2/3}\quad \forall \varepsilon\in \{0,\pm \}.
\]
From \eqref{estimPsi} it also follows that for all $l\geq \lambda^{1/3}$
\[
|\psi^{\varepsilon}_{l+1}-\psi^{\varepsilon}_l|\lesssim \frac{\lambda^{1/6}}{\lambda^{1/3}}|\frac{1}{\sqrt{l}}-\frac{1}{\sqrt{l+1}}|\sim \lambda^{-1/6}\frac{1}{\sqrt{l}\sqrt{l+1}(\sqrt{l+1}+\sqrt{l})}\sim \frac{\lambda^{-1/6}}{l^{3/2}}.
\]
\end{lemma}

With these notations we may write, using \eqref{sumsmalll},
\begin{equation}\label{EsumA}
E_{\lambda}(T,1)=\sum_{\varepsilon\in \{0,\pm\}} \sum_{l=\lambda^{1/3}}^{\lambda/2} e^{\varepsilon}_l(T\lambda) \psi^{\varepsilon}_l +O(1).
\end{equation}
The first Abel formula in \eqref{Abel} applied to the sums in \eqref{EsumA} with $l_1\geq \lambda^{1/3}$, $l_2\leq \lambda/2$ yields, for every $\varepsilon\in \{0,\pm\}$, 
\begin{equation}\label{SEAi}
 \sum_{l=l_1}^{l_2} e^{\varepsilon}_l(T\lambda) \psi^{\varepsilon}_l= \psi^{\varepsilon}_{l_2}\Big(\sum_{l=l_1}^{l_2} e^{\varepsilon}_l(T\lambda)\Big)-  \sum_{l=l_1}^{l_2-1} (\psi^{\varepsilon}_{l+1}-\psi^{\varepsilon}_l) \Big(\sum_{p=l_1}^{l} e^{\varepsilon}_p(T\lambda)\Big).
 \end{equation}
Taking $l_1\geq \lambda^{1/3}$, $l_2\leq \lambda/2$ we obtain from \eqref{SEAi} and Lemma \ref{lempsie}
\begin{equation}\label{SEAib}
\Big| \sum_{l=l_1}^{l_2} e^{\varepsilon}_l(T\lambda) \psi^{\varepsilon}_l\Big|\lesssim \frac{\lambda^{-1/6}}{\sqrt{l_2}}\Big|\sum_{l=l_1}^{l_2} e^{\varepsilon}_l(T\lambda)\Big|+ \sum_{l=l_1}^{l_2-1}   \frac{\lambda^{-1/6}}{l^{3/2}} \Big|\sum_{p=l_1}^{l} e^{\varepsilon}_p(T\lambda)\Big|.
 \end{equation}

Depending on the size of $T\geq \lambda^{1/3}$ and the derivatives of $f^{\varepsilon}_{2/3,T\lambda}$, we estimate the sums in \eqref{EsumA} separately. 
\begin{equation}\label{bdsTlarge}
|E_{\lambda}(T,1)|\leq \min \{(T/\lambda^{1/3})^{1/2}, T^{1/4}\lambda^{1/12}+1\}.
\end{equation}

\begin{enumerate}
\item Let first $\lambda^{1/3}\leq T\lesssim \lambda^{1/2}$, in which case we show that the Van der Corput second derivative test provides the same estimates as Proposition \ref{propresume} for $E_{\lambda}(T,1)$. Notice that for $T\leq \lambda^{1/2}$ we have 
\[
\delta_2^{1/2}=(T/\lambda)^{1/2}\leq (T/\lambda^2)^{1/6}=\delta_3^{1/6},
\]
hence for $T\leq \lambda^{1/2}$ we only need to use Proposition \ref{propVdC2} (as the bounds provided there are the best ones for such $T$). Using \eqref{SEAib} with $l_1= \lambda^{1/3}$, $l_2= \lambda/2$, Lemma \ref{lempsie} and Proposition \ref{propVdC2} yields
\begin{multline}
|E_{\lambda}(T,1)|\lesssim \sum_{\varepsilon\in \{0,\pm\}}\Big(\lambda^{-2/3}\Big|\sum_{l=\lambda^{1/3}}^{\lambda/2} e^{\varepsilon}_l(T\lambda)\Big|+ \sum_{l=\lambda^{1/3}}^{\lambda/2-1}   \frac{\lambda^{-1/6}}{l^{3/2}} \Big|\sum_{p=\lambda^{1/3}}^{l} e^{\varepsilon}_p(T\lambda)\Big|\Big)\\
\lesssim \lambda^{-2/3}(\lambda \delta_2^{1/2})+ \sum_{l=\lambda^{1/3}}^{\delta_2^{-1}}   \frac{\lambda^{-1/6}}{l^{3/2}}\times \delta_2^{-1/2}+\sum_{l=\delta_2^{-1}}^{\lambda/2-1}   \frac{\lambda^{-1/6}}{l^{3/2}}\times l\delta_2^{1/2}\\
\sim \lambda^{1/3} \delta_2^{1/2}+\lambda^{-1/6}(\delta_2^{-1/2}\lambda^{-1/6}+\delta_2^{1/2}\lambda^{1/2})\\
\sim \lambda^{1/3}\Big(\frac{T}{\lambda}\Big)^{1/2}+ \lambda^{-1/6}((\lambda/T)^{1/2}\lambda^{-1/6}+T^{1/2})
\sim (T/\lambda^{1/3})^{1/2}.
\end{multline}
Although maybe not sharp, the bounds obtained by Van der Corput are the same as the ones obtained in Proposition \ref{propresume}. In the following we consider different regimes for $T\geq \lambda^{1/2}$ and improve upon \eqref{bdsTlarge}.

\item Let $\lambda^{1/2}\leq T\leq \lambda^{5/4}$ : in this case we prove that $|E_{\lambda}(T,1)|\lesssim T^{1/6}$. Notice that this regime corresponds to
\[
\delta_2^{1/2}=(T/\lambda)^{1/2}\geq (T/\lambda^2)^{1/6}=\delta_3^{1/6}, \quad \forall T\geq \lambda^{1/2},
\]
\[
\delta_3^{1/6}=(T/\lambda^2)^{1/6}\leq (T/\lambda^3)^{1/14}=\delta_4^{1/14},\quad \forall T\leq \lambda^{5/4}.
\]
For large $l$ we must use the third order derivatives of $f^{\varepsilon}_{T\lambda}(l)$, which, according to the Lemma \ref{lemdelta} with $j=3$, are comparable to $\delta_3=T/\lambda^2$ for $l\gtrsim \lambda/T^{2/3}$. 
We deal separately with the cases $T\leq \lambda$ and $T>\lambda$. 

\begin{lemma}
Let $\lambda^{1/2}\leq T<\lambda$, then $\lambda^{1/3}<\lambda/T^{2/3}$ and we have
\begin{equation}\label{est23T<>}
\Big| \sum_{l=\lambda^{1/3}}^{\lambda/2} e^{\varepsilon}_l(T\lambda) \psi^{\varepsilon}_l\Big| \leq \Big| \sum_{l=\lambda^{1/3}}^{\lambda^{1/3}+\lambda/T^{2/3}} e^{\varepsilon}_l(T\lambda) \psi^{\varepsilon}_l\Big|+\Big| \sum_{l=\lambda^{1/3}+\lambda/T^{2/3}}^{\lambda/2} e^{\varepsilon}_l(T\lambda) \psi^{\varepsilon}_l\Big| \lesssim T^{1/6}.
\end{equation}

\end{lemma}
\begin{proof}
When $T<\lambda$, the part corresponding to values $l\leq \lambda^{1/3}+\lambda/T^{2/3}\sim \lambda/T^{2/3}$ is dealt with using Proposition \ref{propVdC2} for all $\varepsilon\in \{0,\pm\}$, as for all such $l$ we have
$\Big|\sum_{p=\lambda^{1/3}}^{l} e^{\varepsilon}_p(T\lambda)\Big|\leq  l\delta_2^{1/2}+\delta_2^{-1/2}$, $\delta_2\in (0,1)$.
The first sum in \eqref{est23T<>} is therefore bounded as follows

\begin{multline}\label{SEAibiiivar<}
\Big| \sum_{l=\lambda^{1/3}}^{\lambda^{1/3}+\lambda/T^{2/3}} e^{\varepsilon}_l(T\lambda) \psi^{\varepsilon}_l\Big|\lesssim \psi^{\varepsilon}_{\lambda/T^{2/3}}\Big|\sum_{l=\lambda^{1/3}}^{\lambda^{1/3}+\lambda/T^{2/3}} e^{\varepsilon}_l(T\lambda)\Big|
+\sum_{l=\lambda^{1/3}}^{\lambda^{1/3}+\lambda/T^{2/3}-1}   \frac{\lambda^{-1/6}}{l^{3/2}} \Big|\sum_{p=\lambda^{1/3}}^{l} e^{\varepsilon}_p(T\lambda)\Big|\\
\lesssim \frac{\lambda^{-1/6}}{\sqrt{\lambda/T^{2/3}}} ((\lambda/T^{2/3})\delta_2^{1/2})
+\sum_{l=\lambda^{1/3}}^{\lambda^{1/3}+\lambda/T^{2/3}-1}   \frac{\lambda^{-1/6}}{l^{3/2}} (l\delta_2^{1/2}+\delta_2^{-1/2}).
\end{multline}
The first term in the last line equals $\lambda^{-1/6} (\lambda/T^{2/3})^{1/2}(t/\lambda)^{1/2}=(T/\lambda)^{1/6}$. In the second term, we have to separate the cases $l\leq \delta_2^{-1}=\lambda/T$ when $(l\delta_2^{1/2}+\delta_2^{-1/2})\sim \delta_2^{-1/2}$ and $l>\delta_2^{-1}$ when $(l\delta_2^{1/2}+\delta_2^{-1/2})\sim l\delta_2^{1/2}$.
As $l\geq \lambda^{1/3}$, the first situation can only occur for $T<\lambda^{2/3}$. For $\lambda^{1/2}\leq T< \lambda^{2/3}$, we therefore have :

\begin{multline}\label{SEAibiiivar<}
\Big| \sum_{l=\lambda^{1/3}}^{\lambda^{1/3}+\lambda/T^{2/3}} e^{\varepsilon}_l(T\lambda) \psi^{\varepsilon}_l\Big|\lesssim \psi^{\varepsilon}_{\lambda/T^{2/3}}\Big|\sum_{l=\lambda^{1/3}}^{\lambda^{1/3}+\lambda/T^{2/3}} e^{\varepsilon}_l(T\lambda)\Big|
+ \sum_{l=\lambda^{1/3}}^{\delta_2^{-1}-1}   \frac{\lambda^{-1/6}}{l^{3/2}} \Big|\sum_{p=\lambda^{1/3}}^{l} e^{\varepsilon}_p(T\lambda)\Big|
+\sum_{l=\delta_2^{-1}}^{\lambda^{1/3}+\lambda/T^{2/3}-1}   \frac{\lambda^{-1/6}}{l^{3/2}} \Big|\sum_{p=\lambda^{1/3}}^{l} e^{\varepsilon}_p(T\lambda)\Big|\\
\lesssim \frac{\lambda^{-1/6}}{\sqrt{\lambda/T^{2/3}}} ((\lambda/T^{2/3})\delta_2^{1/2})+\lambda^{-1/6}\Big(\sum_{l=\lambda^{1/3}}^{\lambda/T-1}   \frac{\delta_2^{-1/2}}{l^{3/2}} +  \sum_{l=\lambda/T}^{\lambda/T^{2/3}-1}  \frac{\delta_2^{1/2}}{l^{1/2}}  \Big)\\
\lesssim (T/\lambda)^{1/6}+\lambda^{-1/6}( (\lambda/T)^{1/2}\lambda^{-1/6}+(T/\lambda)^{1/2}(\lambda/T^{2/3})^{1/2})\\
=T^{1/6}/\lambda^{1/6}+\lambda^{1/6}/T^{1/2}\sim T^{1/6}/\lambda^{1/6}.
\end{multline}
In the last line we used $T\geq \lambda^{1/2}$ which implies $\lambda^{1/6}/T^{1/2}\leq (T/\lambda)^{1/6}$. For $\lambda^{2/3}\leq T<\lambda$, we have $\delta_2^{-1}=\lambda/T\leq \lambda^{1/3}\leq l$, hence $(l\delta_2^{1/2}+\delta_2^{-1/2})\sim l\delta_2^{1/2}$ and the sum in the last line of \eqref{SEAibiiivar<} is bounded by 
\[
\sum_{l=\lambda^{1/3}}^{\lambda^{1/3}+\lambda/T^{2/3}-1}   \frac{\lambda^{-1/6}}{l^{3/2}} (l\delta_2^{1/2}+\delta_2^{-1/2})\lesssim \lambda^{-1/6}\delta_2^{1/2}(\lambda/T^{2/3})^{1/2}=\lambda^{-1/6} (T/\lambda)^{1/2}(\lambda^{1/2}/T^{1/3})=(T/\lambda)^{1/6}.
\]
We are left with the sum over large $l\geq \lambda^{1/3}+ \lambda/T^{2/3}$ for which we use Proposition \ref{propbestk=3} with $M=l-(\lambda^{1/3}+\lambda/T^{2/3})$ which states that
$ \Big|\sum_{p=\lambda^{1/3}+\lambda/T^{2/3}}^{l} e^{\varepsilon}_p(T\lambda)\Big|\leq  l\delta_3^{1/6}+\delta_3^{-1/3}$.
Hence, for $l-(\lambda^{1/3}+\lambda/T^{2/3})\leq \delta_3^{-1/2}$ the bound is $\delta_3^{-1/3}$ while for $l-(\lambda^{1/3}+\lambda/T^{2/3})\geq \delta_3^{-1/2}$, the bound is $l \delta_3^{1/6}$. As $\delta_3^{-1/2}=(\lambda^2/T)^{1/2}=\lambda/\sqrt{T}\gg \lambda^{1/3}+\lambda/T^{2/3}$ (using $\lambda^{1/2}\leq T<\lambda$), we obtain
\begin{multline}\label{SEAibiiivar>}
\Big| \sum_{l=\lambda^{1/3}+\lambda/T^{2/3}}^{\lambda/2} e^{\varepsilon}_l(T\lambda) \psi^{\varepsilon}_l\Big|\lesssim \psi^{\varepsilon}_{\lambda/2}\Big|\sum_{l=\lambda^{1/3}+\lambda/T^{2/3}}^{\lambda/2} e^{\varepsilon}_l(T\lambda)\Big|\\
+ \sum_{l=\lambda^{1/3}+\lambda/T^{2/3}}^{\delta_3^{-1/2}} \frac{\lambda^{-1/6}}{l^{3/2}} \Big|\sum_{p=\lambda^{1/3}+\lambda/T^{2/3}}^{l} e^{\varepsilon}_p(T\lambda)\Big|+\sum_{l=\delta_3^{-1/2}}^{\lambda/2-1} \frac{\lambda^{-1/6}}{l^{3/2}} \Big|\sum_{p=\lambda^{1/3}+\lambda/T^{2/3}}^{l} e^{\varepsilon}_p(T\lambda)\Big|\\
\lesssim \lambda^{-2/3}\lambda (T/\lambda^2)^{1/6}+\lambda^{-1/6}\Big(\sum_{l=\lambda^{1/3}+\lambda/T^{2/3}}^{\delta_3^{-1/2}} \frac{\delta_3^{-1/3}}{l^{3/2}}+ \sum_{l=\delta_3^{-1/2}}^{\lambda/2-1} \frac{\delta_3^{1/6}}{l^{1/2}}\Big)\\
\lesssim T^{1/6}+\lambda^{-1/6}( (\lambda^2/T)^{1/3}(\lambda^{1/3}+\lambda/T^{2/3})^{-1/2}+(T/\lambda^2)^{1/6}\lambda^{1/2})\\
\sim T^{1/6}+\lambda^{-1/6}(\lambda^{1/6}+T^{1/6}\lambda^{1/6})\sim T^{1/6}.
\end{multline}
where we have used that $\lambda^{1/3}+\lambda/T^{2/3}\sim \lambda/T^{2/3}$ for $T<\lambda$. 
\end{proof}

\begin{lemma}
Let $\lambda\leq T\leq \lambda^{5/4}$, then $\lambda^{1/3}\geq \lambda/T^{2/3}$ and we have
\begin{equation}\label{est23T>}
\Big| \sum_{l=\lambda^{1/3}}^{\lambda/2} e^{\varepsilon}_l(T\lambda) \psi^{\varepsilon}_l\Big| \leq \psi^{\varepsilon}_{\lambda/2}\Big|\sum_{l=\lambda^{1/3}}^{\lambda/2} e^{\varepsilon}_l(T\lambda)\Big|+ \sum_{l=\lambda^{1/3}}^{\delta_3^{-1/2}} \frac{\lambda^{-1/6}}{l^{3/2}} \Big|\sum_{p=\lambda^{1/3}}^{l} e^{\varepsilon}_p(T\lambda)\Big|+\sum_{l=\delta_3^{-1/2}}^{\lambda/2-1} \frac{\lambda^{-1/6}}{l^{3/2}} \Big|\sum_{p=\lambda^{1/3}}^{l} e^{\varepsilon}_p(T\lambda)\Big|  \lesssim T^{1/6}.
\end{equation}

\end{lemma}
\begin{proof}
The first term in \eqref{est23T>} is bounded as before by $\lambda^{-2/3}\lambda (T/\lambda^2)^{1/6}\leq T^{1/6}$ and the last sum in \eqref{est23T>} is bounded as before by $\lambda^{-1/6}\delta_3^{1/6}\lambda^{1/2}\leq T^{1/6}$. The middle term in \eqref{est23T>} is bounded by $\lambda^{-1/6}\delta_3^{-1/3}(\lambda^{1/3})^{-1/2}=\lambda^{-1/6}(\lambda^2/T)^{1/3}\lambda^{-1/6}=(\lambda/T)^{1/3}\leq 1$ for $T\geq \lambda$. The proof is now complete.
\end{proof}

\item Let $\lambda^{5/4}\leq T\leq \lambda^3$ : in this case we prove, using \eqref{exp4} that $|E_{\lambda}(T,1)|\leq \lambda^{5/42}T^{1/14}$. Notice that this regime corresponds to 
\[
\delta_3^{1/6}=(T/\lambda^2)^{1/6}\geq (T/\lambda^3)^{1/14}=\delta_4^{1/14},\quad \forall T\geq \lambda^{5/4} 
\]
\[
\lambda^{-2/3}\times (\lambda \delta_4^{1/14})=\lambda^{1/3} (T/\lambda^3)^{1/14}\leq \lambda^{1/3}, \quad \forall T\leq \lambda^3.
\]
\begin{rmq}
Notice that we do not use Van der Corput's $j$-th derivative tests for $j\geq 5$, as this wouldn't bring any improvement. The main loss comes from the regime where only the third and fourth VdC's derivative tests matter, and it is only by improving the fourth derivative's test that we can hope to do better. Also, for large $T$, all VdC's estimates in Proposition \ref{propVdCk} end up by being larger than $\lambda^{1/3}$ (as they depend upon $T$). 
\end{rmq}
\begin{proof}
Again, we need to determine in which regimes of $l$ we can use the third and fourth derivatives of $f^{\varepsilon}_{T\lambda}(l)$ for $\varepsilon\in \{\pm 1\}$ (as we always have $|\partial^j(f_{T\lambda}(l))|\sim \delta_j=T/\lambda^{j-1}$). As noticed in Lemma \ref{lemdelta}, for $T\geq \lambda^{5/3}$, the condition $l\geq \lambda/T^{2/5}$, which is necessary for \eqref{exp4} to apply, holds for all $l\geq \lambda^{1/3}$.

\begin{enumerate}
\item Let first $\lambda^{5/3}\leq T\leq \lambda^{3}$. We apply \eqref{exp4} to bound $|\sum_{p=\lambda^{1/3}}^l e_p^{\varepsilon}(T\lambda)|\leq l \delta_4^{1/14}+l^{3/4}\delta_4^{-1/4}$, $\delta_4=T/\lambda^3<1$. After applying Abel summation with $l_1=\lambda^{1/3}$ and $l_2=\lambda/2$, the sum over $l$ should separated into two parts corresponding to $l\leq \delta_4^{-4/7}$, when $l \delta_4^{1/14}+l^{3/4}\delta_4^{-1/4}\sim l^{3/4}\delta_4^{-1/4}$, and $l\geq \delta_4^{-4/7}$, when $l^{3/4}\delta_4^{-1/4}\sim l\delta_4^{1/14}$. This is possible for $\lambda^{1/3}<\delta_4^{-4/7}=(\lambda^3/T)^{4/7}\leq \lambda$. i.e. for $T\leq \lambda^{29/12}$. For such $T$,
\begin{multline}\label{est4<>}
\Big| \sum_{l=\lambda^{1/3}}^{\lambda/2} e^{\varepsilon}_l(T\lambda) \psi^{\varepsilon}_l \Big|\leq \psi^{\varepsilon}_{\lambda/2}\Big|\sum_{l=\lambda^{1/3}}^{\lambda/2} e^{\varepsilon}_l(T\lambda)\Big|
+\sum_{l=\lambda^{1/3}}^{\delta_4^{-4/7}} \frac{\lambda^{-1/6}}{l^{3/2}} \Big|\sum_{p=\lambda^{1/3}}^{l} e^{\varepsilon}_p(T\lambda)\Big|+\sum_{l=\delta_4^{-4/7}}^{\lambda/2} \frac{\lambda^{-1/6}}{l^{3/2}} \Big|\sum_{p=\lambda^{1/3}}^{l} e^{\varepsilon}_p(T\lambda)\Big|\\
\lesssim \lambda^{-2/3}(\lambda \delta_4^{1/14})+\lambda^{-1/6}\Big(\sum_{l=\lambda^{1/3}}^{\delta_4^{-4/7}} \frac{\delta_4^{-1/14}}{l^{3/4}}+ \sum_{l=\delta_4^{-4/7}}^{\lambda/2} \frac{\delta_4^{1/14}}{l^{1/2}} \Big)\\
\lesssim \lambda^{1/3}(T/\lambda^3)^{1/14}+\lambda^{-1/6}\Big((\lambda^3/T)^{1/14+1/7}+(T/\lambda^3)^{1/14}\lambda^{1/2}\Big).
\end{multline}
The first term in the last line of \eqref{est4<>} equals $\lambda^{1/3-3/14}T^{1/14}=\lambda^{5/42}T^{1/14}$, and so does the last one. The middle term equals $\lambda^{-1/6+9/14}T^{-3/14}=\lambda^{10/21}T^{-3/14}$ and $\lambda^{10/21}T^{-3/14}\leq \lambda^{5/42}T^{1/14}$ for all $T\geq \lambda^{5/4}$. For $T>\lambda^{29/12}$, as $\delta_4^{-4/7}< \lambda^{1/3}$, we always have $|\sum_{p=\lambda^{1/3}}^l e_p^{\varepsilon}(T\lambda)|\leq l \delta_4^{1/14}$ and $|E_{\lambda}(T,1)|$ is bounded in the same way, where only the first and last contributions in the last line of \eqref{est4<>} do appear.

\item Let now $\lambda^{5/4}\leq T\leq \lambda^{5/3}$ then 
\begin{equation}\label{estTgrand}
 \sum_{l=\lambda^{1/3}}^{\lambda/2} e^{\varepsilon}_l(T\lambda) \psi^{\varepsilon}_l= \sum_{l=\lambda^{1/3}}^{\lambda/T^{2/5}} e^{\varepsilon}_l(T\lambda) \psi^{\varepsilon}_l+  \sum_{l=\lambda/T^{2/5}}^{\lambda/2} e^{\varepsilon}_l(T\lambda) \psi^{\varepsilon}_l.
\end{equation}
The last sum may be dealt with like before, and as $\lambda/T^{2/5}\ll (\lambda^3/T)^{4/7}=\delta_4^{-4/7}$ for all $T\leq \lambda^{5/3}(\ll \lambda^{25/6}$, for which the inequality holds), we proceed in exactly the same way as in \eqref{est4<>}, the only difference being that the sums start at $\lambda/T^{2/5}$ instead of $\lambda^{1/3}$. As this has no importance here (since the estimate \eqref{exp4} yields at least a factor $M^{3/4}$, so we do not add powers $<-1$ of $l$), the bounds are the same.\\

For the first sum in \eqref{estTgrand} we cannot use the forth order derivatives of $f^{\varepsilon}_{2/3,T\lambda}(l)$ as, for $\varepsilon\in \{\pm\}$, they are not of size $\delta_4$ (but much larger). The sum over $\lambda^{1/3}\leq l\leq \lambda/T^{2/5}$ must be split into two parts corresponding to $\lambda^{1/3}\leq l\leq \lambda^{1/3}+\delta_3^{-1/2}$ and $\lambda^{1/3}+\delta_3^{-1/2}\leq l\leq \lambda/T^{2/5}$ (notice that  we always have $\delta_3^{-1/2}\ll \lambda/T^{2/5}$; however, depending on whether $T\leq \lambda^{4/3}$ or $T>\lambda^{4/3}$, we may have $\delta_3^{-1/2}\geq \lambda^{1/3}$ or $\delta_3^{-1/2}< \lambda^{1/3}$).
When $ l\leq \lambda^{1/3}+\delta_3^{-1/2}$, the partial exponentials sums are uniformly bounded by $\delta_3^{-1/3}$ while in when $l> \lambda^{1/3}+\delta_3^{-1/2}$ these sums are bounded by $l\delta_3^{1/6}$. 
We therefore find  
\begin{multline}\label{estTgrandsmalls}
 \Big|\sum_{l=\lambda^{1/3}}^{\lambda/T^{2/5}} e^{\varepsilon}_l(T\lambda) \psi^{\varepsilon}_l\Big|\leq \psi^{\varepsilon}_{\lambda/T^{2/5}} \Big|\sum_{l=\lambda^{1/3}}^{\lambda/T^{2/5}} e^{\varepsilon}_l(T\lambda)\Big|
 +\sum_{l=\lambda^{1/3}}^{\lambda^{1/3}+\delta_3^{-1/2}} \frac{\lambda^{-1/6}}{l^{3/2}} \Big|\sum_{p=\lambda^{1/3}}^{l} e^{\varepsilon}_p(T\lambda)\Big|+\sum_{\lambda^{1/3}+\delta_3^{-1/2}}^{\lambda/T^{2/5}} \frac{\lambda^{-1/6}}{l^{3/2}} \Big|\sum_{p=\lambda^{1/3}}^{l} e^{\varepsilon}_p(T\lambda)\Big|\\
 \lesssim \lambda^{-2/3}T^{1/5} \delta_3^{1/6}(\lambda/T^{2/5})+\lambda^{-1/6}\Big(\delta_3^{-1/3}(\lambda^{1/3})^{-1/2}+\delta_3^{1/6}(\lambda/T^{2/5})^{1/2}\Big) \sim (\lambda/T)^{1/3}+T^{-1/30}\ll 1.
\end{multline}
\end{enumerate}

\end{proof}

\end{enumerate}

\begin{rmq} Notice that, using \eqref{VdC4} instead of \eqref{exp4} with $j=4$, one may improve upon these bounds when considering sums over large values of $l$ (such that $l\geq \delta_4^{-3/5}$). However, when the number of terms in these sums is not too large (corresponding to $l< \delta_4^{-3/5}$ or to values $l\leq \lambda/T^{5/2}$ for which the forth order derivatives of $f^{\varepsilon}_{T\lambda}$ are larger than $\delta_4$), estimating the corresponding sums (which may have less cancellations) can be more difficult. As \eqref{VdC4} would only provide an $\epsilon$ improvement with respect to \eqref{exp4}, we keep the computations simpler and use \eqref{exp4}. 

\end{rmq}

\section{Refined estimates for degenerate oscillatory integrals}\label{sectproofsprops}
In this section we prove Propositions \ref{dispNgrand}, \ref{dispNpetitloin}, \ref{dispNpetitpres} following closely \cite{Iva23}. As in the $1D$ case significant simplifications occur, and since these propositions are key to proving Theorem \ref{thm1D}, we include a detailed proof.

We are left with integrals with respect to the variables $s,\sigma$ to estimate $\|V_{N,h,a}(t,\cdot)\|_{L^{\infty}}$ with $V_{N,h,a}$ defined in \eqref{defVNha}. 
Using Remark \ref{rmqA}, we assume (without changing the contribution of $V_{N,h,a}$ modulo $O(h^{\infty})$) that its symbol $\varkappa$ is supported on $|(\sigma,s)|\leq 2\sqrt{\alpha_c}$. 
Fix $T=t/\sqrt{a}$ and let $N\in [\frac TM, MT]$ with $M>8$ and let $X=\frac x a\leq 1$ and $K=\sqrt{\frac{T}{2N}}$.

\subsection*{Proof of Proposition \ref{dispNgrand}} We start with the case where $\lambda^{1/3}\lesssim N$ and prove the following :
\begin{equation}
  \label{eq:90}
   \left| \int_{\R^{2}} e^{\frac ih \phi_{N,a}}  \varkappa(\sigma,s,t,x,h,a,1/N) \, ds d\sigma \right| \lesssim \frac{\lambda^{-2/3}}{1+\lambda^{1/3}|K^2-1|^{1/2}}\,,
\end{equation}
where $\phi_{N,a}(\sigma,s,\cdot):=\Phi_{N,a,\gamma=a}(\alpha_c,\sigma,s,\cdot)$. We rescale variables with $\sigma=\lambda^{-1/3}p$ and $s=\lambda^{-1/3}q$ and define
 $A=\lambda^{2/3}\Big(K^2-X\Big)$ and $B=\lambda^{2/3}\Big(K^2-1\Big)$.
We are reduced to proving that the following holds uniformly in $(A,B)$
\begin{equation}
  \label{eq:91}
   \left| \int_{\R^{2}} e^{iG_{N,a,\lambda}(p,q,t,x)}  \varkappa(\lambda^{-1/3} p,\lambda^{-1/3}q,t,x,h,a,1/N) \, dp dq \right| \lesssim \frac{1}{1+|B|^{1/2}}\,,
\end{equation}
where the rescaled phase is
$G_{N,a,\lambda}(p,q,t,x):=\frac{1}{h} \Big(\phi_{N,a}(\lambda^{-1/3}p,\lambda^{-1/3}q, t,x)-\phi_{N,a}(0,0,t,x)\Big)$.
As
\begin{equation}\label{derphisig1}
\partial_{\sigma}\Big(\Phi_{N,a,\gamma}(\alpha_c,s,\sigma,\cdot)\Big)=\gamma^{3/2}(\sigma^2+\frac{x}{\gamma}-\alpha_c),\quad
\partial_{s}\Big(\Phi_{N,a,\gamma}(\alpha_c,s,\sigma,\cdot)\Big)=\gamma^{3/2}(s^2+\frac{a}{\gamma}-\alpha_c),
\end{equation}
the first order derivatives of $\phi_{N,a}(\sigma,s,\cdot):=\Phi_{N,a,\gamma=a}(\alpha_c,\sigma,s,\cdot)$ are given by 
\begin{gather*}
  \partial_pG_{N,a,\lambda}=\frac 1h\frac{\partial\sigma}{\partial p}\partial_{\sigma}(\phi_{N,a})|_{(\sigma,s)=(\lambda^{-1/3}p,\lambda^{-1/3}q)}%
  =p^2-\lambda^{2/3}(\alpha_c-X)\,,
\\
\partial_qG_{N,a,\lambda}=\frac 1h \frac{\partial s}{\partial q}\partial_{s}(\phi_{N,a})|_{(\sigma,s)=(\lambda^{-1/3}p,\lambda^{-1/3}q)}%
=q^2-\lambda^{2/3}(\alpha_c-1)\,.
\end{gather*}
From \eqref{eq:alphac}, in our new variables, $\alpha_c$ has the following expansion
\begin{equation}\label{eq:alpha_c212}
\alpha_c|_{(\lambda^{-1/3}p,\lambda^{-1/3}q)}=
\Big(K-\lambda^{-1/3}\frac{p}{2N}-\lambda^{-1/3}\frac{q}{2N}\Big)^2.
\end{equation}
With these notations, we re-write the first order derivatives of $G_{N,a,\lambda}$,
\begin{equation}\label{Gfirstderiv}
\partial_{p}G_{N,a,\lambda}=p^2-A+\frac{\lambda^{1/3}}{N}K(p+q)-\frac{1}{4N^2}(p+q)^2,\quad 
\partial_{q}G_{N,a,\lambda}=q^2-B+\frac{\lambda^{1/3}}{N}K(p+q)-\frac{1}{4N^2}(p+q)^2.
\end{equation}
As $\lambda^{1/3}\leq N$, if $A,B$ are bounded, then \eqref{eq:91} obviously holds for $|(p,q)|$ bounded and by integration by parts if $|(p,q)|$ is large. We assume $|(A,B)|\geq r_{0}$ with $r_{0}\gg 1$. Set $(A,B)=r (\cos(\theta),\sin(\theta))$ and rescale again $(p,q)=r^{1/2} (\tilde p,\tilde q)$: we aim at
\begin{equation}
  \label{eq:93}
    \left| \int_{\R^{2}} e^{i r^{3/2} \tilde G_{N,a,\gamma}}  \varkappa(\lambda^{-1/3}r^{1/2} \tilde p,\lambda^{-1/3}r^{1/2}\tilde q, t,x,h,a,1/N) \, d\tilde p d\tilde q \right| \lesssim \frac{1}{r^{5/4}}\,,
\end{equation}
where $r$ is our large parameter, and $\tilde G_{N,a,\lambda}(\tilde p,\tilde q,t,x) =r^{-3/2}G_{N,a,\lambda}(r^{1/2}p,r^{1/2}q,t,x)$.
Let us compute, using \eqref{Gfirstderiv},
\begin{equation}\label{derivtildep}
\partial_{\tilde p} \tilde G_{N,a,\lambda}=\tilde p^{2}-\cos\theta+\frac{\lambda^{\frac 13} K}{N r^{\frac 12}}(\tilde p+\tilde q)-\frac{(\tilde p+\tilde q)^{2}}{4 N^{2}}, \quad \partial_{\tilde q} \tilde G_{N,a,\lambda}=\tilde q^{2}-\sin\theta+\frac{\lambda^{\frac 13} K}{ N r^{\frac 12}}(\tilde p+\tilde q)-\frac{(\tilde p+\tilde q)^{2}}{4 N^{2}}.
\end{equation}
On the support of $\varkappa(\cdots)$ %
we have $|(\tilde p,\tilde q)|\lesssim \lambda^{1/3} r^{-1/2}\lesssim \lambda^{1/3} r_{0}^{-1/2}$: for $\lambda^{1/3}\leq N$, the last term in both derivatives is $O(r_{0}^{-1})$, while the next to last term is $r_{0}^{-1/2}O(\tilde p,\tilde q)$; as $|\frac{\lambda^{1/3}}{N} K\frac{(\tilde p+\tilde q }{r^{1/2}}|\lesssim r_0^{-1/2} |\tilde p+\tilde q|$. Hence, when $|(\tilde p,\tilde q)|>\tilde C$ with $\tilde C$ sufficiently large, the corresponding part of the integral is $O(r^{-\infty})$ by integration by parts. So we are left with restricting our integral to a compact region in $(\tilde p,\tilde q)$.

We remark that, from $X\leq 1$, we have $A\geq B$ (and $A=B$ if and only if $X=1$), e.g. $\cos \theta \geq \sin \theta$ and therefore $\theta\in (-\frac{3\pi}{4},\frac{\pi}{4})$. We proceed differently upon the size of $B=r\sin\theta$. If $\sin \theta<-C/r^{1/2}$ for some $C>0$ sufficiently large then $\partial_{\tilde q}\tilde G_{N,a,\lambda}>c/(2r^{1/2})$ for some $C>c>0$ and the phase is non stationary.
Indeed, in this case 
\[
\partial_{\tilde q}\tilde G_{N,a,\lambda}\geq \tilde q^{2}+\frac{C}{2 r^{1/2}}+\frac{\lambda^{1/3} K}{N r^{1/2}}(\tilde p +\tilde q )-\frac{(\tilde p +\tilde q)^{2}}{4N^{2}}
\]
and using that $\tilde p, \tilde q$ are bounded, that on the support of $\varkappa$ we have $|r^{1/2}(\tilde p,\tilde q)|\lesssim \lambda^{1/3}$ and that $\frac 1N\lesssim \frac{1}{\lambda^{1/3}}\ll 1$, we then have, for some $C$ large enough
\[
\frac{\lambda^{1/3}}{N} (\tilde p+\tilde q)\Big(\frac{K}{r^{1/2}}-\frac{(\tilde p+\tilde q)}{4N\lambda^{1/3}}\Big)\lesssim \frac{C}{4r^{1/2}}\,.
\]
We recall that on the support of $\psi_2(\alpha)$ we had $\alpha\in [\frac 12,\frac 32]$ and the critical point $\alpha_{c}$ is such that \eqref{eq:alphac} holds (with $\gamma$ replaced by $a$ in this case) hence 
$K$ stays close to $1$ as the main contribution of $\alpha_c$. It follows that $\partial_{\tilde q}\tilde G_{N,a,\lambda}>C/(2r^{1/2})$ and integrations by parts yield a bound $O(r^{-n})$ for all $n\geq 1$.

Next, let $\sin \theta >-C/r^{1/2}$ and assume $A>0$ (since otherwise the non-stationary phase applies), which in turn implies $A>r_{0}/2$. Indeed, $\cos \theta \geq \sin \theta>-C/r^{1/2}$ implies $\theta\in (-\frac{C}{\sqrt{r_0}},\frac{\pi}{4})$ and therefore in this regime $\cos\theta\geq \frac{\sqrt{2}}{2}$. Consider first the case $|\sin \theta |<C/r^{1/2}$. Non degenerate stationary phase always applies in $\tilde p$, at two (almost) opposite values of $\tilde p$, such that $|\tilde p_{\pm}|\sim |\pm\sqrt{\cos\theta}|\geq 1/4$, and the integral in \eqref{eq:93} rewrites
\begin{multline}
  \label{eq:95}
      r \int_{\R^{2}} e^{i r^{3/2} \tilde G_{N,a,\lambda}}  \varkappa(\lambda^{-1/3}r^{1/2} \tilde p,\lambda^{-1/3}r^{1/2}\tilde q, t,x,h,a,1/N) \, d\tilde p d\tilde q \\
      = \frac r {r^{3/4}} \left( \int_{\R} e^{i r^{3/2} \tilde G^+_{N,a,\lambda}}  \varkappa^{+}(\tilde q,t,x,h,a,1/N) \,  d\tilde q\right. 
+\left. \int_{\R} e^{i r^{3/2} \tilde G^-_{N,a,\lambda}}  \varkappa^{-}(\tilde q ,h,a,1/N) \, d\tilde q\right)\,.
\end{multline}
Indeed, the phase is stationary in $\tilde p$ when $\tilde p^2=\cos\theta-\frac{\lambda^{1/3}K_a}{Nr^{1/2}}(\tilde p+\tilde q)+\frac{(\tilde p+\tilde q)^2}{4N^2}$, and from $\cos\theta\geq \frac{\sqrt{2}}{2}$ and $\frac{1}{r}\leq \frac{1}{r_0}\ll 1$, there are exactly two disjoint solutions to $\partial_{\tilde p}\tilde G_{N,a,\lambda}=0$, that we denote $\tilde p_{\pm}=\pm\sqrt{\cos\theta}+O(r^{-1/2})$. We compute, at critical points, %
\[
\partial^2_{\tilde p,\tilde p}\tilde G_{N,a,\lambda}|_{
p_{\pm}}=2\tilde p+\frac{\lambda^{1/3}K}{Nr^{1/2}}(1+O(a))+O(N^{-2})|_{\tilde p_{\pm}},
\]
where we used $\tilde p,\tilde q$ bounded to deduce that all the terms except the first one are small. As $\lambda^{1/3}\lesssim N$, $r^{-1/2}\ll 1$, $K_a$ bounded, close to $1$, for $\tilde p\in\{\tilde p_{\pm}\}$ we get $\partial^2_{\tilde p,\tilde p}\tilde G_{N,a,\lambda}|_{\tilde p_{\pm}}\sim 2\tilde p_{\pm}+O(r^{-1/2})$,
and as $|\tilde p_{\pm}|\geq \frac 14-O(r^{-1/2})$, stationary phase applies. The critical values of the phase at $\tilde p_{\pm}$, denoted $\tilde G^{\pm}_{N,a,\lambda}$, are such that
\begin{equation}\label{derivPhi3q}
\partial_{\tilde q} \tilde G^{\pm}_{N,a,\lambda}(\tilde q,.):=\partial_{\tilde q}\tilde G_{N,a,\lambda}(\tilde q,\tilde p_{\pm},.)
= \tilde q^{2}-\sin\theta +\frac{\lambda^{1/3} K(\tilde p+\tilde q)}{ N r^{1/2}}-\frac{(\tilde p+\tilde q)^{2}}{4 N^{2}}|_{\tilde p=\tilde p_{\pm}}.
\end{equation}
As $|\sin \theta |<C/r^{1/2}$, the phases $\tilde G^{\pm}_{N,a,\lambda}$ may be stationary but degenerate; taking two derivatives in \eqref{derivPhi3q}, one easily checks that
$|\partial^{3}_{\tilde q} \tilde G^{\pm}_{N,a,\lambda}|\geq 2-O(r_{0}^{-1/2})\,$. Hence we get, by Van der Corput Lemma
\begin{equation}
  \label{eq:97}
\left|  \int_{\R} e^{i r^{3/2}\tilde G^{\pm}_{N,a,\lambda}}  \varkappa^{\pm}(\tilde q ,t,x,y,h,a,1/N) \,  d\tilde q\right| \lesssim (r^{3/2})^{-1/3}\,.
\end{equation}
Using \eqref{eq:95} and \eqref{eq:97} eventually yields
\begin{equation}
  \label{eq:98}
  \left|   r \int_{\R^{2}} e^{i r^{3/2}\tilde G_{N,a,\lambda}}  \varkappa(\lambda^{-1/3}r^{1/2}\tilde p, \lambda^{-1/3}r^{1/2}\tilde q, t,x,h,a,1/N) \, d\tilde p d\tilde q\right|\lesssim r^{-1/4}.
\end{equation}
Notice moreover that $|B|=|r\sin\theta|\leq C r^{1/2}$, hence from $r^{2}=A^{2}+B^{2}$, we have $A\sim r$ (large) and $r^{-1/4}\lesssim 1/(1+|B|^{1/2})$: \eqref{eq:91} holds true and, replacing $B$ by $\lambda^{2/3}(K^2-1)$, it yields \eqref{eq:90}. Replacing $A,B$ by their formulas and using $a^2=(h\lambda)^{4/3}$, we obtain from \eqref{eq:90} 
\begin{equation}
|V_{N,h,a}(t,x)|\leq \frac{a^2}{h}\frac{1}{\sqrt{\lambda N}}\frac{\lambda^{-\frac 23}}{(1+\lambda^{\frac 13}|K_a^2-1|^{\frac 12})}
=\frac{2h^{\frac 13}}{2\sqrt{N/\lambda^{\frac 13}}+\lambda^{\frac 16}\sqrt{K_a+1}|4NK_a-4N|^{\frac 12}}\,.
\end{equation}
In the last case $\sin \theta>C/r^{1/2}$ ($A\geq B\geq Cr^{1/2}$), stationary phase holds in $(\tilde p,\tilde q)$: the determinant of the Hessian is at least $C\sqrt{\cos \theta}\sqrt{\sin \theta}$ and we get,  %
\[
\Big|\text{(LHS)\eqref{eq:93}}\Big|\lesssim  \frac 1 {(\sqrt{\cos \theta}\sqrt{\sin \theta})^{1/2} r^{3/2}}  \lesssim\frac 1 r \frac 1  {(r \sqrt{\cos \theta}\sqrt{\sin \theta})^{1/2}}\lesssim \frac 1 r \frac{1}{|AB|^{1/4}}
\]
so in this case our estimate is slightly better than \eqref{eq:90}, as we have
\begin{equation}
  \label{eq:100}
     \left| \int_{\R^{2}} e^{\frac ih \phi_{N,a}}  \varkappa(s,\sigma,t,x,h,a,1/N) \, ds d\sigma \right| \lesssim \frac 1 {\lambda^{2/3}|AB|^{1/4}}\leq \frac 1 {\lambda^{2/3}|B|^{1/2}}\,.
\end{equation}
This completes the proof of Proposition  \ref{dispNgrand} as it eventually yields
\begin{equation}
|V_{N,h,a}(t,x)|\lesssim \frac{(h\lambda)^{4/3}}{h}\frac{\lambda^{-1/2}}{N^{1/2}} \frac 1 {\lambda^{2/3}|B|^{1/2}}\sim h^{1/3}\frac{\lambda^{1/6}}{N^{1/2}}\frac{1}{\lambda^{1/3}|K^2-1|^{1/2}}\,.
\end{equation}

\subsection*{Proof of Propositions \ref{dispNpetitloin} and \ref{dispNpetitpres}}
We follow closely \cite[Prop.5]{ILP3}. Let $1\leq N< \lambda^{1/3}$: we aim at proving
\begin{equation}\label{eq:90bis}
   \left| \int_{\R^{2}} e^{\frac ih \phi_{N,a}}  \varkappa(\sigma,s,t,x,h,a,1/N) \, ds d\sigma \right| \lesssim N^{1/4}\lambda^{-3/4}\,.
\end{equation}
As $N<\lambda^{1/3}$, ignoring the last two terms in the first order derivatives of $\phi_{N,a}$, as we did in the previous case, is no longer possible. Set $\Lambda=\lambda/N^{3}$ to be the new large parameter. Rescale again variables $\sigma=p'/N$ and $s=q'/N$ and set now
\[
\Lambda G_{N,a}(p',q',t,x)=\frac 1h\Big( \phi_{N,a}(\sigma,s, t,x)-\phi_{N,a}(0,0,t,x)\Big).
\]
We are reduced to proving %
$\left| \int_{\R^{2}} e^{i \Lambda G_{N,a}}  \varkappa(p'/N,q'/N,\cdots) \, dp' dq' \right| \lesssim \Lambda^{-3/4}$. Compute %
\begin{equation}\label{GNderiv}
\nabla_{(p',q')}G_{N,a}=\frac{N^3}{h}\Big(\frac{\partial \sigma}{\partial p'}\partial_{\sigma} \phi_{N,a},\frac{\partial s}{\partial q'}\partial_{s}\phi_{N,a}\Big)|_{(p'/N,q'/N)}
=\Big(p'^2+N^2(X-\alpha_c), q'^2+N^2(1-\alpha_c)\Big),
\end{equation}
where, using  \eqref{eq:alphac}, $\alpha_c(\sigma,s,\cdot)|_{(\sigma=p'/N,s=q'/N)}
=\Big(K-\frac{p'}{2N^2}-\frac{q'}{2N^2}\Big)^2$.
Recall that $K=\sqrt{\frac{T}{2N}}=\sqrt{\alpha_c|_{\sigma=s=0}}$ is close to $1$ on the support of $\psi_2$.
We define %
$ A'=(K^{2}-X)N^{2}$ and $ B'=(K^{2}-1)N^{2}$. First order derivatives of $G_{N,a,\lambda}$ read 
\[
\partial_{p'}G_{N,a}=p'^2-A'+K(p'+q')-\frac{1}{4N^2}(p'+q')^2,\quad 
\partial_{q'}G_{N,a}=q'^2-B'+K(p'+q')-\frac{1}{4N^2}(p'+q')^2.
\]
Unlike the previous case, the two last terms are no longer disposable. We start with $|(A',B')|\geq r_{0}$ for some large, fixed $r_{0}$, in which case we can follow the same approach as in the previous case. Set again $A'=r\cos \theta$ and $B'=r\sin \theta$. If $|(p',q')|<r_{0}/2$, then the corresponding integral is non stationary and we get decay by integration by parts. We change variables $(p',q')=r^{1/2}(\tilde p',\tilde q')$ with $r_0\leq r\lesssim N^2$ and aim at proving the following
\begin{equation}
  \label{eq:101}
     \left| r  \int_{\R^{2}} e^{i r^{3/2} \Lambda \tilde G_{N,a}}  \varkappa(r^{1/2}\tilde p'/N,r^{1/2} \tilde q'/N,t,x,h,a,1/N) \, d\tilde p' d\tilde q' \right| \lesssim r^{-1/4 }\Lambda^{-5/6}\,,
\end{equation}
The new phase is $\tilde G_{N,a}(\tilde p',\tilde q',t,x)=r^{-3/2}G_{N,a}(r^{1/2}\tilde p',r^{1/2}\tilde q',t,x)$. 
We compute 
\[
  \partial_{\tilde p'} \tilde G_{N,a}=\tilde p'^{2}-\cos\theta+ \frac{K}{r^{1/2}}(\tilde p'+\tilde q')-\frac{(\tilde p'+\tilde q')^{2}}{4 N^{2}}, \quad
 \partial_{\tilde q'} \tilde G_{N,a}=\tilde q'^{2}-\sin\theta+ \frac{K}{r^{1/2}}(\tilde q'+\tilde q')-\frac{(\tilde p'+\tilde q')^{2}}{4 N^{2}}.
\]
To the extend it is possible to do so, we follow the previous case $\lambda^{1/3}\leq N$. From $X\leq 1$, $ A'\geq  B'$ implying $\cos\theta\geq \sin \theta$. If $|(\tilde p',\tilde q')|\geq \tilde C$ for some large $\tilde C\geq 1$, then $(\tilde p'_c, \tilde q'_c)$ are such that $\tilde p'^2_c\geq \tilde q'^2_c$ and if $\tilde C$ is sufficiently large non-stationary phase applies (pick any $\tilde C>4$.) Therefore we are reduced to bounded $|(\tilde p',\tilde q')|$. We sort out several cases, depending upon $B'=r\sin\theta $ : if $\sin \theta<-\frac{C}{\sqrt{r}}$ for some sufficiently large constant $C>0$, then
\[
\partial_{\tilde q'} \tilde G_{N,a} \geq \tilde q'^{2}+\frac{C}{r^{1/2}}+ \frac{K}{r^{1/2}}(\tilde p'+\tilde q')-\frac{(\tilde p'+\tilde q')^{2}}{4 N^{2}},
\]
and $N$ is sufficiently large in this case (indeed, recall that $r_0\leq r\lesssim N^2$ so that $\frac{1}{\sqrt{r}}\geq \frac 1N$);  then, non-stationary phase applies  as the sum of the last three terms in the previous inequality is greater than $C/(2r^{1/2})$ if $C$ is large enough. If $|\sin\theta|\leq \frac{C}{\sqrt{r}}$ then, again, $\theta\in (-\frac{C}{\sqrt{r_0}},\frac{\pi}{4})$ and $\cos\theta\geq \frac{\sqrt{2}}{2}$. We have $|B'| =|r\sin\theta|\leq C \sqrt{r}$; if $|B'|<C$, then $1+|B'|\lesssim r^{1/2}$, while $|A'|\sim r$. Stationary phase applies in $\tilde p'$ with non-degenerate critical points $\tilde p'_{\pm}$ and yields a factor $(r^{3/2}\Lambda)^{-1/2}$; the critical value of the phase function at these critical points, that we denote $\tilde G^{\pm}_{N,a}$, is always such that $|\partial^3_{\tilde q'}\tilde G^{\pm}_{N,a}|\geq 2-O(r_0^{-1/2})$ and the integral in $\tilde q'$ is bounded by $(r^{3/2}\Lambda)^{-1/3}$ by Van der Corput. We therefore obtain \eqref{eq:101} which yields, using that $|B'|=|N^2(K^2-1)|\leq r^{1/2}$,
\begin{multline*}
|V_{N,h,a}(t,x,y)|=\frac{h^{1/3}\lambda^{4/3}}{\sqrt{\lambda N}N^2}\Big |r \int_{\R^{2}} e^{i r^{3/2} \Lambda \tilde G_{N,a}}  \varkappa(r^{1/2}\tilde p'/N,r^{1/2} \tilde q'/N,t,x,h,a,1/N) \, d\tilde p' d\tilde q'\Big |\\
\lesssim \frac{h^{1/3}\lambda^{5/6}}{N^{5/2}}r^{-1/4}\Big(\frac{\lambda}{N^3}\Big)^{-5/6}\lesssim \frac{h^{1/3}}{(1+|B'|^{1/2})}\sim \frac{h^{1/3}}{(1+N|K-1|^{1/2})}.
\end{multline*}
If $\sin\theta > \frac{C}{\sqrt{r}}$, then $B'=r\sin\theta>C\sqrt{r}$ and therefore $N^2|K^2-1|>Cr^{1/2}$. We do stationary phase in both variables with large parameter $r^{3/2}\Lambda$ as the determinant of the Hessian at critical points is at least $C\sqrt{\cos\theta\sin \theta}$, and obtain, for left hand side term in \eqref{eq:101}, a bound
$\frac{c r}{(\sqrt{\sin\theta}\sqrt{\cos\theta})^{1/2}r^{3/2}\Lambda}=\frac{1}{\Lambda}\frac{1}{|A'B'|^{1/4}}\leq \frac{1}{\Lambda}\frac{1}{|B'|^{1/2}}$.
We just proved that for $N<\lambda^{1/3}$ and not too small $N^2|K-1|$,
  \begin{equation}
    \label{eq:2ffbis}
       \left| V_{N,h,a}(t,x)\right| \lesssim  \frac{h^{1/3}}{\lambda^{1/6}\sqrt{N}|K-1|^{1/2}} \,.
  \end{equation}
We now move to the most delicate case $|(A',B')|\leq r_{0}$. For $|(p',q')|$ large, the phase is non stationary and integrations by parts provide $O(\Lambda^{-\infty})$ decay. So we may replace $\varkappa$ by a cut-off, that we still call $\varkappa$, compactly supported in $|(p',q')|<R$. We proceed by identifying one variable where usual stationary phase applies and then evaluating the remaining $1D$ oscillatory integral using Van der Corput (with different decay rates depending on the lower bounds on derivatives, of order at most $4$.) Using \eqref{GNderiv}, we compute derivatives of $G_{N,a}$
\begin{equation}\label{eq:H_Nderivp}
\partial_{p'}G_{N,a} =p'^2+N^2(X-\alpha_c),\quad \partial_{q'}G_{N,a} =q'^2+N^2(1-\alpha_c).
\end{equation}
The second order derivatives of $G_{N,a}$ are given by
\begin{gather}
 \partial^2_{p'p'}G_{N,a}=2p'-N^2\partial_{p'}\alpha_c,\quad
  \partial^2_{q'q'}G_{N,a}=2q'-N^2\partial_{q'}\alpha_c,\\
   \partial^2_{q'p'}G_{N,a}=-N^2\partial_{q'}\alpha_c
  =\partial^2_{p'q'}G_{N,a} =-N^2\partial_{p'}\alpha_c.
\end{gather}
At critical points, where $\partial_{p'}G_{N,a}=\partial_{q'}G_{N,a}=0$, the determinant of the Hessian reads
\[
\det\text{Hess}_{(p',q')}G_{N,a}|_{\nabla_{(p',q')}G_{N,a}=0}=4p'q'-N^2(p'+q')\partial_{p'}\alpha_c.
\]
If $|\det\text{Hess}_{(p',q')}G_{N,a}|>c>0$ for some small $c>0$ we can apply usual stationary phase in both variables $p',q'$. We expect the worst contributions to occur in a neighborhood of the critical points where $|\det\text{Hess}_{(p',q')}G_{N,a}|\leq c$ for some $c$ sufficiently small. We turn variables with $\xi_{1}=(p'+q')/2$ and $\xi_{2}=(p'-q')/2$. Then $p'=\xi_{1}+\xi_{2}$ and $q'=\xi_{1}-\xi_{2}$, and we also let $\mu:=A'+B'=N^2(2K_a^2-1-X)$,  $\nu:=A'-B'=N^2(1-X)$. The most degenerate situation will turn out to be $\nu=\mu=0$ and $\xi_{1}=0,\xi_{2}=0$. Let $g_{N,a}(\xi_1,\xi_2)=G_{N,a}(\xi_1+\xi_2,\xi_1-\xi_2)$.
\vskip1mm

\paragraph{\bf Case $c\lesssim |\xi_1|$ for small $0<c<1/2$}
For $\xi_1$ outside a small neighbourhood of $0$, non degenerate stationary phase applies in $\xi_2$ and the critical value $g_{N,a}(\xi_1,\xi_{2,c})$ may have degenerate critical points of order at most $2$.  
The phase $g_{N,a}$ is stationary in $\xi_2$ whenever $\partial_{p'}G_{N,a}=\partial_{q'}G_{N,a}$ and $\partial_{p'}\alpha_c=\partial_{q'}\alpha_c$. We have 
\[
\partial^2_{\xi_2,\xi_2}g_{N,a}(\xi_1,\xi_2)=\Big(\partial^2_{p'p'}G_{N,a}-2\partial^2_{p'q'}G_{N,a}+\partial^2_{q'q'}G_{N,a}\Big)(p',q')|_{\xi_1,\xi_2}.
\]
Using the explicit form of the second order derivatives of $G_{N,a}$ given above, at $p'=\xi_1+\xi_2$, $q'=\xi_1-\xi_2$ such that $p'^2+N^2(X-\alpha_c)=q'^2+N^2(1-\alpha_c)$, we obtain 
\[
\partial^2_{\xi_2,\xi_2}g_{N,a}(\xi_1,\xi_2)|_{\partial_{\xi_2}g_{N,a}=0}=2(p'+q')=4\xi_1.
\]
As $|\xi_1|\gtrsim c$, stationary phase applies in $\xi_2$. We denote $\xi_{2,c}$ the critical point, such that 
\[
\partial_{\xi_2}g_{N,a}(\xi_1,\xi_2)=\Big(\partial_{p'}G_{N,a}-\partial_{q'}G_{N,a}\Big)(p',q')|_{p'=\xi_1+\xi_2,q'=\xi_1-\xi_2}=0\,,
\]
which may be rewritten as $(\xi_1+\xi_{2,c})^2+N^2(X-\alpha_c)=(\xi_1-\xi_{2,c})^2+N^2(1-\alpha_c)$, which, in turn, yields $4\xi_1\xi_{2,c}=N^2(1-X)=\nu$ and therefore $\xi_{2,c}=\frac{\nu}{4\xi_1}$. We now compute higher order derivatives of the critical value of $g_{N,a}(\xi_1,\xi_{2,c})$ with respect to $\xi_1$.
\begin{lemma}\label{lemxi1grand}
For $|N|\geq 1$, the phase $g_{N,a}(\xi_1,\xi_{2,c})$ may have critical points degenerate of order at most $2$.
\end{lemma}
\begin{proof}
As $\sqrt{\alpha_c}|_{\partial_{\xi_2}g_{N,a}=0}=K-\frac{\xi_1}{N^2}$, 
\begin{align}
\nonumber
\partial_{\xi_1}(g_{N,a}(\xi_1,\xi_{2,c}))&=\partial_{\xi_1}g_{N,a}(\xi_1,\xi_{2,c})+\frac{\partial \xi_{2,c}}{\partial\xi_1}\partial_{\xi_2}g_{N,a}(\xi_1,\xi_2)|_{\xi_2=\xi_{2,c}} =\Big(\partial_{p'}G_{N,a}+\partial_{q'}G_{N,a}\Big)(p',q')|_{\xi_1,\xi_{2,c}}\\
\label{gfirstderiv}
&=2\xi_1^2(1-\frac{1}{N^2})+2\frac{\nu^2}{16\xi_1^2}-\mu+4K\xi_1.
\end{align}
Taking a derivative of \eqref{gfirstderiv} with respect to $\xi_1$ yields
\begin{equation}
\partial^2_{\xi_1,\xi_1}(g_{N,a}(\xi_1,\xi_{2,c}))=4\xi_1\Big(1-\frac{1}{N^2}\Big)-\frac{\nu^2}{8\xi_1^3}+4K.
\end{equation}
In the same way we compute
\[
\partial^3_{\xi_1,\xi_1,\xi_1}(g_{N,a}(\xi_1,\xi_{2,c}))|_{\partial_{\xi_1}(g_{N,a}(\xi_1,\xi_{2,c}))=\partial^2_{\xi_1,\xi_1}(g_{N,a}(\xi_1,\xi_{2,c}))=0}=4\Big(1-\frac{1}{N^2}\Big)+\frac{3\nu^2}{8\xi_1^4}+O(a)\,.
\]
Let first $|N|\geq 2$, then we immediately see that the third order derivative takes positive values and stays bounded from below by a fixed constant, $\partial^3_{\xi_1,\xi_1,\xi_1}(g_{N,a}(\xi_1,\xi_{2,c}))\geq 2$, and therefore the critical points may be degenerate (when $\partial^2_{\xi_1,\xi_1}(g_{N,a}(\xi_1,\xi_{2,c}))=0$) of order at most $2$. Let now $|N|=1$ when the coefficient of $2\xi_1^2$ in \eqref{gfirstderiv} is $O(a)$. Assume that for $c\lesssim |\xi_1|$ the first two derivative vanish, then $\frac{\nu^2}{8\xi_1^3}= 4K+O(a)$ and therefore the third derivative cannot vanish as its main contribution is $\frac{3\nu^2}{8\xi_1^4}$.
\end{proof}

\paragraph{\bf Case $|\xi_1|\lesssim c$, for small $0<c<1/2$} First, (usual) stationary phase applies in $\xi_1$:
\[
\partial_{\xi_1}g_{N,a}(\xi_1,\xi_{2})%
=(\xi_1+\xi_{2})^2+N^2(X-\alpha_c)+(\xi_1-\xi_{2})^2+N^2(1-\alpha_c)\,,
\]
with $K=\frac{T}{2N}$, $\sqrt{\alpha_c}=K-\frac{(\sigma+s)}{2N}$ 
and $\sigma+s=2\xi_1/N$. As $|\xi_1|\leq c<\frac 12$ small, $a\leq\varepsilon_0$ and $\alpha_c\in [\frac 12,\frac 32]$ on the support of the symbol, from $K=\sqrt{\alpha_c}+O(c/N^2)$ we have $K\in [1/4,2]$ for all $N\geq 1$. The derivative of $g_{N,a}(\xi_1,\xi_2)$ becomes
\begin{equation}\label{part1derivg}
\partial_{\xi_1}g_{N,a}(\xi_1,\xi_{2}) = 2\xi_1^2+2\xi_{2}^2-\mu-2N^2\Big[\Big(K-\frac{\xi_1}{N^2}\Big)^2-K^2\Big]=2\xi_1^2(1-\frac{1}{N^2})+2\xi_{2}^2-\mu+4K\xi_1.
\end{equation}
At the critical point, the second derivative with respect to $\xi_1$ is 
\begin{equation}
\partial^2_{\xi_1,\xi_1}g_{N,a}(\xi_1,\xi_2)|_{\partial_{\xi_1}g_{N,a}(\xi_1,\xi_2)=0}=4\xi_1(1-\frac{1}{N^2})+4K,
\end{equation}
and as $K\in [\frac 14,2]$, the leading order term is $4K$. Stationary phase applies for any $|N|\geq 1$ yielding a factor $\Lambda^{-1/2}$. 
We are left with the integral with respect to $\xi_2$.  We first compute the critical point $\xi_{1,c}$, solution to $\partial_{\xi_1}g_{N,a}(\xi_1,\xi_{2})=0$, as a function of $\xi_2$: 
\begin{equation}\label{xi1c}
2\xi^2_{1,c}(1-\frac{1}{N^2})+4K\xi_{1,c}+2\xi^2_2-\mu=0.
\end{equation}
In order to have real solutions for $|\xi_{1,c}|\leq c$ we must have $|\mu/2-\xi_2^2|\lesssim c$ (as for $|\mu/2-\xi_2^2|> 4c$, the equation \eqref{xi1c} has no real solution $\xi_{1,c}$ such that $|\xi_{1,c}|\leq c$). Explicit computations give :

\begin{lemma}
For all $|N|\geq 1$ and for $|\mu/2-\xi_2^2|\leq 4c$ small enough, \eqref{xi1c} has one real valued solution, 
\begin{equation}\label{eq:formxi1cxi2}
\xi_{1,c}=(\mu/2-\xi^2_2)\Xi,
\end{equation}
where $\Xi=\Xi(\mu/2-\xi_2^2, K,1/N^2)$ is defined as
\begin{equation}\label{defXi0}
\Xi(\mu/2-\xi_2^2, K,1/N^2)=\Bigl(K+\sqrt{K^2+(\mu/2-\xi_2^2)(1-1/N^2)}\Bigr)^{-1}.
\end{equation}
\end{lemma}
Let $\tilde g_{N,a}(\xi_2):=g_{N,a}(\xi_{1,c},\xi_2)$ : we have $\partial_{\xi_2}\tilde g_{N,a}=0$ when $(\partial_{p'}G_{N,a}-\partial_{q'}G_{N,a})(p',q')|_{(\xi_{1,c},\xi_2)}=0$ which is equivalent to $4\xi_{1,c}\xi_2=\nu$. From $\partial_{\xi_2}\tilde g_{N,a}=\nu-4\xi_{1,c}\xi_2$ we find $\partial^2_{\xi_2\xi_2}\tilde g_{N,a} =-4(\xi_2\partial_{\xi_2}\xi_{1,c}+\xi_{1,c})$. Then, critical points $\xi_{2}$ are degenerate if $\xi_{1,c}=-\xi_2\partial_{\xi_2}\xi_{1,c}$ which gives, replacing $\xi_{1,c}$ by \eqref{eq:formxi1cxi2},
\begin{multline}\label{seccritxi2}
(\mu/2-\xi_2^2)\Xi
=-\xi_2\Big( -2\xi_2\Xi+(\mu/2-\xi_2^2)\partial_{\xi_2} \Xi\Big)
=-\xi_2\Big(-2\xi_2 \Xi+(\mu/2-\xi_2^2)\times \frac{\xi_2(1-1/N^2)}{\sqrt{K^2+(\mu/2-\xi_2^2)(1-1/N^2)}}\Xi^2\Big)\\
=2 \xi^2_2\Xi\Big( 1-(\mu/2-\xi_2^2)\times \frac{(1-1/N^2)}{2\sqrt{K^2+(\mu/2-\xi_2^2)(1-1/N^2)}}\Xi\Big)
=2 \xi^2_2\Xi\Big( 1-(\mu/2-\xi_2^2)\times\tilde \Xi\Big),
\end{multline}
where we have used $\partial_{\xi_2}\xi_{1,c}=-2\xi_2\Xi(1-(\mu/2-\xi_2^2)\tilde\Xi)$ and set
\[
\tilde \Xi(\mu/2-\xi_2^2,K_a,1/N^2):=\frac{(1-1/N^2)\Xi(\mu/2-\xi_2^2,K,1/N^2)}{2\sqrt{K^2+(\mu/2-\xi_2^2)(1-1/N^2)}}.
\]
Recall that $K\in [1/4,2]$ and that $|\mu/2-\xi_2^2|\leq 4c$ with $c$ small enough. 
As $\Xi\sim 1/2$ from \eqref{defXi0} doesn't vanish, the critical points are degenerate if 
\begin{equation}
  \label{eq:4}
  \mu/2-\xi_2^2=2\xi_2^2\Big(1-(\mu/2-\xi_2^2)\tilde \Xi(\mu/2-\xi_2^2,K,1/N^2)\Big).
\end{equation}
Rewrite \eqref{eq:4}
\[
(\mu/2-\xi_2^2)\Big(2+\frac{1}{1-(\mu/2-\xi^2_2)\tilde\Xi}\Big)=\mu
\]
which may have solutions only if $\mu$ is also small enough, $|\mu|\leq 10c$. Let $z=\mu/2-\xi_2^2$; for $|z|\leq 4c$ and $|\mu|\leq 10c$ with $c$ small enough, we may now seek the solution to \eqref{eq:4} as $z=\mu Z(\mu,K,1/N^2)$ and obtain $Z(\mu,K,1/N^2)$ explicitly, with $Z(0,K,1/N^2)=\frac 13$. Solutions to \eqref{seccritxi2} (or \eqref{eq:4}) are therefore functions of $\sqrt{\mu}$ which both vanish at $\mu=0$. They may be written under the form 
\begin{equation}\label{solxi2pm}
\xi_{2,\pm}=\pm\frac{\sqrt{\mu}}{\sqrt{6}}\Big(1+\mu\zeta(\mu,K,1/N^2)\Big),
\end{equation}
for some smooth function $\zeta$.
We compute the third derivative of $\tilde g_{N,a}$ at $\xi_{2,\pm}$ defined in \eqref{solxi2pm} whenever the second derivative vanishes. Using again $\partial_{\xi_2}\xi_{1,c}=-2\xi_2\Xi(1-(\mu/2-\xi_2^2)\tilde\Xi)$ yields
\begin{multline}\label{thirdderivxi2}
\partial^3_{\xi_2,\xi_2,\xi_2}\tilde g_{N,a}(\xi_{1,c},\xi_2)|_{\xi_2=\xi_{2,\pm}}=-4(2\partial_{\xi_2}\xi_{1,c}+\xi_2\partial^2_{\xi_2,\xi_2}\xi_{1,c})|_{\xi_{2,\pm}}\\
=16\xi_{2}\Xi\Big(1-(\mu/2-\xi_2^2)\tilde\Xi\Big)
+8\xi_{2}\Xi(1+O(\mu/2-\xi_2^2; \xi_2^2)),
\end{multline}
where the last term in \eqref{thirdderivxi2} comes from $-4\xi_{2,\pm}\partial^2_{\xi_2,\xi_2}\xi_{1,c}$. We do not expand this formula as $\xi_{2,\pm}$ is sufficiently small for what we need. The first term in the second line of \eqref{thirdderivxi2} comes from the formula for $-8\partial_{\xi_2}\xi_{1,c}$. As the third derivative of $\tilde g_{N,a}$ is evaluated at $\xi_{2,\pm}$ given in \eqref{solxi2pm} and as $\Xi=\frac{1}{2K}(1+O(\mu/2-\xi^2))$, we obtain
 \begin{equation}
 \partial^3_{\xi_2,\xi_2,\xi_2}\tilde g_{N,a}(\xi_{1,c},\xi_2)|_{\xi_{2,\pm}}=24\xi_{2,\pm}\Xi(1+O(\mu/2-\xi_2^2; \xi_2^2))|_{\xi_{2,\pm}}\\
 =\frac{12\xi_{2,\pm}}{K}(1+O(\xi_{2,\pm}^2)).
 \end{equation}
It follows that at $\mu=\nu=0$, when $X=K=1$, the order of degeneracy is higher as $\xi_{2,\pm}|_{\mu=\nu=0}=0$ and $\partial^3_{\xi_2,\xi_2,\xi_3}\tilde g_{N,a}|_{\xi_{2,\pm},\mu=\nu=0}=0$. We now write 
\begin{equation}
\tilde g_{N,a}(\xi_2)=\tilde g_{N,a}(\xi_{2,\pm})+(\xi_2-\xi_{2,\pm})\partial_{\xi_2}\tilde g_{N,a}(\xi_{2,\pm})+\frac{(\xi_2-\xi_{2,\pm})^3}{6}\partial^3_{\xi_2,\xi_2,\xi_2}\tilde g_{N,a}(\xi_{2,\pm})+O((\xi_2-\xi_{2,\pm})^4),
\end{equation}
where $\partial^{4}_{\xi_{2}^{4}}\tilde g_{N,a}$ doesn't cancel at $\xi_{2,\pm}$ as it stays close to $12/K\in [6,48]$. We are to have $\partial_{\xi_2}\tilde g_{N,a}(\xi_{2,\pm})=0$, from which $\nu=4\xi_{1,c}|_{\xi_{2,\pm}}\xi_{2,\pm}$, which reads as
\begin{equation}
\nu=4\Big(\pm \frac{\sqrt{\mu}}{\sqrt{6}}(1+\mu\zeta(\mu))\Big)\\
\times (\mu/2-\xi_{2,\pm}^2)\Xi
\end{equation}
and replacing \eqref{solxi2pm} in \eqref{eq:formxi1cxi2} yields $\nu=\pm \frac{\sqrt{2} \mu^{3/2}}{3\sqrt{3}K}(1+O(\mu))$, which is at leading order the equation of a cusp.
At the degenerate critical points $\xi_{2,\pm}$ where $\nu=\pm \frac{\sqrt{2} \mu^{3/2}}{3\sqrt{3}K}(1+O(\mu))$, the phase integral behaves like
\[
I=\int_{\xi_{2}} \rho(\xi_{2}) e^{\mp i \Lambda \frac{\sqrt{2}\sqrt{\mu}} {K_a\sqrt 3} (\xi_{2}-\xi_{2,\pm})^{3}}\,d\xi_{2}\,,
\]
and we may conclude in a small neighborhood of the set $\{\xi_2^2+|\mu|+|\nu|^{2/3}\lesssim c\}$ (as outside this set, the non-stationary phase applies) by using Van der Corput lemma on the remaining oscillatory integral in $\xi_2$ with phase $\tilde g_{N,a}(\xi_2)$. In fact, on this set, $\partial^4_{\xi_2}\tilde g_{N,a}$ is bounded from below, which yields an upper bound $\Lambda^{-1/4}$, uniformly in all parameters. When $\mu\neq 0$, the third order derivative of the phase is bounded from below by $\frac{|\xi_2|}{K}$ : either $|\mu/6-\xi_2^2|\leq |\mu|/12$ and then $|\partial^3_{\xi_2}\tilde g_{N,a}|$ is bounded from below by $|\mu|^{1/2}/(12K_a)$ or
$|\mu/6-\xi_2^2|\geq |\mu|/12$ in which case $|\partial^2_{\xi_2}\tilde g_{N,a}|$ is bounded from below by $|\mu|/(12K)$. Hence, using that $K\in [1/4,2]$, we find $|\partial^3_{\xi_2}\tilde g_{N,a}|+|\partial^3_{\xi_2}\tilde g_{N,a}|\gtrsim \sqrt{|\mu|}$ (recall that here $\mu$ is small so $\sqrt{|\mu|}\geq |\mu|$) which yields an upper bound $(\sqrt{|\mu|}\Lambda)^{-1/3}$. Eventually we obtain
$ |I|\lesssim \inf\Big\{\frac{1}{\Lambda^{1/4}},\frac{1}{|\mu|^{1/6} \Lambda ^{1/3}}\Big\}$.
From $\mu=A'+B'$ and $\nu=A'-B'\sim \pm |\mu|^{3/2}$ and $|\mu|^{3/2}\ll |\mu|$ for $\mu<1$, we deduce that $A'\sim B'$ and therefore $|\mu| \sim 2 |B'|$, which is our desired bound \eqref{eq:2hh} after unraveling all notations, as the non degenerate stationary phase in $\xi_{1}$ provided a factor $\Lambda^{-1/2}$.

\section{Appendix}\label{secapp}

\subsection{Exponential sums estimates}
This section follow closely \cite[Section 3]{Rob16}; for details and proofs we refer to \cite{Rob16} and the references therein. We recall the well known Van der Corput estimates for exponential sums and some recent improvements in order to apply them to estimate the modulus of the following exponential sums 
$\sum_{l=1}^M e^{if(l)}$
where $f:[1,M]\rightarrow \mathbb{R}$ is a $\mathcal{C}^j$ function with $j\geq 1$. The literature on the subject of such trigonometric sums is abundant, and in particular goes back to 1916 with Weyl's results on the equidistribution of a real sequence modulo $1$. Subsequently, Hardy and Littlewood used Weyl's work for Waring's problem (see \cite{Rob16} and the references therein).
Below are some classical examples of such phase functions $f$ (see \cite{Rob16}).
Let $\alpha\in \mathbb{R}\setminus \mathbb{N}$ , $\alpha\neq 0$, $\tau>0$, $M, \lambda\in\mathbb{N}$ such that $M\leq \lambda$ ; we introduce $f_{\tau}(x)=\tau (\frac{\lambda+x}{\lambda})^{\alpha}$, for $x\in [1,M]$ or $f_{\tau}(x)=\tau \log(\lambda+x)$. Then for all $j\geq 1$, $\exists c_{1,2}= c_{1,2}(\alpha,j)>0$ such that 

\begin{equation}\label{fctalpha}
c_1\frac{\tau}{\lambda^j}\leq |f^{(j)}_{\tau}(x)|\leq c_2\frac{\tau}{\lambda^j}, \quad \forall x\in [1,M].
\end{equation}

\paragraph{\bf Van der Corput's second derivative test}\label{secVCD}
\begin{prop}\label{propVdC2}(Van der Corput, 1922, \cite[Thm.1]{Rob16}, \cite[Thm. 2.2]{GrKo91})
Let $\gamma\geq 1$ be a real number. There exists a constant $C(\gamma)>0$ such that for all integer $M\geq 1$, any real number $\delta_2>0$ and any $C^{2}$ function $f:[1,M]\rightarrow\mathbb{R}$ such that 
\[
\delta_2\leq |f''(x)|\leq \gamma\delta_2, \quad \forall x\in [1,M],
\]
one has
\begin{equation}\label{VdC2}
(VdC2)\quad Big|\sum_{l=1}^M e^{if(l)}\Big|\leq C(\gamma)(M\delta_2^{1/2}+\delta_2^{-1/2}).
\end{equation}
\end{prop}
\begin{rmq}
Remarks : the result is uniform with respect to $L,\delta_2$ and $f$. In particular, $\delta_2$ may depend on $M$, the optimal choice being $\delta_2=1/M$. The result is trivial for $\delta_2\geq 1$. However, as soon as $M\geq \delta_2^{-1}>4C(\gamma)^2$, the bound is non-trivial. For an explicit constant $C(\gamma)$ see the section below \cite[Thm. I.6.7]{Tenem08}.
\end{rmq}

\paragraph{\bf Van der Corput's $j$-th derivative test}
\begin{prop}\label{propVdCk}(\cite[Thm. 3]{Rob16}, \cite[Thm. 5.13]{Titch86})
Let $\gamma\geq 1$ be a real number and $j\geq 2$ be an integer. There exists a constant $C(\gamma,j)>0$ such that for any integer $M\geq 1$, any real number $\delta_j>0$ and any $C^j$ function $f:[1,M]\rightarrow\mathbb{R}$ such that 
\[
\delta_j\leq |f^{(j)}(x)|\leq \gamma\delta_j, \quad \forall x\in [1,M],
\]
one has 
\begin{equation}\label{VdCk}
(VdCj) \quad \Big|\sum_{l=1}^M e^{if(l)}\Big|\leq C(\gamma,j)\Big(M\delta_j^{\frac{1}{2^j-2}}+M^{1-2^{2-j}}\delta_j^{-\frac{1}{2^j-2}}\Big).
\end{equation}
\end{prop}

Let $j\geq 2$ and two real numbers $\theta,\beta>0$. We say that $(\theta,\beta)$ is a Van der Corput $j$-couple if for any $\mathcal{C}^j$ function $f:[1,M]\rightarrow \mathbb{R}$ such that $|f^{(j)}(x)|\sim \delta_j$ for $1\leq x\leq M$, one has
\begin{equation}\label{expnumbj}
\Big|\sum_{l=1}^M e^{if(l)}\Big|\lesssim M\delta_j^{\theta} \quad \text{ for all } M\geq \delta_j^{-\beta}.
\end{equation}

\paragraph{\bf Improvements for $j=3$}
\begin{rmq}
For $j=3$, Proposition \ref{propVdCk} gives 
\begin{equation}\label{exp3}
(VdC3)\quad \Big|\sum_{l=1}^M e^{if(l)}\Big|\leq C(\gamma,3)\Big(M\delta_3^{1/6}+M^{1/2}\delta_3^{-1/6}\Big).
\end{equation}
In particular, the exponents $(\theta_3,\beta_3)$ such that
$\Big|\sum_{l=1}^M e^{if(l)}\Big|\lesssim M\delta_3^{\theta_3}$ for all $M\geq \delta_3^{-\beta_3}$
are $(\theta_3,\beta_3)=(\frac16,\frac 23)$. Unlike the analogue for the second derivative test, it turns out that $\beta_3$ may be replaced by $\frac 12$ : this has been proven independently by Sargos and Gritsenko by different methods (see \cite[Corollary 4.2]{Sar} and \cite{Grit}) :
\end{rmq}

\begin{prop}\label{propbestk=3}(Sargos \cite{Sar}, Gritsenko \cite{Grit}) For any $M\geq 1$, any $\delta_3\in (0,1)$ and any $ \mathcal{C}^3$ function $f:[1,M]\rightarrow \mathbb{R}$ such that $|f^{'''}(x)|\sim \delta_3$ for $x\in [1,M]$, we have 
\begin{equation}\label{VdC3}
\Big|\sum_{l=1}^M e^{if(l)}\Big|\leq C(\gamma)(M\delta_3^{1/6}+\delta_3^{-1/3}).
\end{equation}
The exponents $(\theta_3,\beta_3)$ from Proposition \ref{propbestk=3} are $(\theta_3,\beta_3)=(\frac 16, \frac 12)$. 
\end{prop}
\begin{rmq}
The exponent $\theta_3=1/6$ is optimal, as shown in the counter-example in \cite[Lemma 7]{Rob16} where $f(l)=l^{3/2}$ and $M\sim \delta_3^{-2/3}$. However, for larger sums $M\sim \delta_3^{-1}$, the exponent $\theta_3=1/6$ is not optimal anymore (as large sums might be subject to more cancellations). In \cite[Thm.1]{Sar}, Sargos proved that if one adds to the condition $|f'''(l)|\sim \delta_3$ for all $l\in [1,M]$ the condition that $f'''$ is monotonous, then $\Big|\sum_{l=1}^M e^{if(l)}\Big|\lesssim M\delta_3^{\theta}$ with $\theta=1/6+1/1354$, hence $\theta_3=1/6$ is no longer optimal for $M\sim \delta_3^{-1}$. Here, both the monotony of $f'''$ and the size of $M$ are crucial.
\end{rmq}

\paragraph{\bf Improvements for $j=4$}

\begin{rmq}
For $j=4$, Proposition \ref{propVdCk} gives 
\begin{equation}\label{exp4}
(VdC4)\quad \Big|\sum_{l=1}^M e^{if(l)}\Big|\leq C(\gamma,4)\Big(M\delta_4^{1/14}+M^{3/4}\delta_4^{-1/14}\Big).
\end{equation} 
In particular, the exponents $(\theta_4,\beta_4)$ such that
$\Big|\sum_{l=1}^M e^{if(l)}\Big|\lesssim M\delta_4^{\theta_4}$ for all $M\geq \delta_4^{-\beta_4}$
are $(\theta_4,\beta_4)=(\frac{1}{14},\frac{4}{7})$. Improvements for $j=4$ are the following  :
\begin{itemize}
\item Robert \cite{Rob18} proved that $\theta_4=\frac{1}{14}$ may be replaced by any $\theta<\frac{1}{12}$ and shows that, uniformly for $l_0\in \mathbb{R}$, $M\geq \delta_4^{-3/5}$,
$\Big|\sum_{l=l_0}^{M+l_0} e^{if(l)}\Big|\leq C(\gamma,\epsilon)(M^{1+\epsilon}\delta_4^{1/12}+M^{\frac{11}{12}+\epsilon})$.


\item Robert $\&$ Sargos \cite{RoSa} proved that $\theta_4$ may be replaced by any $\theta<\frac{1}{13}$ provided that $\beta_4$ is replaced by $\frac{8}{13}$.
\end{itemize}
\end{rmq}

\paragraph{\bf  Further discussions on exponential sums in relation with the Lindelöf conjecture and with Conjecture \ref{thmLindel}}\label{secExpSums}
These kind of bounds for exponential sums have been extensively studied, in particular in order to find bounds for the rate of growth of the Riemann zeta function on the critical line. In fact, from \cite[Lemma 2.11]{GrKo91}, one has
\begin{equation}\label{zetacritline}
\zeta(\frac 12+i\tau)=\sum_{k\leq t}k^{-(1/2+i\tau)}+O(|\log \tau|),
\end{equation}
 where $\zeta$ is the Riemann function. Using \eqref{zetacritline} followed by a dyadic and Abel summation one has, for $|\tau|\geq 3$ 
\begin{equation}\label{sommexp}
|\zeta(\frac 12+i\tau)|\lesssim |\log \tau| \max_{\lambda\leq |\tau| } \lambda^{-1/2}\max_{1\leq K\leq \lambda}\Big|\sum_{k=\lambda+l, l\in \{1,...,K\}}e^{i\tau\log(\lambda+l)}\Big|.
\end{equation}
If $\sigma\in\mathbb{R}$, we define $\mu(\sigma)$ to be the infimum of all real numbers $a$ such that $|\zeta(\sigma+i\tau)|=O(\tau^a)$. 
The case $\sigma=\frac12$ is of particular interest and is called the Lindelöf problem. The Lindel\"of hypothesis asserts that, for any $\epsilon>0$, when $\tau\rightarrow\infty$ one should have 
 \begin{equation}\label{zeta}
 \zeta(\frac 12+i\tau)\lesssim C_{\epsilon} \tau^{\epsilon}, \quad |\tau|\geq 3.
 \end{equation}
This is equivalent to asserting optimal cancellation in the exponential sums \eqref{sommexp} connected to the zeta function and is deeply linked to the Riemann Hypothesis. 
The Phragmen-Lindel\"of theorem implies that $\mu$ is a convex function and the Lindel\"of hypothesis states that $\mu(\frac 12)=0$; the convexity property together with $\mu(1)=0$, $\mu(0)=\frac 12$, implies that $0\leq \mu(1/2)\leq 1/4$. This $1/4$ bound obtained by Lindel\"of has been lowered by Hardy and Littlewood to $1/6$ by applying Weyl's method of estimating exponential sums to the approximate functional equation. Since then, it has been lowered to slightly less than $1/6$ by several authors using very sophisticated arguments. 
More generally, the generalized Lindelöf hypothesis extends this principle to more general families of exponential sums of the form $\lambda^{-1/2} \Big|\sum_{k\sim \lambda} e^{i f(k/\lambda)}\Big|$ with $f$ satisfying \eqref{fctalpha}, predicting that for any such smooth phase function $f$ and any $\epsilon>0$, the associated exponential sum exhibits sub-polynomial growth in the parameter of the form $\tau^{\epsilon}$. 
While in the past, the tool for estimating such exponential sums was the Van der Corput iteration ((VdC3) implies $\epsilon=\frac 16$), more recent works strongly explored the Bombieri-Iwaniec method \cite{BoIw861}, \cite{BoIw862} which provided $\epsilon\leq \frac{9}{56}=\frac 16-\frac{1}{168}$. In \cite{Hux96}, Huxley developed and refined the Bombieri-Iwaniec approach \cite{Hux05} (see also \cite{Hux93} or Huxley-Koleskin \cite{HuKo91}) and produced $\epsilon\leq \frac{32}{205}=\frac 16-\frac{13}{6\times 205}$. 
The best known bound belongs to Bourgain \cite{Bourg17}, which proved (first $\epsilon\leq\frac{53}{342}=\frac 16- \frac{2}{171}$, followed by) $\epsilon \leq \frac{13}{84}=\frac 16-\frac{1}{84}$ using a decoupling inequality for curves.

In the context of dispersive and Strichartz estimates, generalized bounds imply stronger cancellation in oscillatory sums arising from the spectral decomposition, potentially leading to optimal space-time bounds. Our conjectured improvements thus rely on the validity of generalized Lindelöf-type exponential bounds for the relevant exponential sums associated with the quantum bouncing ball. However, the functions $f^{\varepsilon}_{\tau}$ defined in \eqref{deff2pe3} for $\varepsilon=\pm 1$ do {\it not} satisfy the key assumption \eqref{fctalpha} (from the exponent pair conjecture) due to the additional Airy phase terms exhibiting different behaviour over certain small ranges of $x$ (see Lemma \ref{lemdelta} and $x=1+l/\lambda$). This prevents direct application of the exponential sums bounds from \cite{Bourg17} (which yield an exponent $\epsilon=\frac 16-\frac{1}{84}$ and would imply a loss of $\frac 16+\frac{13}{14}\times \frac{1}{24}$ in \eqref{dispT>} instead of $\frac 16+\frac{5}{114}=\frac 16+\frac{20}{19}\times \frac{1}{24}$) to get better bounds for $E_{\lambda}$. Instead, in section \ref{secT>}, we apply Van der Corput derivative tests (VdCj) or (VdC $(j\pm1)$) up to order $j\leq 4$, carefully avoiding "bad" sets corresponding to values $k=\lambda+l$ with small $l$, when the derivatives of $f^{\varepsilon}_{\tau}$ and $f_{\tau}$ mismatch.

\end{document}